\documentclass[preprint,12]{imsart}
\usepackage[T1]{fontenc}
\usepackage{algorithm}
\usepackage[latin9]{inputenc}
\usepackage{amsmath}
\usepackage{amsfonts}
\usepackage{amsthm}
\usepackage{amssymb}

\usepackage{algorithm}
\usepackage{algpseudocode}
\usepackage{graphicx}
\allowdisplaybreaks 

\RequirePackage{natbib}
\RequirePackage[colorlinks,citecolor=blue,urlcolor=blue]{hyperref}

\usepackage{booktabs}
\newcommand{\ra}[1]{\renewcommand{\arraystretch}{#1}}

\makeatletter
\numberwithin{equation}{section}
\numberwithin{figure}{section}
\theoremstyle{plain}
\newtheorem{thm}{\protect\theoremname}
  \theoremstyle{plain}
  \newtheorem{lem}[thm]{\protect\lemmaname}

  \newtheorem{rem}{Remark}
    \newtheorem{prop}{Proposition}
        \newtheorem{assumption}{Assumption}
      
\usepackage[english]{babel}
  \providecommand{\lemmaname}{Lemma}
  \providecommand{\examplename}{Example}
\providecommand{\theoremname}{Theorem}
\usepackage{color}

\begin{document}
        
\begin{frontmatter}
        \title{   \sc two-sample testing in non-sparse \\ high-dimensional linear models}
        \runtitle{High-dimensional two-sample testing without sparsity}

        \begin{aug}
                \author{\fnms{Yinchu} \snm{Zhu}\ead[label=e1]{yinchu.zhu@rady.ucsd.edu}},
                \and
                \author{\fnms{Jelena} \snm{Bradic}
                        \ead[label=e2]{jbradic@math.ucsd.edu}}

                \runauthor{Zhu and Bradic}
                
                \affiliation{University of California, San Diego}
                
                \address{Rady School of Management,\\ University of California, San Diego\\
                                        La Jolla, California 92093,\\
                                        USA\\
                        \printead{e1}}
                
                \address{Department of Mathematics, \\ University of California, San Diego\\
                        La Jolla, California 92093,\\
                        USA\\
                        \printead{e2}}
        \end{aug}
        
        \begin{abstract}
  \ \  In analyzing high-dimensional    models, sparsity  of the model parameter is  a common but often undesirable assumption. Though different methods have been proposed for   hypothesis testing under sparsity, no systematic theory exists for inference methods that are robust to failure of the sparsity assumption. 
                In this paper, we study the following two-sample testing problem: given two samples generated by two high-dimensional linear models, we aim to test whether the regression coefficients of the two linear models are identical.
                We
                propose a   framework named TIERS (short for TestIng Equality of Regression Slopes), which solves the two-sample testing problem without  making any
                assumptions on the sparsity of the regression parameters.
                TIERS builds a new model by convolving the two samples in such a way that the original hypothesis translates into a new  moment condition. 
                A self-normalization construction is then developed to form a moment test.
                We provide rigorous theory for the developed framework.
                Under very weak conditions of the feature covariance, we show that the   accuracy of the proposed test in controlling  Type I errors  is robust both to the lack of sparsity in the features and to the heavy tails in the error distribution, even when the sample size is much smaller than the feature dimension. Moreover, we discuss minimax optimality and efficiency properties of the proposed test. 
                Simulation analysis demonstrates excellent finite-sample performance of our test. In deriving the test, we also develop tools that are of independent interest. The test is built upon a novel estimator, called Auto-aDaptive Dantzig Selector (ADDS), which not only automatically chooses an appropriate scale (variance) of the error term but also incorporates prior information. 
                To effectively approximate the critical value of the test statistic, we develop a  novel  high-dimensional plug-in approach that complements the recent advances in Gaussian approximation theory.
        \end{abstract}

\end{frontmatter}
\section{Introduction}\label{sec: intro}

High-dimensional data are increasingly encountered in
many applications of statistics and most prominently in biological and financial research.  A common feature of the statistical models used to study high-dimensional data is that while
the dimension $p$  of the model parameter is high, the sample size $n$ is relatively small. This is the so-called High-Dimensional Low-Sample Setting (HDLSS) where $p/n \to \infty$ as $n \to \infty$. For such models, a common underlying   theoretical assumption  is that the effective dimensionality of the data is small, i.e., the number of non-zero components of the model parameter, denoted by $s$, is either fixed or grows slowly as   $s/n \to 0$ with $n\to \infty$.
Hypothesis testing in such HDLSS settings has recently gained a great deal of attention (see \cite{dezeure2015} for a review). The  problem is fundamentally difficult to solve due to the bias propagation induced by the regularized estimators (\cite{zhang2014confidence}), all of which are specifically designed to  resolve the curse-of-dimensionality, $p \gg n$, phenomenon. Bias propagation is apparent across models and across  regularizations (\cite{zhang2008,FanLv11}).  

However, since consistent estimation without sparsity (i.e., $s\gg n$) has not been successfully resolved, hypothesis testing in such models has not been addressed until now.  Yet, many  scientific areas  do not support the sparsity assumption and thus require development of new inferential methods that are robust to non-sparsity.  A natural question is whether or not we can design effective method for hypothesis testing that allows for such non-sparse high-dimensional models.
We
 answer this question in the context of two-sample simultaneous tests of equality of parameters of the 
 regression models.
Below we present a couple of examples that highlight the importance and applications of the problem we consider.

 (a)  (Dense Differential Regressions)  
In many situations,   two natural subgroups in the data occur;
for example, 
suppose that treatments $1$ and
$2$ are given to $n_1$ and $n_2$ patients, respectively, and that the $j$-th subject who
receives treatment $i$ ($i = 1 \mbox { or } 2$) has response $y_{ij}$ to treatment at  the level $x_{ij}$,
where the expected value of $y_{ij}$ is  $ x_{ij}^\top \beta_i$ for $x_{ij}\in \mathbb{R}^p$ and $\beta_i \in \mathbb{R}^p$. 
 Specifically, we intend to test the null
hypotheses that  
$$H_0: \beta_1 = \beta_2.$$
 The effects of the treatment are typically collected in the vectors $x_{ij}$.
Such effects are rarely sparse as 
therapies for highly diverse illnesses  largely   benefit from complicated treatments, for example,  radio-chemotherapy (\cite{Allan25092001}).  In this case,
treatment affects a large number of  levels of the molecular composition
of the cell  and many malfunctions have been
shown to relate to cancer-like behavior.
Hence, in such setting, it is pertinent to consider vector $x_{ij}$ that measures   the  treatment  on the genetic cell or molecular level, leading to a  high-dimensional  feature vector with $p \gg n$ that is in principle   not sparse. 

(b)  (Dense Differential Networks)
Hypotheses concerning differences
in molecular influences or biological network structure using high-throughput data have become prevalent in modern FMRI studies. Namely, one is interested in discovering ``brain connectivity networks'' and comparing two of such networks between subgroups of population.
Here, one is interested in discovering the difference between two populations modeled by Gaussian graphical models. Significance testing for network differences is a challenging statistical problem, involving high-dimensional estimation and comparison of non-nested hypotheses. 
In particular, one can consider $p$-dimensional
  Gaussian random vectors  $Z_1 \sim \mathcal{N}(0, \Sigma_1)$
 and $Z_2 \sim \mathcal{N}(0, \Sigma_2)$ .
  The conditional independence graphs are characterized
by the non-zero entries of the precision matrices $\Sigma_1^{-1}$  and $\Sigma_2^{-1}$ (\cite{meinshausen2006high}). Given an
$n_1$-sample of $Z_1$
 and an $n_2$-sample of $Z_2$ with $n_1,n_2 \ll p$, the objective is to test
 \[
 H_0 : \Sigma_1^{-1} = \Sigma_2^{-1}
 \]
 without assuming any sparse structure on the precision matrices.

\subsection{This paper}

In this paper we consider the following two-sample, linear regression models
\begin{equation}
y_{A}=x_{A}^{\top}\beta_{A}+u_{A}\label{eq: original model A}
\end{equation}
and 
\begin{equation}
y_{B}=x_{B}^{\top}\beta_{B}+u_{B}\label{eq: original model B}
\end{equation}
with the unknown parameters of interest, $\beta_A \in \mathbb{R}^p$ and $\beta_B \in \mathbb{R}^p$.
For simplicity of presentation, we consider 
 Gaussian
random design, i.e.  $p$-dimensional design vectors $x_{A}$ and \(x_{B\ }\) follow normal distribution $\mathcal{N} (0, \Sigma_A)$
and  $\mathcal{N} (0, \Sigma_B) $
with unknown covariance matrices $\Sigma _A$ and $\Sigma _B$. Moreover, the noise components $u_A,u_B$ are mean zero, independent from the design and  have a  distribution with unknown standard deviations $\sigma_{u,A}$ and $\sigma_{u,B}$.
  In this formal setting, our objective is to test whether  model \eqref{eq: original model A} is the same as the model \eqref{eq: original model B}, i.e., to  develop a test for the hypothesis  of our  interest
\begin{equation}
H_{0}:\ \beta_{A}=\beta_{B}
\label{eq: null hypo}
\end{equation}
with $\beta_{A},\beta_{B}\in\mathbb{R}^{p}$ and $p \gg n$.
We will work with  two independent samples $\{(x_{A,i},y_{A,i})\}_{i=1}^{n}$
and $\{(x_{B,i},y_{B,i})\}_{i=1}^{n}$ of size $n$ but similar constructions can be exploited for the two samples of unequal sizes. 
 
 Here, we propose a hypothesis test for  \eqref{eq: null hypo} that provides a valid error control without making any assumptions about the sparsity of the two model parameters.  Although testing in high-dimensions has gained a lot of attention recently,   the procedure proposed is, to the best of our knowledge, the first  to possess all of the following properties.
 \begin{itemize}
  \item[(a)] 
The proposed test, named TestIng Equality of Regression Slopes ({TIERS} from here on), 
is  valid regardless of the assumption of sparsity of the model parameters $\beta_A$ and $\beta_B$. Type I error converges to the nominal level $\alpha$ even if $p\gg n$ and $s/p\rightarrow c\in[0,1]$. 
 
 \item[(b)] TIERS remains   valid   under any  distribution of the model errors
 $u_A, u_B$ even in the high-dimensional case.
In particular, TIERS is robust to heavy-tailed distribution  in the errors $u_A$ and $u_B$, such as the Cauchy distribution.  
 \item[(c)] Under weak regularity conditions, TIERS is nearly efficient compared to an oracle testing procedure and enjoys minimax optimality.
   Whenever, $\Sigma_A=\Sigma_B$,  TIERS achieves the aforementioned efficiency and optimality properties regardless of the sparsity in  $\beta_A$ and $\beta_B$.

 \end{itemize}
 
 The approach also extends naturally to  groups of  regressions and can provide Type I and Type II errors using only convex optimization or linear programming.  Additionally, the test can be generalized to the case where one considers nested models. Likewise, tests of the null hypothesis $H_{0,j,k}$:  $\beta_{A,j} = \beta_{B,k}$ can be performed.

 \subsection{Previous work}

Hypothesis testing in high-dimensional regression is  extremely challenging. Most estimators cannot guard against inclusion of noise variables unless  restrictive and unverifiable assumptions (for example,  irrepresentable condition (\cite{Zhao:2006}) or minimum signal strength (\cite{fan2004})  are made.  Early work on  p-values  includes methods  based on multiplicity correction that are conservative  in their Type I error control (\cite{MMB09,buhlmann2013statistical}), bootstrap methods that  control false discovery rate  (\cite{MMB09,mandozzi2015hierarchical}) and  inference methods guaranteeing asymptotically exact tests under  the  irrepresentable condition (\cite{fan2001variable}).
More recently, 
there have been  series of important studies  that  design  asymptotically valid tests while relaxing the irrepresentable condition. Pioneering work of 
\cite{zhang2014confidence} develops de-biasing technique and shows  that low-dimensional projections are an efficient way of constructing  confidence intervals.
In a major generalization, \cite{van2014asymptotically}  consider  
a range of models which
includes the linear models and the generalized linear models and obtain  valid inference methods when
$s \log (p) = o(n^{1/2})$. The literature has also seen work similar in spirit   for Gaussian graphical models (\cite{ren2015asymptotic}). \cite{2015arXiv150802757J} compute optimal sample size and minimax optimality of a modified de-biasing method  of \cite{javanmard2014confidence}, which  allows for non-sparse precision matrix in the design. \cite{Guang16}  and \cite{Dezeure16} evaluate approximating the overall level
of significance for simultaneous testing  of $p$ model parameters.  They demonstrate that
the multiplier bootstrap of  \cite{chernozhukov2013gaussian} can accurately approximate the overall level of significance whenever the true model is sparse enough.   

 However,   the above mentioned work  requires   model sparsity  i.e. $s /n \to 0$ as $n \to \infty$ and is hence too restrictive to be used in  many scientific data applications (social or biological networks for example).
 Moreover, despite the above progress in one-sample testing, two-sample hypothesis testing has not been addressed much in the existing literature. 
 In the context
of high-dimensional two-sample comparison of means, \cite{bai1996effect,chen2010,NIPS2011_4260,Cai2014}  have
introduced global tests to compare the means of two high-dimensional Gaussian
vectors with unknown variance with and without direct sparsity in the model. Recently  \cite{Cai2013a} and \cite{li2012} develop two-sample tests for
covariance matrices of two high-dimensional vectors specifically extending the Hotelling's $T^2$ statistics to high-dimensional setting, while \cite{2016arXiv160600252Z} develop new spectral test that allows for block sparsity patterns. \cite{charbonnier2015} address the more general heterogeneity test but heavily relies on the direct sparsity in the model parameters.
The latest effort in this direction is the work of \cite{stadler2016two} where the authors extend screen-and-clean procedure of \cite{wasserman2009} to the two-sample setting. The authors provide  asymptotic  Type I error guarantees that are only valid under sparse models. Notice that the last two methods are based on model selection and hence do not apply to dense models because, if $s=p$, perfect model selection simply means including all the features ($p$-dimensional with $p\gg n$).

\subsection{Challenges of two-sample testing in non-sparse models}

  The  central problem  of two-sample testing with $p \geq n$ is in finding an adequate  measure of comparison of the estimators between the two samples. In light of  the great success in designing tests for one sample, it becomes natural to expect that the developed methods trivially apply to the  two sample case. However, the situation is far from trivial when the sparsity assumption fails. Below we illustrate this problem via a specific example of dense and high-dimensional linear model. 

In particular, we show that a naive extension of the powerful de-biasing procedure to the two-sample problem, fails whenever the model is dense enough.  To that end, we consider a simple Gaussian design with $x_{A},x_{B}\sim \mathcal{N}(0,I_{p})$ and assume that the errors are independent with the standard Gaussian distribution, $u_{A},u_{B}\sim \mathcal{N}(0,1)$. The true parameters have $p$ non-zero elements (dense) and take the following  form 
\begin{equation}
\beta_{A}=\beta_{B}=\mathbf{1}_{p}c/\sqrt{n},
\end{equation}
 where $\mathbf{1}_{p}$
        denotes a $p$-dimensional vector of ones and $c\in\mathbb{R}$ is
        a constant. 

For simplicity, we consider the ``oracle'' de-biasing (de-sparsifying) estimator, which is the estimator defined in Equation (5) of  \cite{van2014asymptotically} except that
the node-wise Lasso estimators for the precision matrices ($\Sigma_A^{-1}$ and $\Sigma_B^{-1}$) are replaced by their true values, $I_p$. We compute this ``oracle'' de-biasing estimators for both samples 
        $$\tilde{\beta}_{k}=\hat{\beta}_{k}+n^{-1}\sum^{n } _{i=1}x_{k,i}(y_{k,i}-x_{k,i}^\top\hat{\beta}_{k})$$
        for $k\in\{A,B\}$,  where  the initial estimator  $\hat{\beta}_{k}$  is the scaled Lasso estimator
          (\cite{sun2012scaled})  defined as 
        \begin{equation}
        (\hat{\beta}_{k},\hat{\sigma}_{k})=\underset{\beta,\sigma}{\arg\min} \left\{ \frac{\sum_{i=1}^{n}(y_{k,i}-x_{k,i}^\top\hat{\beta}_{k})^{2}}{2n\sigma}+\frac{\sigma}{2}+\lambda\|\beta\|_{1} \right\} .\label{eq: scaled lasso}
        \end{equation}
        Here,  $\lambda$ is set to be $\lambda=\lambda_{0}\sqrt{n^{-1}\log p}$, where  $\lambda_{0}>\sqrt{2}$  is a constant. 
As discussed after Theorem 2.2 of \cite{van2014asymptotically}, de-biasing principles suggest a ``naive'' generalization of the one-sample T-test  with a test statistic of the form
        \[
        M_n = \max_{1\leq j\leq p}\sqrt{n}|\tilde{\beta}_{A,j}-\tilde{\beta}_{B,j}|/\sqrt{\hat{\Sigma}_{A,j,j}+\hat{\Sigma}_{B,j,j}}.
        \]
       A test of  nominal size $\alpha\in(0,1)$, would reject $H_0$ whenever $M_n > m_\alpha$ for the critical value $m_\alpha$ defined as 
        $
         m_\alpha = F^{-1}(1-\alpha,0,\hat{\Sigma}_{A},\hat{\Sigma}_{B})$,
         where  $\hat{\Sigma}_{k}=n^{-1}\sum_{i=1}^{n}x_{k,i}x_{k,i}^{\top}$ and $F^{-1}(\cdot,c,\Sigma_{A},\Sigma_{B})$ is the inverse with
        respect to the first argument of $F(x,c,\Sigma_{A},\Sigma_{B})$ with
        $$F(x,c,\Sigma_{A},\Sigma_{B})=\mathbb{P}\left(\max_{1\leq j\leq p}|\zeta_{j}|/\sqrt{\Sigma_{A,j,j}+\Sigma_{B,j,j}}\leq x\right)$$
        and a $p$-dimensional Gaussian vector  $\zeta$ with mean $(\Sigma_{A}-\Sigma_{B})\mathbf{1}_{p}c$ and covariance  $\Sigma_{A}+\Sigma_{B}$.

\begin{lem}
        \label{lem: naive example}Let the null hypothesis $H_{0}$ (\ref{eq: null hypo}) hold. Consider de-biased estimators $\tilde{\beta}_{k}$, $k\in \{A,B \}$ with 
        $\lambda=\lambda_{0}\sqrt{n^{-1}\log p}\ {\rm for}\ \lambda_{0}>\sqrt{2}$. Then, in the above  setting, as $n\to \infty$ 
        \[
        \mathbb{P}\left( M_n > m_\alpha \right)=M(\alpha,c)+o(1),
        \]
        where $M(\alpha,c)=\mathbb{E}\left[F\left(F^{-1}(1-\alpha,0,\hat{\Sigma}_{A},\hat{\Sigma}_{B}),c,\hat{\Sigma}_{A},\hat{\Sigma}_{B}\right)\right]$. \end{lem}

Lemma \ref{lem: naive example} establishes that  under $H_{0}$,   the ``naive''
test $M_n > m_\alpha$  has power approaching $M(\alpha,c)$. Notice that
for the naive test to be valid, we need $M(\alpha,c)\leq\alpha$ for all $ c\in\mathbb{R}$. In Figure \ref{fig: naive test size}, we plot the function $M(0.05,\cdot) $ for several combinations of $(n,p)$.   As we can see, $M(0.05,c)$
can be substantially larger than 0.05 and even reach one for $c=0.002$. In other words, the probability of rejecting
a true hypothesis can approach one when the model is not sparse.  Hence, naively applying existing methods does not solve the two-sample testing problem in dense high-dimensional  models.

\begin{figure}[h]
\begin{centering}
\caption{\label{fig: naive test size}\small$M(0.05,c)$: the rejection probability
of the naive test of nominal size 5\% under $H_{0}$ (\ref{eq: null hypo}).The blue, red and green curves  
correspond to $(n,p)=(100,500)$, $(200,500)$ and $(100,300)$, respectively. }

\includegraphics[scale=0.30]{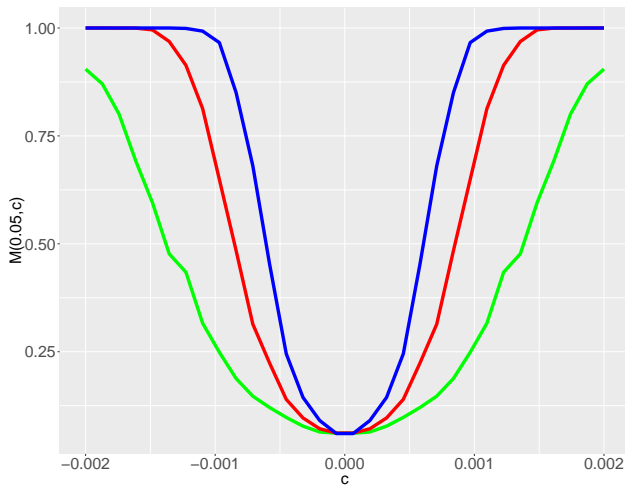}
\par\end{centering}
\end{figure}

\subsection{Notations and Organization of the paper}
Throughout this paper, $^{\top}$ denotes the matrix transpose and
$I_{p}$ denotes the $p\times p$ identity matrix. The (multivariate)
Gaussian distribution with mean (vector) $\mu$ and variance (matrix)
$\Sigma$ is denoted by $\mathcal{N}(\mu,\Sigma)$. The cumulative distribution of the standard normal distribution is denoted by $\Phi(\cdot)$. For a vector $v\in\mathbb{R}^{k}$, we define its $\ell_{q}$-norm
as follows: $\|v\|_{q}=(\sum_{i=1}^{k}|v_{i}|^{q})^{1/q}$ for $q\in(0,\infty$),
$\|v\|_{\infty}=\max_{1\leq i\leq k}|v_{i}|$ and $\|v\|_{0}=\sum_{i=1}^{k}\mathbf{1}\{v_{i}=0\}$,
where $\mathbf{1}\{\}$ denotes the indicator function. For matrix $A$, its $(i,j)$ entry is denoted by $A_{i,j}$, its $i$-th row by $a_i$ and its $j$-th column by $A_j$. We define $\|A\|_\infty=\max|A_{i,j}|$, where the maximum is taken over all $(i,j)$ indices. We use $\sigma_{\max}(\cdot)$ and  $\sigma_{\min}(\cdot)$ to denote the maximal and minimal singular values, respectively.  For two sequences
$a_{n},b_{n}>0$, we use $a_{n}\asymp b_{n}$ to denote that there
exist positive constants $C_{1},C_{2}>0$ such that $\forall n$,
$a_{n}\leq C_{1}b_{n}$ and $b_{n}\leq C_{2}a_{n}$. For two real numbers $a$ and $b$, let $a\vee b$ and $a\wedge b$ denote $\max\{a,b\}$ and $\min\{a,b\}$, respectively. 
We use ``s.t.'' as the abbreviation for ``subject to''.
For two random quantities $X$ and $Y$ (scalars, vectors and matrices), $X\perp Y$ denotes the  independence of  $X$ and $Y$.

The rest of the paper is structured as follows. We develop the new methodology and the testing procedure in Section \ref{sec: methodology}, where a novel ADDS  estimator   is also introduced. In Section \ref{sec: HD approx}, we develop a general theory of approximating the distribution of a large of class of test statistics with functions of Gaussian distributions. In Section \ref{sec: theory}, we derive theoretical properties of our test, such as size and power guarantees and efficiencies, and extend the results to non-Gaussian designs. In Section \ref{sec: MC simulations},  Monte Carlo simulations are employed to assess the finite-sample performance of the proposed test. The proofs for all of the theoretical results are contained in the appendix.

\section{Methodology}\label{sec: methodology}
  In this section we propose a new two-step methodology  for testing in non-sparse and  high-dimensional two-sample models. In the first step, we match
the variables from the two samples to obtain a convolved  sample, which satisfies a new model, referred to as the convolution regression model. This novel construction allows us to reformulate the hypothesis of interest \eqref{eq: null hypo} into testable moment conditions on the convolved  sample. In the second step, we construct a suitable moment test  by utilizing  plug-in principles and    self-normalization.
Therefore,  convolving the two-samples allows us to work directly with a specific moment constraint. This  substantially simplifies the problem and 
facilitates the theoretical analysis in non-sparse high-dimensional models.
In contrast, the traditional Wald or Score methods for  testing are no longer suitable
in the current context.

\subsection{Constructing testable moment conditions via convolution}

This section introduces the idea of constructing a convolution regression equation, which naturally generates  moment conditions that are equivalent to the original null hypothesis (\ref{eq: null hypo}).
 Throughout this section, our interest is on analyzing Gaussian designs, which
are common in  applications. Extensions to non-Gaussian designs are presented in Section \ref{sec: theory}.

We propose to   reformulate the models (\ref{eq: original model A}) and (\ref{eq: original model B}) into a new model,  in which $\beta_A-\beta_B$ appears as a regression coefficient, and  then use it to     derive a moment condition corresponding to the testing problem $H_0: \beta_A= \beta_B$. 
Instead of naively concatenating two samples, we consider convolving the variables from the two respective samples as follows. Define the convolved  response and  error 
  $y=y_{A}+y_{B}$ and   $u=u_{A}+u_{B}$ and the new design matrices  $w=x_{A}+x_{B }$ and as well as  $z=x_{A}-x_{B}$. With this notation at hand,  (\ref{eq: original model A}) and  (\ref{eq: original model B}) imply the following convolved  regression model:  
\begin{equation}\label{eq:model2}
y=w^{\top}\theta_{*}+z^{\top }\gamma_{*}+u,
\end{equation}
with unknown parameters 
$$\theta_{*}=(\beta_{A}+\beta_{B})/2 \qquad \mbox{and} \qquad \gamma_{*}=(\beta_{A}-\beta_{B})/2.$$

In the above convolved model, parameter $\gamma^*$ is of main interest as the null hypothesis $H_{0}$ in (\ref{eq: null hypo}) is equivalent
to 
$$H_0: \gamma_{*}=0.$$
 Due to the high-dimensionality
and potential lack of sparsity in $\theta_{*}$ and $\gamma_*$, we cannot simply
estimate $\theta_{*}$ and $\gamma_{*}$ and test $\gamma_{*}=0$.
In order to  construct a  number of testable moment restrictions, we use 
the 
conditioning information   between $w$ and $z$. Namely, we  introduce  the following
parameter 
$$\Pi_{*,j}=\left(\mathbb{E}\left[ ww^{\top} \right]\right)^{-1}\mathbb{E} \left[wz_{j} \right]\in\mathbb{R}^{p},$$
describing correlation between the column $z_j \in \mathbb{R}$ and the design matrix $w \in \mathbb{R}^{n \times p}$.
 In the case of Gaussian designs  parameter $\Pi_{*,j}$ also encodes their dependence structure. 
If we know $\Pi_{*,j}$ then we can decouple the   parameter of interest $\gamma_*$ from  the nuisance parameter $\theta_*$ in \eqref{eq:model2}.
In order to see that, observe that 
\begin{equation}
v_{j} =z_{j}-w^{\top}\Pi_{*,j},\label{eq: graphical model}
\end{equation}
defined for  $1 \leq j \leq p$ satisfies that $\mathbb{E} \left[ vw^{\top} \right]=0$, where $v=(v_{1},\cdots,v_{p})^{\top}\in \mathbb{R}^p$.
Additionally for Gaussian  design  we have $w\perp v$.
   However,  observe that  $v_1,\dots,v_p$ can be highly dependent.

The following lemma characterizes the unknown parameter  $\Pi_{*,j}$ and the covariance structure of the  vector $v$  of \eqref{eq: graphical model}, i.e., $\mathbb{E}\left[vv^{\top}\right]$. The proof is merely straight-forward computation and is thus omitted.
\begin{lem}\label{lem:1}
Under (\ref{eq: original model A}), (\ref{eq: original model B})
and (\ref{eq: graphical model}), we have $\Pi_{*}=(\Pi_{*,1},\cdots,\Pi_{*,p})$ with 
$$\Pi_*=(\Sigma_{A}+\Sigma_{B})^{-1}(\Sigma_{A}-\Sigma_{B})\in\mathbb{R}^{p\times p}$$
and $\mathbb{E} \left[ vv^{\top} \right]=4(\Sigma_{A}^{-1}+\Sigma_{B}^{-1})^{-1}$. 
\end{lem}
Now we  illustrate how vector $v$ can be utilized to decouple $\gamma_*$ from $\theta_*$. 
First, observe that  
$$\mathbb{E} \left[ v(y-w^{\top}\theta_{*}) \right]=\mathbb{E}\left[ vz^{\top}\right]\gamma_{*}=\mathbb{E}\left[vv^{\top}\right]\gamma_{*}$$
where  equalities hold by (\ref{eq:model2}) and (\ref{eq: graphical model}), respectively. Second, we observe that by Lemma \ref{lem:1}, the matrix $E[vv^{\top}]$ is nonsingular as long as $\Sigma_A$ and $\Sigma_B$ are nonsingular.  Since $\gamma_*=(\beta_A-\beta_B)/2$, it follows that the original hypothesis   $H_{0} : \beta_A=\beta_B$
 holds if and only if $\mathbb{E}\left[ v(y-w^{\top}\theta_{*}) \right]=0$.
Therefore,  the two problems are equivalent and we can proceed to test the following  moment condition 
\begin{equation}
H_{0}:\ \mathbb{E} \left[ (z^{\top}-w^{\top}\Pi_{*})(y-w^{\top}\theta_{*}) \right]=0.\label{eq: moment condition}
\end{equation}

We will show that this restriction allows for the construction of a highly successful test while allowing $\theta_*$ to be fully dense vector with $p \geq n$.

\subsection{Testing the moment condition (\ref{eq: moment condition})}

Now we discuss the effective construction of the test. The moment condition
(\ref{eq: moment condition}) involves the unknown quantities $\Pi_{*}$
and $\theta_{*}$, which we replace with estimated counterparts. 
However, observe that the convolved  model does not have the same distributional or sparsity properties as the original linear models.
We
do not assume consistent estimation for $\theta_{*}$ and the validity
of our test is built upon consistent estimation of $\Pi_{*}$ only. 
We start by introducing some notations.

 Given the two samples  $\{(x_{A,i},y_{A,i})\}_{i=1}^{n}$
and $\{(x_{B,i},y_{B,i})\}_{i=1}^{n}$, we define a new response vector $Y=(y_{A,1}+y_{B,1},\cdots,y_{A,n}+y_{B,n})^\top\in\mathbb{R}^n$, and use a matrix 
$X_{A}=(x_{A,1},\cdots,x_{A,n})^{\top}\in\mathbb{R}^{n\times p}$, and $X_{B}=(x_{B,1},\cdots,x_{B,n})^{\top}\in\mathbb{R}^{n\times p}$.
Furthermore, we define two new design matrices, $W=X_{A}+X_{B}\in \mathbb{R}^{n\times p}$ and $Z=X_{A}-X_{B} \in \mathbb{R}^{n\times p}$. We denote by $u_i^\top$ the $i$-th row of $U=(Y_A-X_A\beta_A)+(Y_B-X_B\beta_B)$.
Finally, the \(j\)-th column of $X_A$, $X_B$ and $Z$ will be denoted by $X_{A,j}$, $X_{B,j}$ and $Z_j$, respectively.
\subsubsection{Auto-aDaptive Dantzig Selector (ADDS)} \label{sec:asds-intro}
We propose a novel estimator for $\theta_{*}$, which can be used
to construct the test statistic.  
Let  
\begin{equation}\label{eq: def DS path}
\begin{array}{cccc}
\hat{\theta}(\sigma)&=\underset{\theta\in\mathbb{R}^{p}}{\arg\min}\ & \|\theta\|_{1}     \\
& \ \ \  \ \ \  \ \mbox{ s.t. } & \Bigl\|n^{-1}(X_A + X_B)^{\top}(Y_A + Y_B-X_A\theta - X_B \theta) \Bigl \|_{\infty}&\leq\eta\sigma   
\end{array},
\end{equation}
  
where $\eta\asymp\sqrt{n^{-1}\log p}$ is a tuning parameter that
does not depend on $\sigma_{u}$.
 Moreover, let 
\begin{equation}\label{eq: auto-scaling DS}
\begin{array}{cccc} 
\tilde{\sigma}_{u}&=\underset{\sigma\geq0}{\arg\max}\  & \sigma\\
& \ \ \  \ \ \  \ \mbox{ s.t. } &  {\Bigl\|Y_A + Y_B- X_A \hat{\theta}(\sigma) - X_B \hat \theta (\sigma) \Bigl\|_{2}^{2} }&\geq n{\sigma^{2}}/{2}
\end{array}.
\end{equation}

Our estimator for $\theta_{*}$, referred to as the Auto-aDaptive Dantzig
Selector (ADDS), is defined as 
\begin{equation}
\widetilde{{\theta}}=\hat{\theta}(\tilde{\sigma}_{u})\label{eq: theta estimate}.
\end{equation}

An advantage of the proposed estimator is that it simultaneously estimates
$\sigma_{u}^{2}=\mathbb{E}[u^{2}]$ and $\theta_{*}$ and is thus ``scale-free''. 
The intuition behind the proposed ADDS
 is that estimation of the signal and estimation of its variance are closely related and  can benefit
from each other: a more accurate estimation of the variance can lead to a better signal
estimation and a more accurate signal estimation can help estimate the variance better. Notice that the solution is well defined if ${\rm rank}(W)=n$; when ${\rm rank}(W)=n$, $\hat{\theta}(0)$ exists and  $0$ satisfies the constraint in the optimization problem (\ref{eq: auto-scaling DS}).
Moreover,  the solution path $\sigma\mapsto\hat{\theta}(\sigma)$ can be computed very
efficiently using algorithms such as DASSO   (\cite{james2009dasso})
and the parametric simplex method (\cite{pang2014fastclime,vanderbei2015linear}). Once the solution path is obtained,
computing $\tilde{\sigma}_{u}$ is only a one-dimensional optimization
problem and thus can be solved efficiently.


The estimator for $\Pi_{*}$ also takes the  form of ADDS estimator and is defined as  follows. For every $j \in \{ 1, \dots, p\}$, $\hat{\Pi}_{j}(\tilde{\sigma}_{j}) \in \mathbb{R}^p$ is defined by
\begin{equation}\label{eq:???}
\begin{array}{cccc}
\hat{\Pi}_{j}(\sigma)&=\underset{\pi\in\mathbb{R}^{p}}{\arg\min}\  & \|\pi\|_{1} \\
& \ \ \  \ \ \  \ \mbox{ s.t. } &  
\Bigl \|n^{-1}(X_A + X_B)^{\top}(X_{A,j} - X_{B,j}-X_A\pi - X_B \pi) \Bigl\|_{\infty}&\leq\eta\sigma
\end{array},
\end{equation}
and
\begin{equation}\label{eq:????}
\begin{array}{cccc} 
\tilde{\sigma}_{j}&=\underset{\sigma\geq0}{\arg\max}\  & \sigma \\
& \ \ \  \ \ \  \ \mbox{ s.t. } &  {\Bigl\|X_{A,j} - X_{B,j}-X_A\hat{\Pi}_{j}(\sigma) - X_B \hat{\Pi}_{j}(\sigma) \Bigl\|_{2}^{2} }&\geq n{\sigma^{2}}/{2}
\end{array}.
\end{equation}
where  $\eta\asymp\sqrt{n^{-1}\log p}$ is a tuning parameter. 

Then, the  ADDS estimator of $\Pi_*$ is defined as 
\begin{equation}
\widetilde \Pi = (\widetilde{\Pi}_1,\cdots,\widetilde{\Pi }_p)\in \mathbb{R}^{p\times p}\quad{\rm with}\quad \widetilde{\Pi}_j=\hat{\Pi}_{j}(\tilde{\sigma}_{j})\in \mathbb{R}^p.\label{eq: Pi estimate}
\end{equation}

Observe that ADDS estimator $\widetilde \Pi$ allows for adaptive and more accurate tuning of the heteroscedastic components $v_1,\dots, v_p$. Another advantage of ADDS is that the structure of the ADDS allow us to derive certain properties of our test without  any restrictions on the distribution of the vector $v$. 
 We relegate further discussions of ADDS for the Section \ref{sec:asds}.

\subsubsection{A new test statistics}
In this section we propose a new test statistics and a new simulation method to obtain its critical value.

We develop a test statistic for testing \eqref{eq: null hypo}  as a scale-adjusted estimator of the moment condition  \eqref{eq: moment condition} that requires estimates of $\Pi_*$ and $\theta_*$.  For $\widetilde \Pi$ and $\widetilde{\theta}$ defined in (\ref{eq: Pi estimate}) and (\ref{eq: theta estimate}), we define our test statistic
\begin{equation}
T_{n}=n^{-1/2}\hat{\sigma}_{u}^{-1} \Bigl \|(Z-W\widetilde{\Pi})^{\top}(Y-W\widetilde{\theta}) \Bigl \|_{\infty}, \label{eq: test stat Gauss}
\end{equation}
with $\hat{\sigma}_{u}$ defined as 
\[
\hat{\sigma}_{u}=\bigl \|Y-W\widetilde{\theta} \bigl \|_{2}/\sqrt{n}.
\]
The value of $T_n$ tends to be moderate when $H_0$ is true, and large when $H_0$ is false. Therefore, our  test  is to reject  $H_0$ in favor  of $
H_1:\  \beta_A \neq \beta_B$
if $T_n$ is 
``too large.'' 
Observe that the distribution of the test statistic $T_n$ is difficult to obtain due to the complicated dependencies between different entries of $(Z-W\widetilde{\Pi})^{\top}(Y-W\widetilde{\theta})$. However, we propose a new plug-in Gaussian approximation method that automatically takes into account such inter-dependence and the high-dimensionality. The critical value is defined as the pre-specified quantile of  the $\ell_\infty$ norm of a zero-mean Gaussian vector with known covariance matrix and is thus easy to compute by simulation.

Under $H_{0}$, the behavior of the test statistic can be analyzed
through the following decomposition
\[
n^{-1/2}(Z-W\widetilde{\Pi})^{\top}(Y-W\widetilde{\theta})=n^{-1/2}V^{\top}\hat U+\Delta,
\]
where $\hat U =Y-W\widetilde{\theta}$,
$$\Delta =n^{-1/2}(\Pi_{*}-\widetilde{\Pi})^{\top}W^{\top} \hat U$$
 and  
$V=(V_{1},\cdots,V_{p})$ with  $$V_{j}=Z_{j}-W\Pi_{*,j}.$$ As
we will show in the proof, 
$$\| \Delta \|_{\infty}\hat{\sigma}_{u}^{-1}=o_{P}(1)$$
and the behavior of $T_{n}$ is driven by $\|n^{-1/2}V^{\top}\hat U\hat{\sigma}_u^{-1}\|_{\infty}$.

Our approximation procedure is rooted in the implicit independence in the linear approximation term $\sum_{i=1}^{n}v_{i}\hat{u}_{i}$, where $\hat{u}_i$ is the $i$-th entry of $\hat{U}$ and $v_i^\top$ is the $i$-th row of $V$. 
 Suppose that \(H_{0}\) \eqref{eq: null hypo} holds. Notice that $V$ is independent of $(Y,W)$ by construction and  that we purposely constructed ADDS estimator  $\widetilde{\theta}$ as a function of 
 $(Y,W)$ only. Therefore, $V$ is independent of $\hat U$.
Because of this independence, 
\begin{equation}
\left\Vert n^{-1/2}\sum_{i=1}^{n}v_{i}\hat{u}_{i}\hat{\sigma}_u^{-1}\right\Vert _{\infty} \ {\rm conditional\ on}\ \{\hat{u}_{i}\}_{i=1}^{n}\ {\rm has\ the\ same\ distribution\ as\  }\  \|\xi\|_\infty, \label{eq: decomp}
\end{equation}
where $$\xi\sim \mathcal{N}(0,Q)\qquad \mbox{and} \qquad   Q=n^{-1}\sum_{i=1}^{n}\hat{\sigma}_u^{-2}\hat{u}_i^2\mathbb{E}[v_{i}v_{i}^\top]=\mathbb{E}[v_1v_1^\top ].$$
 Observe that covariance matrix $Q$ is an unknown nuisance parameter in approximating  the distribution of $T_n$ with  $\|\mathcal{N}(0,Q)\|_\infty$. Per Lemma \ref{lem:1}, $Q$ is a positive-definite matrix of growing dimensions. 
  We propose to construct $\hat{Q}$, an
estimator for $Q$, and then use a plug-in approach for the critical value by simulating the distribution of   $\|\mathcal{N}(0,\hat{Q})\|_\infty$. 
We consider a natural  estimator
\begin{equation}
\hat{Q}=n^{-1}\sum_{i=1}^{n}\hat{v}_{i}\hat{v}_{i}^{\top} ,\label{eq: Q hat def}
\end{equation}
where $\hat{v}_i^\top$ is the $i$-th row of $\hat{V}=Z-W \widetilde{\Pi}\in\mathbb{R}^{n\times p}$.  In Section \ref{sec: HD approx}, we develop a general approximation theory that does not depend on the specific form of $\hat Q$, as long as it is a sufficiently good estimator of $Q$.
For the ease of presentation, we introduce the function $\Gamma(x,A):=\mathbb{P}(\|\xi\|_{\infty}\leq x)$ with
$\xi\sim \mathcal{N}(0,A)$. 
Notice that for a given matrix $A$, $\Gamma(\cdot,A)$ can be easily computed via simulation. We summarize  our method in Algorithm \ref{alg: main}.

\begin{rem}
Notice that we assume the Gaussianity of $v_{i}$ and thus the distribution
of $\|n^{-1/2}\sum_{i=1}^{n}v_{i}\hat{u}_{i}\hat{\sigma}_u^{-1}\|_{\infty}$ is exactly
Gaussian. The above setup is
more general in that it also applies when the Gaussianity of $v_{i}$
fails. See Section \ref{sec: non-Gaussian} for details.
\end{rem}

\begin{algorithm} 
\small
\caption{Testing of Equality of Regression Slopes (TIERS)} \label{alg: main}
\begin{algorithmic}[1]
    \Require  Two samples $(X_A,Y_A)$ and $(X_B,Y_B)$ and  level $\alpha \in (0,1)$   of the test.
    \Ensure  Decision whether or not to reject the null hypothesis (\ref{eq: null hypo})
    \State Construct $W=X_{A}+X_{B}$, $Z=X_{A}-X_{B}$ and $Y=Y_{A}+Y_{B}$
    \State Compute $\widetilde{\theta}$ and $\widetilde{\Pi}$ as in (\ref{eq: theta estimate})
and (\ref{eq: Pi estimate}), respectively, with tuning parameter
$\eta\asymp\sqrt{n^{-1}\log p}$

        \State Compute the test statistic $T_n$ as in (\ref{eq: test stat Gauss}) and $\hat{Q}$ as in (\ref{eq: Q hat def}). \textcolor{red}{} 
        \State Compute approximately $\Gamma^{-1}(1-\alpha,\hat{Q})$  (by simulation), where $\Gamma^{-1}(\cdot,\hat{Q})$ is the inverse of $\Gamma(\cdot,\hat{Q})$.

   \Return 
  Reject $H_{0}$ (\ref{eq: null hypo}) if and only if $T_{n}>\Gamma^{-1}(1-\alpha,\hat{Q})$.
\end{algorithmic}
\end{algorithm}

The methodology  for our approximation  when the   null  hypothesis prevails closely parallels that for construction of 
critical values in classical statistics, in which the limiting distribution of the test statistics can be derived but contains unknown nuisance parameters. In low-dimensional problems it is common to resort to a plug-in principle where the nuisance parameter is replaced by its consistent estimate. In this paper, we deal with high-dimensional problems for which the extreme dimensionality $p\gg n$ renders the classical central limit theorems non applicable and poses challenges in deriving accurate approximations of the distributions of test statistics. With recent advances in high-dimensional central limit theorem (see \cite{chernozhukov2013gaussian} for example), we are able to generalize the classical plug-in method to the high-dimensional settings.

\section{A high-dimensional plug-in principle}\label{sec: HD approx}

Many test statistics related to ratios, correlation and regression coefficients in statistics may be expressed as a nonlinear function of the vector  of population quantities. Additionally, in high-dimensional setting test statistics often take the form of the maximum of a large number of random quantities   
$\max_{1\leq j \leq p} |G_{n,j}|$,  with $G_{n,j} \in \mathbb{R}$. Individual $G_{n,j}$'s can be studentized t-statistics, such as in \cite{Dezeure16}, or simply the difference between a parameter and its estimator, such as in \cite{Guang16}.
Studying asymptotic distribution of such non-linear quantities is extremely difficult. 
However,  linearization may prove to be  a useful technique. Linearization decomposes a test statistic of interest into a linear term and an approximation error: 
\[
G_{n,j} = n^{-1/2}\sum_{i=1}^{n}\Psi_{i,j} + \Delta_{n,j}.
\]
In display above we consider  $\Psi_{i}=(\Psi_{i,1},\cdots,\Psi_{i,p})^\top\in\mathbb{R}^{p}$ with $p\gg n$ and $i=1,\dots, n$.

In this section, we propose a  general method of computing the critical value of the test statistic $\max_{1\leq j \leq p} |G_{n,j}|$. This method 
 is simple to implement and applies to a wide range of problems for which the above decomposition holds.
The method is based on Gaussian  approximations, as they enable easy approximations by simulation. Apart from high-dimensionality, the challenge is the presence of  the approximation error $\Delta_{n}=(\Delta_{n,1},\cdots,\Delta_{n,p})^\top \in \mathbb{R}^p$
and the fact that the linear terms $\Psi_{i}$ are often not observed, i.e. depend on unknown parameters.

To present the main result, we define the covariance matrix $Q=n^{-1}\sum_{i=1}^{n}\mathbb{E}(\Psi_{i}\Psi_{i}^{\top}\mid\mathcal{F}_{n})$, where  $\mathcal{F}_{n}$ is a $\sigma$-algebra such that $\Psi_i$, conditional on $\mathcal{F}_n$, is independent (or weakly dependent, e.g., strong mixing) across $i$.  Moreover, we make assumptions on the structure  of matrix $Q$ and linearization terms.

\begin{assumption} \label{as:1}
Suppose that 
\\
(i) there exist constants $b_{1},b_{2}\in(0,\infty)$ such that $$\mathbb{P}\left(b_{1}\leq\min_{1\leq j\leq p}Q_{j,j}\leq\max_{1\le j\leq p}Q_{j,j}\leq b_{2}\right)\rightarrow1;$$
\\
(ii) $\sup_{x\in\mathbb{R}}\left|\mathbb{P}\left(\|n^{-1/2}\sum_{i=1}^{n}\Psi_{i}\|_{\infty}\leq x\mid\mathcal{F}_{n}\right)-\Gamma(x,Q)\right|=o_{P}(1)$;\\
(iii) $\|\Delta_{n}\|_{\infty}\sqrt{\log p}=o_{P}(1)$ and there exists a matrix $\hat Q$  such that $\|\hat{Q}-Q\|_{\infty}\sqrt{\log p}=o_{P}(1)$. 
\end{assumption}

Assumption \ref{as:1}(ii) states
a Gaussian approximation for the partial sum $n^{-1/2}\sum_{i=1}^{n}\Psi_{i}$.
Sufficient conditions for this assumption are provided by  \cite{chernozhukov2013gaussian,chernozukov2014central}; see Proposition \ref{prop: HD CLT CCK} in the appendix.  Assumption \ref{as:1}(iii) is very mild and only assumes entry-wise consistency of the matrix estimator. For example, if $\{\Psi_i\}_{i=1}^n$ is sub-Gaussian (\cite{vershynin2010introduction}) and is observed, Bernstein's inequality and the union bound imply that the sample covariance matrix satisfies 
  this assumption if $\log p=o(\sqrt{n})$. 

\begin{thm}
\label{thm: approx assuming CLT} Suppose that Assumption \ref{as:1} holds. Then, as $n,p \to \infty$
\[
\sup_{x\in\mathbb{R}}\left|\mathbb{P}\left(\max_{1\leq j \leq p} |G_{n,j}|\leq x\mid\mathcal{F}_{n}\right)-\Gamma(x,\hat{Q})\right|=o_{P}(1).
\]
\end{thm}

Theorem   \ref{thm: approx assuming CLT} states that, under the regularity conditions stated in Assumption \ref{as:1}, the distribution of the test statistic $\max_{1\leq j \leq p} G_{n,j}$ can be approximated by $\Gamma(\cdot,\hat{Q})$, which can be easily  simulated.
Notice that the  multiplier bootstrap method by \cite{chernozhukov2013gaussian,chernozukov2014central}, which requires explicit observations of $\Psi_i$, does not apply in our context. 

\section{Theoretical results} \label{sec: theory}

In this section we present theoretical guarantees and optimality of the test proposed in Section \ref{sec: methodology}. We also consider extensions to non-Gaussian designs and present theoretical results in this setting as well. We start by deriving the theoretical properties of ADDS.

\subsection{ADDS properties} \label{sec:asds}

Auto-aDaptive Dantzig Selector introduced in Section \ref{sec:asds-intro} is broadly applicable to a class of linear models where an estimator of the high-dimensional parameter is needed together with  its scale. In this section we provide more details of the proposed estimator and its properties in a setup where apart from a sparsity constraint we allow for a general class of   constraints as well. Our result is comparable to the Dantzig selector; see \cite{candes2007dantzig} and  \cite{bickel2009simultaneous}.

With a slight abuse in notation, we consider a model  
$$H=Gb_{*}+\varepsilon,$$ 
where $H\in\mathbb{R}^{n}$ is a response vector, $G\in\mathbb{R}^{n\times p}$  is a design matrix, 
$\varepsilon\in\mathbb{R}^{n}$ is an error of the model and 
$b_{*}\in\mathbb{R}^{p}$ is the unknown parameter of interest. Let $\sigma \mapsto \mathcal{B}(\sigma)$ be a mapping such that (1) $\mathcal{B}(\sigma)\subseteq\mathbb{R}^p$  for any $ \sigma\geq0$ and (2) $\mathcal{B}(\sigma_{1})\subseteq\mathcal{B}(\sigma_{2})$
for $\sigma_{1}\leq\sigma_{2}$. We shall provide more discussion on the set $\mathcal{B}(\sigma)$ later. Then a generalized ADDS is defined as 
\begin{equation}\label{eq: DS optimization step 1}
\begin{array}{cccc}
\hat{b}(\sigma)&=\underset{b\in\mathbb{R}^{p}}{\arg\min}\ & \| b \|_{1}     \\
& \ \ \  \ \ \  \ \mbox{ s.t. } & \left \|n^{-1}G^{\top}(H-Gb)\right \|_{\infty}&\leq\eta\sigma   \\
& \ \ \  \ \ \  \   & b\in\mathcal{B}(\sigma)
\end{array},
\end{equation}
  
where $\eta\asymp\sqrt{n^{-1}\log p}$ is a tuning parameter that
does not depend on the magnitude of $\varepsilon$.
Then we compute 

\begin{equation}\label{eq: auto-scaling DS step 2}
\begin{array}{cccc} 
\tilde{\sigma}&=\underset{\sigma\geq0}{\arg\max}\  & \sigma\\
& \ \ \  \ \ \  \ \mbox{ s.t. } &  {\|H-G\hat{b}(\sigma)\|_{2}^{2} }&\geq n{\sigma^{2}}/{2}
\end{array}.
\end{equation}

Now  the ADDS estimator for $b_{*}$  is defined as 
\begin{equation}
\widetilde{b}=\hat{b}(\tilde{\sigma})\label{eq: theta estimate-1}.
\end{equation}

In the proposed method, the estimation starts with the  framework of Dantzig selector where the  tuning parameter is split into two components: a variance-tuning component $\sigma$ and variance-free component $\eta$.  Variance is therefore treated as another tuning parameter and the standardized ``regularization'' tuning parameter is fixed at an optimal  theoretical value proportional to $\sqrt{\log (p)/n}$.  In practice, we can obtain $\eta$
by simulating $\|n^{-1}G^{\top}\xi\|_{\infty}$, where $\xi \sim \mathcal{N}(0,I_{n})$. In the first step (\ref{eq: DS optimization step 1}), we compute the solution path, which maps the variance tuning parameter $\sigma$ to its  estimate $\hat{b}(\sigma)$.  Then in the second step (\ref{eq: auto-scaling DS step 2}), we compute an ``optimal'' choice for the variance tuning parameter $\tilde{\sigma}$. The ADDS is then defined as the point on the solution path corresponding to this optimal choice  $\tilde{\sigma}$.

The set $\mathcal{B}(\sigma)$ is introduced for a general setup,
where additional constraints other than the usual Dantzig restrictions
are imposed. These additional constraints are represented by the set $\mathcal{B}(\sigma)$. This set   could  incorporate prior knowledge of the parameter of interest, e.g., a bound for the signal-to-noise ratio represented by  $\mathcal{B}(\sigma) = \left\{ b \in \mathbb{R}^p: \|b\|_2 \leq c\sigma \right\}$ with some pre-specified $c>0$. In case of non-Gaussian designs (see Algorithm \ref{alg: main nonGaussian}),   an additional constraint is placed to ensure the high-dimensional central limit
theorem. Other strategies in literature, such as the square-root Lasso by \cite{belloni2011square}, the scaled Lasso by \cite{sun2012scaled} and self-tuned Dantzig selector by \cite{gautier2013pivotal}, do not have this flexibility. 
 
\begin{assumption}\label{assu:asds}
There exist constants $\eta,\kappa,\sigma_{*}>0$
such that (i) $\|n^{-1}G^{\top}\varepsilon\|_{\infty}\leq\eta\sigma_{*}$,
(ii) $3\sigma_{*}^{2}/4\leq n^{-1}\|\varepsilon\|_{2}^{2}\leq2\sigma_{*}^{2}$,
(iii) $b_{*}\in\mathcal{B}(\sigma_{*})$,  (iv) $28\eta\sqrt{\|b_{*}\|_{0}/\kappa}\leq1$
and the matrix $G$ is such that the Restricted Eigenvalue condition holds, i.e., (v)
\begin{equation}
\min_{J_{0}\subseteq\{1,\cdots,p\},|J_{0}|\leq\|b_{*}\|_{0}}\min_{a\neq0,\|a_{J_{0}^{c}}\|_{1}\leq\|a_{J_{0}}\|_{1}}\frac{\|Ga\|_{2}^{2}}{n\|a_{J_{0}}\|_{2}^{2}}\geq\kappa.\label{eq: RE condition}
\end{equation}
\end{assumption}
 
In the usual linear regression setup, one can typically show that  Assumption \ref{assu:asds} holds with high probability. This is in line with the usual argument in high-dimensional statistics, where the conclusion often states the properties of an estimator on an event that occurs with probability close to one. 
 Now we present the behavior of ADDS under Assumption \ref{assu:asds}.

\begin{thm}
\label{thm: ADDS}
Let Assumption \ref{assu:asds} hold. Consider $\tilde{\sigma}$ and  $\widetilde{b}$  defined in  \eqref{eq: auto-scaling DS step 2} and \eqref{eq: theta estimate-1}, respectively. 
Then,  for $\hat{\sigma}=n^{-1/2}\|H-G\widetilde{b}\|_{2}$, we have
\begin{align}
\sigma_{*}&\leq\tilde{\sigma}\leq3\sigma_{*},\label{ADDS part 1}
\\
\|G(\widetilde{b}-b_{*})\|_{2}&\leq8\sigma_{*}\eta\sqrt{n\|b_{*}\|_{0}/\kappa},\label{ADDS part 2}
\\
\|\widetilde{b}-b_{*}\|_{1}&\leq16\eta\sigma_{*}\|b_{*}\|_{0}/\kappa,\label{ADDS part 3}\\
\sigma_{*}/\sqrt{2} & \leq \hat{\sigma} \leq2\sigma_{*},\label{ADDS part 4}\\
\|n^{-1}G^{\top}(H-G\widetilde{b})\|_{\infty}\hat{\sigma}^{-1}&\leq3\sqrt{2}\eta.\label{ADDS part 5}
\end{align}
 \end{thm}

Notice that $\sigma_{*}$ might not always be equal to the standard deviation $\sigma_\varepsilon:=\sqrt{\mathbb{E}\|\varepsilon\|_{2}^{2}/n}$.
In fact, $\sigma_*$ is only a rough proxy of $\sigma_\varepsilon$ since any number in $[0.9\sigma_{\varepsilon},1.1\sigma_{\varepsilon}]$
can serve as  $\sigma_{*}$. This flexibility is especially useful for misspecified models, where $\varepsilon$ is correlated with the design and $\sigma_*$ is a quantity that depends on this correlation. Since $\tilde{\theta}$ defined in (\ref{eq: theta estimate}) does not contain $Z$, the alternative hypothesis corresponds to a misspecified regression and we shall derive the power properties by exploiting the aforementioned flexibility in the interpretation of $\sigma_*$. Due to this flexibility, it is not reasonable to expect consistent estimator for $\sigma_*$, but Theorem \ref{thm: ADDS} implies that ADDS estimator can generate an estimator that automatically approaches  $\|\varepsilon\|_{2}/\sqrt{n}$. Under Assumption \ref{assu:asds}, 
$$n^{-1/2}\left |\|H-G\widetilde{b}\|_{2}-\|\varepsilon\|_{2} \right|\leq n^{-1/2}\|G(\widetilde{b}-b_{*})\|_{2}\leq8\sigma_{*}\eta\sqrt{\|b_{*}\|_{0}/\kappa}.$$
In other words, whenever $\sigma_{*}\eta\sqrt{\|b_{*}\|_{0}/\kappa}=o(1)$, we obtain that $\|H-G\widetilde{b}\|_{2}/\sqrt{n}$ is a consistent estimator
for $\|\varepsilon\|_{2}/\sqrt{n}$.

\subsection{Size properties}

We now turn to the properties  of  the introduced TIERS test  while imposing extremely weak conditions when both $n$ and $p$ tend to $\infty$.

\begin{assumption}
\label{assu: regularity condition}Consider the models (\ref{eq: original model A})
and (\ref{eq: original model B}). Suppose that the following hold:\\
(i) the design follows Gaussian distributions: $x_{A}\sim \mathcal{N}(0,\Sigma_{A})$ and $x_{B}\sim \mathcal{N}(0,\Sigma_{B})$\\
(ii) there exist constants $\kappa_{1},\kappa_{2}\in(0,\infty)$ such
that the eigenvalues of $\Sigma_{A}$ and $\Sigma_{B}$ lie in $(\kappa_{1},\kappa_{2})$\\
(iii) for $s_{\Pi}=\max_{1\leq j\leq p}\|\Pi_{*,j}\|_{0}$, let  $s_{\Pi}=o\left(\sqrt{n/\log^{3}p}\right)$.
\end{assumption}

Conditions (i)-(iii) of Assumption \ref{assu: regularity condition} are very mild. 
  Gaussian designs are considered for the simplicity of the proofs. 
  We  study general   sub-Gaussian
designs in Section \ref{sec: non-Gaussian}.   Moreover, well-behaved designs 
 with bounded eigenvalues of the covariance matrix  is a common condition imposed in
the literature. Finally, condition (iii) imposes column-wise sparsity of the matrix $\Pi_{*}$.  When $\Sigma_{A}=c\Sigma_{B}$ for $c>0$, Lemma \ref{lem:1} implies that $\Pi_{*}=I_{p}(c-1)/(c+1)$,
regardless of the sparsity of $\Sigma_{A}$ and $\Sigma_{B}$, hence satisfying the imposed sparsity assumption. Observe that in contrast to the existing literature we do not assume sparsity of $\Sigma_A$ and $\Sigma_B$ (or their inverses),  by Lemma \ref{lem:1} we only require certain products to be approximately sparse, a condition that is weaker and more flexible; see \cite{cai2012direct}. 
Our first result is on the Type I error of the introduced TIERS test.
\begin{thm}
\label{thm: size result}Consider Algorithm \ref{alg: main}. Let
Assumption \ref{assu: regularity condition} hold. Then, under $H_{0}$
(\ref{eq: null hypo}),  as long as  $\log p=o(\sqrt{n})$ and $n \to \infty$, we have
\[
\mathbb{P}\left(T_{n}>\Gamma^{-1}(1-\alpha,\hat{Q})\right)\rightarrow\alpha\qquad\forall\alpha\in(0,1).
\]

\end{thm}
There are two unique features of this result. Firstly,  we do not assume
any sparsity condition on $\beta_{A}$ and $\beta_{B}$. This is remarkable
in high dimensions with $p\gg n$. Secondly, the result of Theorem \ref{thm: size result}  holds without imposing any restriction on the distribution of the
errors $u_{A}$ and $u_{B}$ in the models (\ref{eq: original model A})
and (\ref{eq: original model B}), respectively.  This surprising property is  achieved by the special
design of  the ``partial self-normalization''  of the test statistic and scale-free estimation of the introduced ADDS estimator. Notice that (\ref{eq: decomp}) holds does not require any assumption on the distribution of the error terms, such as the existence of probability densities. In light of this, we show that regardless of the sparsity of $\theta_*$ and/or the distribution of $U$, the  term $n^{-1/2}V^{\top}(Y-W\hat{\theta})\hat{\sigma}_{u}^{-1}$ ``partially self-normalized'' by $\hat{\sigma}_u=\|Y-W\hat{\theta}\|_2/\sqrt{n}$ is   free of  scales of the error terms  and has a normal distribution under $H_0$ (\ref{eq: null hypo}).
 Moreover, as pointed out by \cite{delapena2004}, 
self-normalization often eliminates or weakens moment assumptions. 
  
 \subsection{Power properties}
 
 Due to the convolved regression model (\ref{eq:model2}), we can  assess the power properties of the TIER test by considering the following alternative hypothesis
\begin{equation}  
H_{1}:\ \gamma_*\neq 0. \label{eq: alter hypo}
\end{equation}

It is clear
that the difficulty of differentiating  $H_0$ from $H_1$ depends on the magnitude of $\gamma_*$. We shall establish the rate for  magnitude of $\gamma_*$ such that our test has power approaching one; see Theorem \ref{thm: main result power sparse}. Later, we shall also show that this rate is optimal; see Theorems \ref{thm: LR test} and \ref{thm: minimax}. 
\begin{assumption}
\label{assu: regularity power}Let Assumption \ref{assu: regularity condition}
hold. In addition, suppose (1) that $s_{\Pi}\|\beta_{A}-\beta_{B}\|_{0}+\|\beta_{A}+\beta_{B}\|_{0}=o(\sqrt{n}/\log p)$ and (2) that there exist constant $\delta,\kappa_{3}\in (0,\infty)$ such that   $\mathbb{E}|u_{A}|^{2+\delta}\leq \kappa_{3}$ and $\mathbb{E}|u_{B}|^{2+\delta}\leq \kappa_{3}$. 
\end{assumption}
Assumption \ref{assu: regularity power} is reasonably  weak. It is
not surprising that certain sparse structure is needed to guarantee asymptotic 
power of high-dimensional tests. However, we still allow for lack of sparsity structure in the model parameters $\beta_{A}$
and $\beta_{B}$. In particular, we only require sparsity of  $\beta_{A}+\beta_{B}$
 and  that the product $s_{\Pi}\|\beta_{A}-\beta_{B}\|_{0}$ is  small.
For example, if $\beta_{A}=-\beta_{B}=\beta_{*}$ for some dense $\beta_{*}$
and $\Sigma_{A}=\Sigma_{B}$, then $s_{\Pi}=0$ and $\beta_{A}+\beta_{B}=0$,
satisfying Assumption \ref{assu: regularity power}.  Moreover, for sparse vectors $\beta_A$ and $\beta_B$ the rate condition $o(\sqrt{n}/\log p)$ matches those of one-sample testing; see  \cite{van2014asymptotically} and \cite{cai2015confidence}.

\begin{thm}
\label{thm: main result power sparse}Let $H_{1}$ in (\ref{eq: alter hypo})
and Assumption \ref{assu: regularity power} hold. Suppose that $n \to \infty$ and $p \to \infty$ with $\log p=o(\sqrt{n})$. Then, there exist
constants $K_{1},K_{2}>0$ depending only on the constants in Assumption
\ref{assu: regularity power} such that, whenever
$$\|\Sigma_{V}\gamma_{*}\|_{\infty}\geq\sqrt{n^{-1}\log p}(K_{1}\|\gamma_{*}\|_{2}+K_{2}),$$
with $\Sigma_V$ defined in Lemma \ref{lem:1}, the test in Algorithm \ref{alg: main} is asymptotically powerful, i.e.,
\[
\mathbb{P}\left(T_{n}>\Gamma^{-1}(1-\alpha,\hat{Q})\right)\rightarrow1\qquad\forall\alpha\in(0,1).
\]

\end{thm}

For power comparison, we consider two benchmarks in the next two sections: the most powerful
test, which is infeasible, and the minimax optimality. We show that,
in terms of rate for the magnitude of deviations from the null hypothesis,
our test differs from the most powerful test by only a logarithm factor
and achieves the minimax optimality whenever the model possesses certain
sparsity properties. In this sense, our test is efficient in sparse
settings and is robust to the lack of sparsity and heavy tails. 

\subsection{Efficiency}
In the rest of the section, we assume that the data is jointly Gaussian and derive the optimal power of the likelihood-ratio test of one distribution against another. As such a test is  the most
powerful test for distinguishing two given distributions  (\cite{lehmann2006testing}), we named it the oracle test. We proceed to show that the power of  our test differs from that of the  oracle test by a logarithmic factor.

Let the distribution of the data be  indexed by $\lambda=(\Sigma_{A},\Sigma_{B},\sigma_{u,A}^{2},\sigma_{u,B}^{2},\beta_{A},\beta_{B})$.
The probability, expectation and variance under $\lambda$ are denoted
by $\mathbb{P}_{\lambda}$, $\mathbb{E}_{\lambda}$ and $Var_{\lambda}$, respectively.
Consider the problem of testing 
\[
H_{0}:\ \lambda=(\Sigma_{A},\Sigma_{B},\sigma_{u,A}^{2},\sigma_{u,B}^{2},\beta_{A},\beta_{A})
\]
versus 
\[
H_{1}:\ \lambda=(\Sigma_{A},\Sigma_{B},\sigma_{u,A}^{2},\sigma_{u,B}^{2},\beta_{A},\beta_{A}+\gamma_{*})\ {\rm for\ a\ given}\ \gamma_{*}\neq0.
\]

\begin{thm}
\label{thm: LR test}Let the data be jointly Gaussian. Consider the likelihood ratio test of nominal
size $\alpha$ for the above problem. 
Then, as $n\rightarrow\infty$, the power of the likelihood ratio test   is 
\[
\Phi\left(d_{n}-\Phi^{-1}(1-\alpha)\right)+o(1),
\]
  with   
\[
d_{n}=\frac{\sqrt{n}\gamma_{*}^{\top}\Sigma_{B}\gamma_{*}}{\sqrt{(\gamma_{*}^{\top}\Sigma_{B}\gamma_{*})^{2}/2+\gamma_{*}^{\top}\Sigma_{B}\gamma_{*}\sigma_{u,B}^{2}}}.
\]

\end{thm}
 Due to the optimality of the likelihood ratio test, Theorem \ref{thm: LR test} says that there does not exist
any test that has power approaching one against the alternatives where  $\|\gamma_{*}\|_{2}=O(n^{-1/2})$,
even if  $\Sigma_{A}$, $\Sigma_{B}$, $\sigma_{u,A}^{2}$,
$\sigma_{u,B}^{2}$ and $\beta_{A}$ are known. In the extreme sparse setting with 
 $\|\beta_{A}\|_{0}=O(1)$ and $\|\beta_{B}\|_{0}=O(1)$, Theorem \ref{thm: LR test} in turn  implies 
that one should not expect perfect power against $\|\beta_{A}-\beta_{B}\|_{\infty}\asymp n^{-1/2}$.
On the other hand, Theorem \ref{thm: main result power sparse} says
that our test has asymptotically perfect power against $\|\beta_{A}-\beta_{B}\|_{\infty}\asymp\sqrt{n^{-1}\log p}$
when $\|\beta_{A}\|_{2}=O(1)$, $\|\beta_{B}\|_{2}=O(1)$ and $\Sigma_{V}$
is sparse. In this sense, our test is nearly optimal -- in terms of the magnitude of deviations from the null hypothesis,
our test differs from the most powerful test up to  a mere logarithm factor. 
\subsection{Minimax Optimality}
Notice that the critical value of the above likelihood ratio test
depends on $\gamma_{*}$ in the alternative hypothesis. In practice,
the value of $\gamma_{*}$ is often unknown; in fact, the values of
$\Sigma_{A}$, $\Sigma_{B}$, $\sigma_{u,A}$, $\sigma_{u,B}$ and
$\beta_{A}$ are usually unknown as well. We thus compare our test
with a benchmark test that has guaranteed power against a class of
alternative hypotheses (in terms of $\gamma_{*}$). 

We define 
\begin{align*}
\Lambda_{0}=\Bigl\{\lambda= & (\Sigma_{A},\Sigma_{B},\sigma_{u,A}^{2},\sigma_{u,B}^{2},\beta_{A},\beta_{B}):\\
 & \sigma_{\min}(\Sigma_{A}),\sigma_{\min}(\Sigma_{B}),\sigma_{\max}(\Sigma_{A}),\sigma_{\max}(\Sigma_{B}),\sigma_{u,A},\sigma_{u,B}\in[M_{1},M_{2}]\ {\rm and}\ \beta_{A}=\beta_{B}\Bigr\},
\end{align*}
where $M_{1},M_{2}\in(0,\infty)$ are constants. For $\tau>0$, we
also define
\begin{align*}
\Lambda(\tau)=\Bigl\{\lambda= & (\Sigma_{A},\Sigma_{B},\sigma_{u,A}^{2},\sigma_{u,B}^{2},\beta_{A},\beta_{B}):\\
 & \sigma_{\min}(\Sigma_{A}),\sigma_{\min}(\Sigma_{B}),\sigma_{\max}(\Sigma_{A}),\sigma_{\max}(\Sigma_{B}),\sigma_{u,A},\sigma_{u,B}\in[M_{1},M_{2}]\\
 & {\rm and}\ \|\Sigma_{V}(\beta_{B}-\beta_{A})\|_{\infty}\geq\tau\sqrt{n^{-1}\log p}(\|\beta_{B}-\beta_{A}\|_{2}+1)\Bigr\}.
\end{align*}
We consider the problem of testing 
\[
H_{0}:\ \lambda\in\Lambda_{0}\qquad{\rm versus}\qquad H_{1}:\ \lambda\in\Lambda(\tau).
\]
\begin{thm}
\label{thm: minimax}Let $\tau=(M_{1}+M_{2})/8$. Suppose that the
data is jointly Gaussian, $ \log p=o(\sqrt{n})$ and $p\rightarrow \infty$. Then for any test $\phi_{n}=\phi_{n}(Y_{A},Y_{B},X_{A},X_{B})$
satisfying $\limsup_{n\rightarrow\infty}\sup_{\lambda\in\Lambda_{0}}\mathbb{E}_{\lambda}\phi_{n}\leq\alpha$,
we have 
\[
\liminf_{n\rightarrow\infty}\inf_{\lambda\in\Lambda(\tau)}\mathbb{E}_{\lambda}\phi_{n}\leq\alpha.
\]

\end{thm}
Theorem \ref{thm: minimax} says that there does not exist any test
that has power against all the alternatives in $\Lambda(\tau)$ for
some fixed $\tau>0$. This means that power can only be guaranteed
uniformly against alternatives with deviations of magnitude of  at least
$\sqrt{n^{-1}\log p}$ in terms of $\|\Sigma_{V}(\beta_{B}-\beta_{A})\|_{\infty}/(\|\beta_{B}-\beta_{A}\|_{2}\vee1)$.
 Comparing with the the power that TIER test achieves 
in Theorem \ref{thm: main result power sparse}, the test  TIER is rate-optimal when  Assumption \ref{assu: regularity power} holds. 

\subsection{Considerations for Non-Gaussian Designs}\label{sec: non-Gaussian}

Here, we highlight the   extension of the proposed methodology for non-Gaussian designs.
For non-Gaussian designs, in order to better control the estimation error, we propose to augment the ADDS estimator for $\theta$ with an additional constraint on the size of the residuals. Namely, we define
\begin{equation}\label{eq: non-gaussian DS-2}
\begin{array}{cccc}
\check{\theta}(\sigma)&=\underset{\theta\in\mathbb{R}^{p}}{\arg\min}\ & \|\theta\|_{1}     \\
& \ \ \  \ \ \  \ \mbox{ s.t. } & \left \|n^{-1}W^{\top}(Y-W\theta)  \right \|_{\infty}&\leq\eta\sigma    \\
& \ \ \  \ \ \  \   & \left \| Y-W\theta \right \|_{\infty}&\leq\mu\sigma  
\end{array},
\end{equation}  
For some tuning parameter $\eta\asymp\sqrt{n^{-1}\log p}$ and $\mu\asymp n^{1/9}\log^{1/3}p$.
Then, we compute
\begin{equation}\label{eq: auto-scaling DS-2}
\begin{array}{cccc} 
\breve{\sigma}_{u}&=\underset{\sigma\geq0}{\arg\max}\  &  \sigma \\
& \ \ \  \ \ \  \ \mbox{ s.t. } &  {\|Y- W \check \theta (\sigma)\|_{2}^{2} }&\geq n{\sigma^{2}}/{2}
\end{array}.
\end{equation}
Now our estimator for $\theta_{*}$,  is defined as 
\begin{equation}
{\check{\theta}}^{{\tiny +}} =\check{\theta}(\breve{\sigma}_{u})\label{eq: theta estimate-2}.
\end{equation}

\begin{algorithm} 
\small
\caption{Testing of Equality of Regression Slopes with Non-Gaussian Design (TIERS$^{+}$)} \label{alg: main nonGaussian}
\begin{algorithmic}[1]
 \Require  Two samples $(X_A,Y_A)$ and $(X_B,Y_B)$ and  level $\alpha \in (0,1)$   of the test.
    \Ensure  Decision whether or not to reject the null Hypothesis (\ref{eq: null hypo})
    \State Given: $(X_A,Y_A)$ and $(X_B,Y_B)$ return   whether or not to reject the null Hypothesis (\ref{eq: null hypo}).
    \State Construct $W=X_{A}+X_{B}$, $Z=X_{A}-X_{B}$ and $Y=Y_{A}+Y_{B}$
    \State Compute $\check{\theta}$ in (\ref{eq: theta estimate-2}) with tuning parameters
$\eta\asymp\sqrt{n^{-1}\log p}$ and $\mu\asymp n^{1/9}\log^{1/3}p$. 
\State Compute $\widetilde{\Pi}$ in (\ref{eq: Pi estimate}) with tuning parameter
$\eta\asymp\sqrt{n^{-1}\log p}$. 
        \State Compute the test statistic $  T_{n}^+$ as 
        \begin{equation}
 T_{n}^+=n^{-1/2}\check{\sigma}_{u}^{-1} \|(Z-W\widetilde{\Pi})^{\top}(Y-W\check{\theta} )\|_{\infty}, \label{eq: test stat}
\end{equation}
 
 with $\check{\sigma}_{u}=n^{-1/2}\|Y-W\check{\theta} \|_{2}$ and
$\hat{Q}$ as in (\ref{eq: Q hat def}).  
        \State Compute  $\Gamma^{-1}(1-\alpha,\hat{Q})$  (by simulation), where $\Gamma^{-1}(\cdot,\hat{Q})$ is the inverse of $\Gamma(\cdot,\hat{Q})$.

   \Return  
  Reject $H_{0}$ (\ref{eq: null hypo}) if and only if $T_{n}>\Gamma^{-1}(1-\alpha,\hat{Q})$
 
\end{algorithmic}
\end{algorithm}

\subsubsection{Type I error control}

Next, we present size properties for the developed test TIERS$^+$ summarized in Algorithm \ref{alg: main nonGaussian}. Such a result requires certain 
  high-level conditions on the errors and the design, which we present below.
 \begin{assumption}
\label{assu: size non-Gaussian}Consider the model (\ref{eq: original model A})
and (\ref{eq: original model B}). Suppose that \\
(i)
There exist constants $\kappa_{1},\kappa_{2}\in(0,\infty)$ such
that the eigenvalues of both   $\Sigma_A=\mathbb{E}(x_{A}x_{A}^\top)$ and   $\Sigma_B=\mathbb{E}(x_{B}x_{B}^\top)$   lie in $(\kappa_{1},\kappa_{2})$.\\
 (ii) There exists a constant $\kappa_{3}\in(0,\infty)$ such that
  $ \| x_{A}\|_{\psi_2} < \kappa_3$ and $ \| x_{B}\|_{\psi_2}< \kappa_3$.  \\
(iii) The error of the model \eqref{eq: graphical model}, $v_{j}$, is independent of the design vector $w$ for all $1 \leq j \leq p$. \\
\end{assumption} 

Here, the design matrices are allowed to be sub-Gaussian and the errors of the original models  (\ref{eq: original model A})
and (\ref{eq: original model B}) are still unrestricted.  The following Theorem \ref{thm: size nonGaussian} shows that TIERS$^+$ test has asymptotically Type I-error  equal to the nominal level $\alpha \in(0,1)$.

\begin{thm}
\label{thm: size nonGaussian}Consider Algorithm \ref{alg: main nonGaussian}.
Let Assumption \ref{assu: size non-Gaussian} hold. Suppose that $n \to \infty$ and $p \to \infty$ with  $\log p = o(n^{1/27})$. Then, under $H_{0}$
(\ref{eq: null hypo}), 
\[
\mathbb{P}\left(  T_{n}^+>\Gamma^{-1}(1-\alpha,\hat{Q})\right)\rightarrow\alpha\qquad\forall\alpha\in(0,1).
\]
\end{thm}

Similar to Theorem \ref{thm: size result}, Theorem \ref{thm: size nonGaussian} is not based on the assumption of sparsity and does not need any assumption on the distribution of the error terms $u_A$ and $u_B$.
Theorem \ref{thm: size nonGaussian} establishes though somewhat stronger conditions on the growth of the dimension $ p $. Now,  $p$ can grow with $n$ as $\log p = o(n^{1/27})$, rather than $\log p=o(\sqrt{n})$. This can be considered as a price to pay for allowing for such weak distributional assumptions on both the errors and the designs of the models.
\subsubsection{Type II error control}

Now we turn to power considerations and establish asymptotically that TIERS$^+$ is powerful as long as certain assumptions on the model structure are imposed.

\begin{assumption}
\label{assu: power nonGaussian}Consider the model (\ref{eq: original model A})
and (\ref{eq: original model B}). Let Assumption \ref{assu: size non-Gaussian}
hold. In addition, suppose that   $\mathbb{E}|u_{A,i}|^{9}$ and $\mathbb{E}|u_{B,i}|^{9}$
are bounded above with a  constant $\kappa_{3}>0$ and that $s_{\Pi}\|\beta_{A}-\beta_{B}\|_{0}+\|\beta_{A}+\beta_{B}\|_{0}=o(\sqrt{n}/\log p)$ \end{assumption}

Assumption \ref{assu: power nonGaussian} only imposes bounded ninth moment of  the error distribution in addition to the size requirements of the model parameter (a condition needed for Gaussian designs as well -- see Assumption \ref{assu: regularity power}). 

The next result establishes asymptotic power of the TIERS$^+$ test for a class of alternatives defined in \eqref{eq: alter hypo} where $\gamma_* = (\beta_A-\beta_B)/2$.

\begin{thm}
\label{thm: power sparse nonGauss}Consider Algorithm \ref{alg: main nonGaussian}.
Let Assumption \ref{assu: power nonGaussian} hold. Suppose that $n \to \infty$ and $p \to \infty$ with $\log p = o(n^{1/27})$. Then, there exist
constants $K_{1},K_{2}>0$ depending only on the constants in Assumption
\ref{assu: power nonGaussian} and with $\Sigma_V$ as defined in Lemma \ref{lem:1} such that, whenever
$$\|\Sigma_{V}\gamma_{*}\|_{\infty}\geq\sqrt{n^{-1}\log p}(K_{1}\|\gamma_{*}\|_{2}+K_{2}),$$
the test in Algorithm \ref{alg: main nonGaussian} is asymptotically powerful, i.e.,
 \[
\mathbb{P}\left(  T_{n}^{+}>\Gamma^{-1}(1-\alpha,\hat{Q})\right)\rightarrow1\qquad\forall\alpha\in(0,1).
\]

\end{thm}

Per Theorem \ref{thm: power sparse nonGauss} we conclude that TIERS$^+$ preserves power properties similar to TIERS. In particular, whenever the model is sparse, the test achieves optimal power and does not lose efficiency compared to tests designed only for sparse models.

\section{Numerical Examples}\label{sec: MC simulations}
In this section we present finite-sample evidence of the accuracy of the proposed method. We consider two broad groups of examples: differential regressions and differential networks.

\subsection{Differential  Regressions}
 In all the setups, we consider $n=200$ and $p=500$. We consider the Toeplitz
design $x_{A}\sim \mathcal{N}(0,\Sigma_{A})$ and $x_{B}\sim \mathcal{N}(0,\Sigma_{B})$
with $(\Sigma_{A})_{i,j}=(0.4)^{|i-j|}$ and $\Sigma_{B}=c\Sigma_{A}$.
We consider two specifications for the model parameters:
\begin{itemize}
\item[(1)] In the sparse
regime, $\beta_{A}=(1,1,1,0,\cdots,0)^{\top}$ and $\beta_{B}=(1+h,1,1,0,\cdots,0)^{\top}$,
i.e., $\|\beta_{A}\|_{0}=\|\beta_{B}\|_{0}=3$;
\item[(2)]  In the dense regime,
$\beta_{A}=\zeta\|\zeta\|_{2}^{-1}$ with entries of $\zeta$ being
drawn from the uniform distribution on $[0,1]$ and $\beta_{B}=\beta_{A}+(h,0,\cdots,0)^{\top}$.
\end{itemize}
The null hypothesis $H_{0}$ (\ref{eq: null hypo}) corresponds to
$h=0$ and alternative hypotheses correspond to $h\neq0$. We also
consider two specifications for the error distributions
\begin{itemize}
\item[(a)] In the light-tail
case, $u_{A}$ and $u_{B}$ are drawn from the standard normal distribution, 
\item[(b)] In the
heavy-tail case, $u_{A}$ and $u_{B}$ are drawn from the standard
Cauchy distribution. 
\end{itemize}

These different specifications will be denoted as follows: SL (for sparse
and light-tail), SH (for sparse and heavy-tail), DL (for dense and light-tail)
and DH (for dense and heavy-tail). 


\begin{table*}[h]
\caption{Rejection probability for $c=2$} \label{tab:1}
\centering
\ra{0.50}
\begin{tabular}{@{}rrrrcrrrcrrrcrrr@{}}\toprule
& \multicolumn{2}{c}{SL} & \phantom{a}& \multicolumn{2}{c}{SH} &
\phantom{a} & \multicolumn{2}{c}{DL}&\phantom{a} & \multicolumn{2}{c}{DH}\phantom{a}\\ 
\cmidrule{2-3} \cmidrule{5-6} \cmidrule{8-9} \cmidrule{11-12}
& $h$ & $P$-{\tiny TIERS}  && $h$ & $P$-{\tiny TIERS}   && $h$ & $P$-{\tiny TIERS}  && $h$ & $P$-{\tiny TIERS} \\ \midrule
$$\\ 
$$ & 0.00 & 4\%  && 0& 2\% && 0 & 3\% && 0 & 2\%\\
$$ &0.32& 14\% && 4 & 16\%  && 0.48&6\% && 4 & 10\%\\
$$ &0.44&42\% && 8& 42\%&& 2.16 & 44\%&& 8 & 40\%\\
$$\\
$$ & 0.48& 52\%&& 12 & 57\% &&2.40 & 56\% &&12 & 56\% \\
$$ & 0.52&67\% && 16& 67\% &&2.64 & 68\%&&16& 68\%\\
$$ &0.60 &78\% && 20 & 74\%  && 2.88 & 78\%&& 20 & 74\% \\
$$\\
$$ & 0.64& 88\%&& 24 & 83\% &&3.12 & 85\% &&24 & 82\% \\
$$ & 0.68&92\% && 44& 90\%&&3.36& 90\%&&44& 90\%\\
$$ &0.72&96\% && 68 & 93\% && 3.60& 92\%&& 64 & 92\% \\
$$\\
$$ & 0.76& 97\%&& 124 & 86\% &&3.84 & 95\% &&124 & 96\% \\
$$ & 0.88&99\% && 168& 97\%&&4.08& 98\%&&168& 97\%\\
$$ &0.92&100\%&& 268 &100\% && 4.56& 100\%&& 264 & 99\% \\
\bottomrule
\end{tabular}
\end{table*}

\begin{table*}[h]
\caption{Rejection probability for $c=1$} \label{tab:2}
\centering
 \ra{0.50}
\begin{tabular}{@{}rrrrcrrrcrrrcrrr@{}}\toprule
& \multicolumn{2}{c}{SL} & \phantom{a}& \multicolumn{2}{c}{SH} &
\phantom{a} & \multicolumn{2}{c}{DL}&\phantom{a} & \multicolumn{2}{c}{DH}\phantom{a}\\
\cmidrule{2-3} \cmidrule{5-6} \cmidrule{8-9} \cmidrule{11-12}
& $h$ & $P$-{\tiny TIERS}  && $h$ & $P$-{\tiny TIERS}   && $h$ & $P$-{\tiny TIERS}  && $h$ & $P$-{\tiny TIERS} \\ \midrule
$$\\ 
$$ & 0.00 & 6 \% && 0&  2 \% && 0 & 2 \% && 0 & 6 \%\\
$$ &0.44&7 \%&& 8&10 \%&& 2.16 &9 \%&& 8 & 9 \%\\
$$ & 0.48& 8 \%&& 12 & 23 \% &&2.40 & 12 \%&&12 & 28 \% \\
$$\\
$$ & 0.96& 51  \% && 28&  56 \%&&3.60 &  50 \%&&28&  55 \%\\
$$ & 1.04& 65 \% &&48&  69 \%&&4.08 &  69 \%&&44& 67 \%\\
$$ &1.16 &78  \%&& 56 & 77 \%  && 4.32 & 79 \%&& 56 & 78 \% \\
$$\\
$$ & 1.20&  85 \%&& 64 & 82 \% &&4.56&  82 \% &&64 &  82 \% \\
$$ & 1.24& 93 \% && 108&  90 \%&&5.04 &  92 \%&&92&  90 \%\\
$$ &1.28& 95  \%&& 140 &  93  \%&& 5.38&  93 \%&& 136 & 92  \%\\
$$\\
$$ & 1.32&  96 \%&& 200&  94 \% &&5.76 &  96  \%&&200 &  95 \% \\
$$ & 1.60& 99 \% &&284&  97 \%&&6.00&  97 \%&&284 &  97 \%\\
$$ &1.64&100 \%&&320&100  \%&& 6.24& 100 \%&& 330& 100 \%\\
\bottomrule
\end{tabular}

\end{table*}

The summary of the results is presented in Tables \ref{tab:1} - \ref{tab:3} where the rejection probabilities are computed based on  $100$ repetitions. 
The tuning parameter $\eta$  is chosen as  adaptively as 
$$\eta=\sqrt{2\log (p)/n} \max_{1\leq j\leq p}\|W_j\|_2/\sqrt{n}.$$ 
We vary deviations from the null to highlight power properties as well as the Type I errors.
From the tables we observe that the TIERS  performs exceptionally well in both Gaussian and heavy-tailed setting for the error terms. The rejection probabilities under the null hypothesis ($h=0 $) are close to the nominal size 5\% in all the settings, as expected from our theory. 

In the case of  light-tailed models (SL and DL) TIERS achieves  perfect power relatively quickly independent of the sparsity of the underlying model. The dense case required larger deviations from the null to reach power of one.
 In the case of heavy-tailed models,  TIERS shows excellent performance irrespective of the sparsity of the model. In comparison with the models with light-tailed errors,  the models with  the heavier tails need larger deviations from the null in order to reach the same power. This is expected as the simulated Cauchy errors  had an average variance of   
 about $140$ over $100$ independent repetitions. Remarkably, dense models with heavy tailed errors performed extremely close to those of sparse models with heavy tailed errors, indicating that the heavy tails are the main driver of the power loss.

\begin{table*}
\caption{Rejection probability for $c=1/2$} \label{tab:3}
\centering
 \ra{0.50}
\begin{tabular}{@{}rrrrcrrrcrrrcrrr@{}}\toprule
& \multicolumn{2}{c}{SL} & \phantom{a}& \multicolumn{2}{c}{SH} &
\phantom{a} & \multicolumn{2}{c}{DL}&\phantom{a} & \multicolumn{2}{c}{DH}\phantom{a}\\
\cmidrule{2-3} \cmidrule{5-6} \cmidrule{8-9} \cmidrule{11-12}
& $h$ & $P$-{\tiny TIERS}  && $h$ & $P$-{\tiny TIERS}   && $h$ & $P$-{\tiny TIERS}  && $h$ & $P$-{\tiny TIERS} \\  \midrule
$$\\ 
$$ & 0.00 & 2 \%  && 0&  5 \% && 0 & 4 \% && 0 & 7 \%\\
$$ &0.28& 12 \% && 4 & 10  \% && 0.48&6 \% && 4 &  12 \%\\
$$ &0.44&23 \% && 8&  26 \%&& 2.16 & 32 \%&& 8 & 26 \%\\
$$\\
$$ & 0.48& 27 \%&& 12 &  48 \% &&2.40 & 46 \% &&12 &  47 \% \\
$$ & 0.52& 36 \% && 16&  56 \% &&2.64 &  56 \%&&16& 56 \%\\
$$ &0.64 & 64  \%&& 20 &  62  \% && 2.88 &  66 \%&& 20 &61 \% \\
$$\\
$$ & 0.72& 75 \%&& 32 & 74  \%&&3.12 &  77 \% &&32 & 74 \% \\
$$ & 0.80& 85 \% && 72&  86 \%&&3.36&  85 \%&&72&  86 \%\\
$$ &0.84& 91  \%&& 124 &  92 \% && 3.60&  91 \%&& 124 & 92 \% \\
$$\\
$$ & 0.92& 97 \%&& 160 &  95 \% &&3.84 &  95 \% &&180 &  9 \% \\
$$ & 1.04& 99 \% && 244&  97 \%&&4.32&  97 \%&&244&  97 \%\\
$$ &1.20&100 \%&& 288 &100 \% && 4.80& 100 \%&& 288 &  99 \% \\
\bottomrule
\end{tabular}

\end{table*}

\subsection{Differential  Networks}

We
consider two independent Gaussian graphical models, where $z_{A}\sim \mathcal{N}(0,\Omega_{A}^{-1})$
and $z_{B}\sim \mathcal{N}(0,\Omega_{B}^{-1})$. The goal is to conduct inference
on the first row of the precision matrices by testing 
\[
H_{0}:\ \Omega_{A,1,j}=\Omega_{B,1,j}, \qquad\forall1\leq j\leq p.
\]

We consider the follow specification
\begin{equation*}
\Omega_{A}=\begin{pmatrix}\sigma^{-2} & -\beta_{A}^{\top}\sigma^{-2}\\
-\beta_{A}\sigma^{-2} & \Omega+\beta_{A}\beta_{A}^{\top}\sigma^{-2}
\end{pmatrix}\qquad {\rm and}\qquad \Omega_{B}=\begin{pmatrix}\sigma^{-2} & -\beta_{B}^{\top}\sigma^{-2}\\
-\beta_{B}\sigma^{-2} & \Omega+\beta_{B}\beta_{B}^{\top}\sigma^{-2}
\end{pmatrix}.
\end{equation*}

Notice that we can write the model in the regression form: $z_{A,1}=z_{A,-1}^{\top}\beta_{A}+\varepsilon_{A}$ and $z_{B,1}=z_{B,-1}^{\top}\beta_{B}+\varepsilon_{B}$, where $\varepsilon_{A}$ and $\varepsilon_B$ are the error terms. Therefore, it is equivalent to testing  
\[
H_0:\ \beta_A=\beta_B.
\]

We set $\sigma=0.5$ and 
\[
\beta_{B,j}=\begin{cases}
\beta_{A,1}+h & {\rm if}\ j=1\\
\beta_{B,j} & {\rm otherwise}.
\end{cases}
\]

As in the linear case, $h$ represents the deviations from the null
hypothesis and $h=0$ corresponds to the null hypothesis. We consider
two cases for $\beta_{A}$:
\begin{enumerate}
        \item In the sparse case, $\beta_{A}=(1,1,1,0,\cdots,0)^{\top}/\sqrt{3}$.
        \item In the dense case, $\beta_{A}=\zeta\|\zeta\|_{2}^{-1}$ with entries
        of $\zeta$ being drawn from the uniform distribution on $[0,1]$.
\end{enumerate}

We consider two regimes for $\Omega$: sparse regime and dense regime. 
\begin{itemize}
        \item[(a)]  For the sparse regime, we follow \cite{ren2015asymptotic} by setting
        $\Omega={\rm D}$, where ${\rm D}\in\mathbb{R}^{(p-1)\times(p-1)}$
        is block diagonal with ${\rm D}_{1}\in\mathbb{R}^{p_{1}\times p_{1}}$,
        ${\rm D}_{2}\in\mathbb{R}^{p_{2}\times p_{2}}$ and ${\rm D}_{3}\in\mathbb{R}^{p_{2}\times p_{2}}$
        on the diagonal. Here, $p=500$, $p_{1}=249$ and $p_{2}=125$. For
        $k=1,2,3$, we set ${\rm D}_{k,j,j}=\alpha_{k}$, ${\rm D}_{k,j,j-1}={\rm D}_{k,j-1,j}=0.5\alpha_{k}$
        and ${\rm D}_{k,j,j-2}={\rm D}_{k,j-2,j}=0.4\alpha_{k}$, where $\alpha_{1}=1$,
        $\alpha_{2}=2$ and $\alpha_{3}=4$.
        \item[(b)] For the dense regime, we set $\Omega={\rm D}^{-1}$. 
\end{itemize}
It is worth pointing out that no existing method applies to the setting of $\Omega={\rm D}^{-1}$, i.e. the setting of large-scale and dense graphical models.
The results are summarized in Table \ref{tab: GGM}, where four cases
are considered: {\small{}S$\beta$+S$\Omega$ (for sparse $\beta_{A}$
        and sparse $\Omega$), D$\beta$+S$\Omega$ (for dense $\beta_{A}$
        and sparse $\Omega$), S$\beta$+D$\Omega$ (for sparse $\beta_{A}$
        and dense $\Omega$) and D$\beta$+D$\Omega$ (for dense $\beta_{A}$
        and dense $\Omega$). }{\small \par}
In both sparse and dense settings, TIER performs well in terms of
(1) controlling the size for $h=0$ and (2) exhibiting power against
alternatives $h\neq0$. As should be expected, TIER has better power
in sparse specifications than in dense specifications, although the
power eventually reaches on in all the settings. 

\begin{table}
        \caption{\label{tab: GGM}Rejection probabilities for Gaussian graphical models}     
        \centering{}{\small{}}%
          \ra{0.60}  
        \begin{tabular}{ccc||ccc}
              \toprule
                {\small{}$h$} & {\small{}S$\beta$+S$\Omega$} & {\small{}D$\beta$+S$\Omega$} & {\small{}$h$} & {\small{}S$\beta$+D$\Omega$} & {\small{}D$\beta$+D$\Omega$}\\  
               \midrule
                0.00 & 5\% & 3\% & 0.00 & 5\% & 7\%\tabularnewline
                0.48 & 14\% & 11\% & 1.68 & 15\% & 5\%\tabularnewline
                0.52 & 21\% & 13\% & 1.82 & 30\% & 6\%\tabularnewline
                &  &  &  &  & \tabularnewline
                0.60 & 35\% & 25\% & 1.96 & 41\% & 8\%\tabularnewline
                0.72 & 59\% & 51\% & 2.10 & 53\% & 8\%\tabularnewline
                0.84 & 76\% & 71\% & 2.38 & 69\% & 15\%\tabularnewline
                &  &  &  &  & \tabularnewline
                0.88 & 81\% & 77\% & 2.66 & 78\% & 25\%\tabularnewline
                0.92 & 84\% & 82\% & 3.78 & 89\% & 69\%\tabularnewline
                0.96 & 88\% & 88\% & 3.92 & 90\% & 69\%\tabularnewline
                &  &  &  &  & \tabularnewline
                1.08 & 94\% & 97\% & 5.00 & 96\% & 92\%\tabularnewline
                1.44 & 99\% & 99\% & 7.00 & 98\% & 95\%\tabularnewline
                1.60 & 100\% & 100\% & 8.00 & 100\% & 100\%\tabularnewline
        \bottomrule 
        \end{tabular}{\small \par}
\end{table}

\section*{Conclusions and discussions}
  
We have presented a framework for performing inference on the equivalence of the coefficients vectors
between two linear regression models. We show
 that the rejection probability of the proposed tests under the null hypothesis converges to the  nominal size under extremely weak conditions: (1) no    assumption on the structure or sparsity of the coefficient vector and (2) no assumption on  the error distribution. If the features in the two samples have the same variance matrix, then our result does not require any assumption on the feature covariance either.  To the best of our knowledge, this is the first result for
two-sample testing in the high-dimensional setting that is robust
to the failure of the model sparsity and thus is truly novel.  Moreover, we establish both efficiency and optimality of our procedure.
   Applying this
framework to performing inference on differential regressions and differential networks, we obtain new procedures of constructing accurate  inferences in situations
not previously addressed in the literature.

Our work also opens doors to new research areas in statistics. In
terms of methodology, our work exploits the implication of the null
hypothesis and can be extended to other inference problems. For example,
it is practically important to extend our method to the high-dimensional ANCOVA problems, where the hypothesis of interest involves equivalence of parameters in multiple samples. Another extension is the inference of partial equivalence. Consider
the problem of two or more samples generated by linear models, where
the goal is to test the hypothesis that certain components in the model
parameter are identical in the two samples.  One important application is specification tests in large-scale models. Suppose
that the data is collected from many different sources and contain
several subgroups of observations. A common approach of extracting
information from these datasets is to assume that all the subgroups
are generated from linear models but these different subgroups share
the same values in certain entries of the model parameters. Numerous
methods  have been developed in order to estimate the common components.
However, as far as we know,  no work exists  that can be used to verify
whether it is reasonable to assume common components. Our work can be  further extended to provide simple
tests for the specification of the common parameter values.

\appendix

\section{Sufficient condition of Gaussian approximation in high dimensions }\label{sec: appendix HD CLT}

 For the reader's convenience,
we state the following result that applies to the conditional probability
measure $\mathbb{P}(\cdot\mid\mathcal{F}_{n})$. 
This result can be used to verify assumption (ii) of Theorem \ref{thm: approx assuming CLT}.\begin{prop}[Proposition 2.1 of \cite{chernozukov2014central}]
\label{prop: HD CLT CCK} Let $\{\Psi_{i}\}_{i=1}^{n}$ be a sequence
of random vectors in $\mathbb{R}^{p}$ such that, conditional on some
$\sigma$-algebra $\mathcal{F}_{n}$, is independent across $i$ and
has zero mean. Let  $Q=n^{-1}\sum_{i=1}^{n}\mathbb{E}(\Psi_{i}\Psi_{i}^{\top}\mid\mathcal{F}_{n})$. Suppose that
\begin{itemize}
\item[(i)] there exists a constant $b\in(0,\infty)$ such that $\mathbb{P}(\min_{1\leq j\leq p}Q_{j,j}\geq b)\rightarrow1$.
\item[(ii)] there exists a sequence of $\mathcal{F}_{n}$-measurable random
variables $B_{n}>0$ such that $B_{n}=o(\sqrt{n}/\log^{7/2}(pn))$,
$\max_{1\leq j\leq p}n^{-1}\sum_{i=1}^{n}\mathbb{E}(|\Psi_{i,j}|^{3}\mid\mathcal{F}_{n})\leq B_{n}$
and $\max_{1\leq j\leq p}n^{-1}\sum_{i=1}^{n}\mathbb{E}(|\Psi_{i,j}|^{4}\mid\mathcal{F}_{n})\leq B_{n}^{2}$.
\item[(iii)] either (1) $\max_{1\leq i\leq n,\ 1\leq j\leq p}\mathbb{E}[\exp(|\Psi_{i,j}|/B_{n})\mid\mathcal{F}_{n}]\leq2$
or \\ (2) $\max_{1\leq i\leq n}\mathbb{E}(\max_{1\leq j\leq p}|\Psi_{i,j}|^{q}\mid\mathcal{F}_{n})\leq B_{n}^{q}$
for some $B_{n}=o(n^{1-2/q}/\log^{3}(pn))$ and $q>0$. 
\end{itemize}
Then, when $n\to \infty$ and $p \to \infty $, we have 
\[
\sup_{x\in\mathbb{R}}\left|\mathbb{\mathbb{P}}\left(\left\Vert n^{-1/2}\sum_{i=1}^{n}\Psi_{i}\right\Vert _{\infty}\leq x\mid\mathcal{F}_{n}\right)-\Gamma(x,Q)\right|=o_{P}(1).
\]

\end{prop}

\section{Proof of Lemma 1  }\label{sec: proof Sec 1}

\begin{proof}[\textbf{Proof of Lemma \ref{lem: naive example}}]
        Let $Y_{k}=(y_{k,1},\cdots,y_{k,n})^\top \in \mathbb{R}^n$ and \\ $X_{k}=(x_{k,1},\cdots,x_{k,n})^\top \in \mathbb{R}^{n\times p}$ and define the event $\mathcal{J}=\mathcal{J}_{A}\bigcap\mathcal{J}_{B}$,
        where 
        $$\mathcal{J}_{k}=\{\|n^{-1}X_{k}^{\top}Y_{k}\|_{\infty}\leq\lambda n^{-1/2}\|Y_{k}\|_{2} \}$$
        for $k\in\{A,B\}$. We proceed in two steps: (1) show the result assuming
        $\mathbb{P}(\mathcal{J})\rightarrow1$ and (2) show $\mathbb{P}(\mathcal{J})\rightarrow1$.
        
        \vskip 5pt
        \textbf{Step 1: show the result assuming $\mathbb{P}(\mathcal{J})\rightarrow1$. }
              \vskip 5pt
              
        Let $\beta_{*}=\mathbf{1}_{p}c/\sqrt{n}$ and 
        $$\zeta=\sqrt{n}(\hat{\Sigma}_{A}-\hat{\Sigma}_{B})\beta_{*}+n^{-1/2}\sum_{i=1}^{n}(x_{A,i}u_{A,i}-x_{B,i}u_{B,i}).$$
        For $k\in\{A,B\}$, on the event $\mathcal{J}_{k}$, the trivial estimator $\bar{\beta}_{k}=0$
        satisfies the KKT optimality conditions of the scaled Lasso optimization (\ref{eq: scaled lasso}):
        $$
          \begin{cases}
          n^{-1}X_{k,j}^{\top}(Y_{k}-X_{k}\bar{\beta}_{k})=\lambda \bar{\sigma}_{k} {\rm sign}(\bar \beta_{k,j})& {\rm if\ } \bar \beta_{k,j}\neq0\\
|n^{-1}X_{k,j}^{\top}(Y_{k}-X_{k}\bar{\beta}_{k})|\leq\lambda \bar{\sigma}_{k} & {\rm if\ }\bar \beta_{k,j}=0\\
          \bar{\sigma}_{k}=\|Y_{k}-X_{k}\bar{\beta}_{k}\|_{2}/\sqrt{n},
          \end{cases}
        $$
        where $\beta_{k,j}$ denotes the $j$-th entry of $\beta_k$, $X_{k,j}$ denotes the $j$-th column of $X_k$ and ${\rm sign}(\cdot)$ is the sign function defined by  ${\rm sign}(x)=x/|x|$ for $x\neq0$ and  ${\rm sign}(0)=0$.  Therefore, on the event $\mathcal{J}$, a trivial solution of all zeros is a scaled Lasso solution, i.e. $\hat{\beta}_{A}=\hat{\beta}_{B}=0$.
        Since $\beta_{A}=\beta_{B}=\beta_{*}$ (by $H_{0}$ (\ref{eq: null hypo})),
        we have that, on the event $\mathcal{J}$, 
        \begin{align*}
        \sqrt{n}(\tilde{\beta}_{A}-\tilde{\beta}_{B})&=n^{-1/2}X_{A}^{\top}Y_{A}-n^{-1/2}X_{B}^{\top}Y_{B}
        \\
       & =\sqrt{n}(\hat{\Sigma}_{A}\beta_{A}-\hat{\Sigma}_{B}\beta_{B})+n^{-1/2}\sum_{i=1}^{n}(x_{A,i}u_{A,i}-x_{B,i}u_{B,i})\\
        &=\sqrt{n}(\hat{\Sigma}_{A}-\hat{\Sigma}_{B})\beta_{*}+n^{-1/2}\sum_{i=1}^{n}(x_{A,i}u_{A,i}-x_{B,i}u_{B,i})=\zeta.
        \end{align*}

        Since $\mathbb{P}(\mathcal{J})\rightarrow1$, we have 
        \begin{equation}
        \mathbb{P}\left(\sqrt{n}(\tilde{\beta}_{A}-\tilde{\beta}_{B})=\zeta\right)\rightarrow1.\label{eq: naive example eq 1}
        \end{equation}

        Notice that, conditional on $\{(x_{A,i},x_{B,i})\}_{i=1}^{n}$, $\zeta$
        is Gaussian with mean $\sqrt{n}(\hat{\Sigma}_{A}-\hat{\Sigma}_{B})\beta_{*}$
        and variance $\hat{\Sigma}_{A}+\hat{\Sigma}_{B}$. 
        
        Therefore, $F(\cdot,c,\hat{\Sigma}_{A},\hat{\Sigma}_{B})$
        is the conditional distribution function of \\ $\max_{1\leq j\leq p}|\zeta_{j}|/\sqrt{\hat{\Sigma}_{A,j,j}+\hat{\Sigma}_{B,j,j}}$.
        Thus, the desired result follows by (\ref{eq: naive example eq 1}).
        
              \vskip 5pt
        \textbf{Step 2: show $\mathbb{P}(\mathcal{J})\rightarrow1$. }
              \vskip 5pt
              
        We show $\mathbb{P}(\mathcal{J}_{k})\rightarrow1$. Recall that $\lambda=\lambda_0\sqrt{n^{-1}\log p}$. Let $U_{k}=Y_{k}-X_{k}\beta_{k}$.   For simplicity,
        we drop the subscript $k$ and write $X$, $Y$ and $U$, instead
        of $X_{k}$, $Y_{k}$ and $U_{k}$. Moreover, for $1\leq l\leq p$,
        we define $X_{l}$ as the $l$-th column of $X$ and 
        $$\varepsilon_{l}=n^{-1/2}c\sum_{j=1,j\neq l}^{p}X_{j}+U.$$
        
Notice that $X_{l}\perp\varepsilon_{l}$
        and $Y=cn^{-1/2}X_{l}+\varepsilon_{l}$ for any $1\leq l\leq p$. Let $\lambda_{1}$    be a constant satisfying $\sqrt{2}<\lambda_{1}<\lambda_{0}$;
        this is possible since $\lambda_{0}>\sqrt{2}$ is a constant. 
        
        Observe that 
        $$\mathbb{E}\|Y\|_2^2/n=\sigma_*^2$$ with $\sigma_{*}=\sqrt{n^{-1}pc^{2}+1}$.  Since $n^{-1}\|Y\|_{2}^{2}\sigma_*^{-2}$ is the average of $n$
        independent $\chi^{2}(1)$ random variables, the classical central
        limit theorem implies that 
        $$n^{-1}\|Y\|_{2}^{2}\sigma_*^{-2}=1+O_{P}(n^{-1/2}).$$
        This means that 
        \[
        \mathbb{P}\left(n^{-1/2}\|Y\|_{2}>\left(1-n^{-1/2}\log^{1/4}p\right)\sigma_{*}\right)\rightarrow1.
        \]

Since $n^{-1}\|X_{l}\|_{2}^{2}$ is also the average
        of $n$ independent $\chi^{2}(1)$ random variables, which are sub-exponential,
        the Bernstein's inequality and the union bound imply that, there exists a constant $c_1>0$ such that   
        \[
        \mathbb{P}\left(\max_{1\leq l\leq p}n^{-1}\|X_{l}\|_{2}^{2}<1+c_{1}\sqrt{n^{-1}\log p}\right)\rightarrow1.
        \]

        Hence, $ \mathbb{P}(\mathcal{A})\rightarrow1$ with
        \begin{equation}
     \mathcal{A}=\left\{ \max_{1\leq l\leq p}n^{-1}\|X_{l}\|_{2}^{2}<1+c_{1}\sqrt{n^{-1}\log p}\right\} \bigcap\left\{ n^{-1/2}\|Y\|_{2}>\left(1-n^{-1/2}\log^{1/4}p\right)\sigma_{*}\right\} .\label{eq: naive test eq 2}
        \end{equation}

        Conditional on $X_{l}$, $n^{-1/2}X_{l}^{\top}\varepsilon_{l}$ is
        Gaussian with mean zero and variance 
        $\sigma_{l}^{2}=n^{-1}\|X_{l}\|_{2}^{2}\sigma_{+}^{2}$, where $\sigma_{+}^{2}=n^{-1}c^{2}(p-1)+1$.
         Recall the elementary
        inequality that for $Z\sim \mathcal{N}(0,\sigma^{2})$ and $t>0$, $\mathbb{P}(|Z|>t)\leq2\exp(-t^{2}/(2\sigma^{2}))$.
        It follows that 
        \[
        \mathbb{P}\left(|n^{-1/2}X_{l}^{\top}\varepsilon_{l}|>\sigma_{+}t\ {\rm and}\ \mathcal{A}\right)\leq2\mathbb{E}\exp\left[-\frac{t^{2}}{2}/\left(1+c_{1}\sqrt{n^{-1}\log p}\right)\right]\qquad\forall t>0.
        \]

        Therefore, 
        \begin{align}
         &\mathbb{P}\left(\max_{1\leq l\leq p}|n^{-1/2}X_{l}^{\top}\varepsilon_{l}|>\lambda_{1}\sigma_{+}\sqrt{\log p}\right)\nonumber  
         \\ & \leq\mathbb{P}(\mathcal{A}^{c})+\sum_{l=1}^{p}\mathbb{P}\left(|n^{-1/2}X_{l}^{\top}\varepsilon_{l}|>\lambda_{1}\sigma_{+}\sqrt{\log p}\ {\rm and}\ \mathcal{A}\right)\nonumber \\
        & \leq\mathbb{P}(\mathcal{A}^{c})+2p\exp\left[-\frac{\lambda_{1}^{2}\log p}{2\left(1+c_{1}\sqrt{n^{-1}\log p}\right)}\right]\overset{(i)}{=}o(1),\label{eq: naive example eq 3}
        \end{align}
        where $(i)$ holds by $\lambda_{1}>\sqrt{2}$ and (\ref{eq: naive test eq 2}).
        Since 
        $$Y=cn^{-1/2}X_{l}+\varepsilon_{l},$$ we have  $n^{-1/2}X_{l}^{\top}Y=n^{-1}\|X_{l}\|_{2}^{2}c+n^{-1/2}X_{l}^{\top}\varepsilon_{l}$.
        It follows, by (\ref{eq: naive example eq 3}) and (\ref{eq: naive test eq 2}),
        that 
        \begin{multline}
        \mathbb{P}\left(\max_{1\leq l\leq p}|n^{-1}X_{l}^{\top}Y|>\lambda_{1}\sigma_{+}\sqrt{n^{-1}\log p}+2cn^{-1/2}\right)\\
        \leq\mathbb{P}\left(\max_{1\leq l\leq p}|n^{-1/2}X_{l}^{\top}\varepsilon_{l}|>\lambda_{1}\sigma_{+}\sqrt{\log p}\right)+\mathbb{P}\left(\max_{1\leq l\leq p}n^{-1}\|X_{l}\|_{2}^{2}>2\right)=o(1).\label{eq: naive example eq 5}
        \end{multline}

        Notice that 
        \begin{align}
         &  \mathbb{P}\left(n^{-1/2}\|Y\|_{2}\lambda\geq\lambda_{1}\sigma_{+}\sqrt{n^{-1}\log p}+2cn^{-1/2}\right) \label{eq: naive example eq 6}\\
          =& \mathbb{P}\left(n^{-1/2}\|Y\|_{2}\lambda_{0}\sqrt{n^{-1}\log p}\geq\lambda_{1}\sigma_{+}\sqrt{n^{-1}\log p}+2cn^{-1/2}\right)\nonumber \\
         \geq & \mathbb{P}\left(n^{-1/2}\|Y\|_{2}\lambda_{0}\sqrt{n^{-1}\log p}\geq\lambda_{1}\sigma_{+}\sqrt{n^{-1}\log p}+2cn^{-1/2}\ {\rm and}\ \mathcal{A}\right)\nonumber \\
        \overset{(i)}{\geq} & \mathbb{P}\left(\left(1-n^{-1/2}\log^{1/4}p\right)\sigma_{*}\lambda_{0}\geq\lambda_{1}\sigma_{+}+2c/\sqrt{\log p}\ {\rm and}\ \mathcal{A}\right)\nonumber \\
         \geq & \mathbb{P}\left(\left(1-n^{-1/2}\log^{1/4}p\right)\sigma_{*}\lambda_{0}\geq\lambda_{1}\sigma_{+}+2c/\sqrt{\log p}\right)-\mathbb{P}(\mathcal{A}^{c})         \overset{(ii)}{\geq}  1-o(1),  \nonumber 
        \end{align}
        where $(i)$ holds by  (\ref{eq: naive test eq 2})  and $(ii)$  holds by (\ref{eq: naive test eq 2}) and $\lambda_{1}<\lambda_{0}$.
        Hence, by (\ref{eq: naive example eq 5}) and (\ref{eq: naive example eq 6}),
        we have $\mathbb{P}(\|n^{-1}X^{\top}Y\|_{\infty}\leq\lambda  n^{-1/2}\|Y\|_{2} )\rightarrow1$.
        The proof is complete. 
\end{proof}

\subsection{Proof of  Theorem  \ref{thm: approx assuming CLT} }
\begin{proof}[\textbf{Proof of Theorem \ref{thm: approx assuming CLT}}]
 For notational convenience, we denote by $\mathbb{P}_{\mathcal{F}_{n}}(\cdot)$
by $\mathbb{P}(\cdot\mid\mathcal{F}_{n})$. Let $\bar{G}=\max_{1\leq j\leq p} |G_{n,j}| $. We proceed in two steps,
where we show that (1) $\mathbb{P}_{\mathcal{F}_{n}}(\bar{G}\leq x)$
can be approximated by $\Gamma(x,Q)$ and (2) $\Gamma(x,Q)$ can be approximated by $\Gamma(x,\hat{Q})$. Since $\|\hat{Q}-Q\|_{\infty}=o_{P}(1)$, we
have that 
\begin{equation}
\mathbb{P}\left(\mathcal{J}_{n}\right)\rightarrow1,\qquad{\rm where}\qquad \mathcal{J}_{n}=\left\{ \hat{Q}_{j,j} \mbox{ and } Q_{j,j}\in[b_{1}/2,2b_{2}],\ \forall1\leq j\leq p\right\} .\label{eq: main prop eq 0.5}
\end{equation}

\vskip 10pt 
\textbf{Step 1: show that $\mathbb{P}_{\mathcal{F}_{n}}(\bar{G}\leq x)$
can be approximated by $\Gamma(x,Q)$.}
\vskip 10pt

Fix an arbitrary $\delta>0$ and define $\varepsilon_{n}=\delta/\sqrt{\log p}$.
Let $S^{\Psi}=n^{-1/2}\sum_{i=1}^{n}\Psi_{i}$, $Q_{S}^{\Psi}(x)=\mathbb{P}_{\mathcal{F}_{n}}\left(\|S^{\Psi}\|_{\infty}\leq x\right)$
and $q_{n}=\sup_{x\in\mathbb{R}}|\Gamma(x,Q)-Q_{S}^{\Psi}(x)|$. By
assumption, $q_{n}=o_{P}(1)$. Let $\xi\mid\mathcal{F}_{n}\sim \mathcal{N}(0,Q)$.
Notice that 
\begin{eqnarray}
 &  & \sup_{x\in\mathbb{R}}\left|\mathbb{P}_{\mathcal{F}_{n}}\left(\|S^{\Psi}\|_{\infty}\in(x-\varepsilon_{n},x+\varepsilon_{n}]\right)-\mathbb{P}_{\mathcal{F}_{n}}\left(\|\xi\|_{\infty}\in(x-\varepsilon_{n},x+\varepsilon_{n}]\right)\right|\nonumber \\
 & = & \sup_{x\in\mathbb{R}}\left|\left[Q_{S}^{\Psi}(x+\varepsilon_{n})-Q_{S}^{\Psi}(x-\varepsilon_{n})\right]-\left[\Gamma(x+\varepsilon_{n},Q)-\Gamma(x-\varepsilon_{n},Q)\right]\right|\nonumber \\
 & \leq & \sup_{x\in\mathbb{R}}\left|Q_{S}^{\Psi}(x+\varepsilon_{n})-\Gamma(x+\varepsilon_{n},Q)\right|+\sup_{x\in\mathbb{R}}\left|Q_{S}^{\Psi}(x-\varepsilon_{n})-\Gamma(x-\varepsilon_{n},Q)\right|\overset{(i)}{=}2q_{n},\label{eq: main prop eq 0.7}
\end{eqnarray}
where $(i)$ holds by the definition of $q_{n}$. Therefore, 
\begin{align}
& \sup_{x\in\mathbb{R}}\left|\mathbb{P}_{\mathcal{F}_{n}}(\bar{G}\leq x)-\Gamma(x,Q)\right| \label{eq: main prop eq 0.8}  \\ &  =\sup_{x\in\mathbb{R}}\left|\mathbb{P}_{\mathcal{F}_{n}}\left(\|S^{\Psi}+\Delta_{n}\|_{\infty}\leq x\right)-\Gamma(x,Q)\right| \nonumber \\
 & \leq\sup_{x\in\mathbb{R}}\left|\mathbb{P}_{\mathcal{F}_{n}}\left(\|S^{\Psi}+\Delta_{n}\|_{\infty}\leq x\right)-\mathbb{P}_{\mathcal{F}_{n}}\left(\|S^{\Psi}\|_{\infty}\leq x\right)\right|\nonumber \\
 & \quad+\sup_{x\in\mathbb{R}}\left|\mathbb{P}_{\mathcal{F}_{n}}\left(\|S^{\Psi}\|_{\infty}\leq x\right)-\Gamma(x,Q)\right|\nonumber \\
 & \overset{(i)}{\leq}\mathbb{P}_{\mathcal{F}_{n}}\left(\|\Delta_{n}\|_{\infty}>\varepsilon_{n}\right)+\sup_{x\in\mathbb{R}}\mathbb{P}_{\mathcal{F}_{n}}\left(\|S^{\Psi}\|_{\infty}\in(x-\varepsilon_{n},x+\varepsilon_{n}]\right)+q_{n}\nonumber \\
 & \overset{(ii)}{\leq}\mathbb{P}_{\mathcal{F}_{n}}\left(\|\Delta_{n}\|_{\infty}>\varepsilon_{n}\right)+\sup_{x\in\mathbb{R}}\mathbb{P}_{\mathcal{F}_{n}}\left(\|\xi\|_{\infty}\in(x-\varepsilon_{n},x+\varepsilon_{n}]\right)+3q_{n}\nonumber \\
 & \overset{(iii)}{\leq}\mathbb{P}_{\mathcal{F}_{n}}\left(\|\Delta_{n}\|_{\infty}>\varepsilon_{n}\right)+C_{b}\varepsilon_{n}\sqrt{\log p}+3q_{n}\ {\rm for\ some\ constant}\ C_{b}>0, \nonumber
\end{align}
where $(i)$ follows by Lemma \ref{lem: basic prob 1} and the definition
of $q_{n}$, $(ii)$ follows by (\ref{eq: main prop eq 0.7}) and
$(iii)$ follows by Lemma \ref{lem:anti-concentration}. It follows
that, $\forall\delta>0$, 
\[
\mathbb{E}\sup_{x\in\mathbb{R}}\left|\mathbb{P}_{\mathcal{F}_{n}}(\bar{G}\leq x)-\Gamma(x,Q)\right|\leq\mathbb{P}\left(\|\Delta_{n}\|_{\infty}\sqrt{\log p}>\delta\right)+C_{b}\delta+3\mathbb{E}q_{n}\overset{(i)}{=}o(1)+C_{b}\delta,
\]
where $(i)$ holds by $\|\Delta_{n}\|_{\infty}\sqrt{\log p}=o_{P}(1)$
and $\mathbb{E}q_{n}=o(1)$: since $q_{n}=o_{P}(1)$ and $q_{n}$
is bounded (hence uniformly integrable), Theorem 5.4 on page 220 of
\cite{gut2013probability} implies $\mathbb{E}q_{n}=o(1)$. Since
$\delta>0$ is arbitrary, we have 
$$\mathbb{E}\sup_{x\in\mathbb{R}}\left|\mathbb{P}_{\mathcal{F}_{n}}(\bar{G}\leq x)-\Gamma(x,Q)\right|=o(1).$$
By the Markov's inequality, 
\begin{equation}
\sup_{x\in\mathbb{R}}\left|\mathbb{P}_{\mathcal{F}_{n}}(\bar{G}\leq x)-\Gamma(x,Q)\right|=o_{P}(1).\label{eq: main prop eq 1}
\end{equation}

\vskip 10pt 
\textbf{Step 2: show that $\Gamma(x,Q)$ can be approximated by $\Gamma(x,\hat{Q})$. }
\vskip 10pt

Let $\hat{\xi}\mid\mathcal{G}_{n}\sim \mathcal{N}(0,\hat{Q})$, where $\mathcal{G}_{n}$
is the $\sigma$-algebra generated by $\mathcal{F}_{n}$ and $\hat{Q}$.
Let $\mathcal{D}(\xi)=(\xi^{\top},-\xi^{\top})^{\top}\in\mathbb{R}^{2p}$
and $\mathcal{D}(\hat{\xi})=(\hat{\xi}^{\top},\hat{\xi}^{\top})^{\top}\in\mathbb{R}^{2p}$.
Hence, we have that 
\begin{equation}
\mathcal{D}(\xi)\mid\mathcal{F}_{n}\sim \mathcal{N}\left(0,\begin{pmatrix}Q & -Q\\
-Q & Q
\end{pmatrix}\right)\qquad{\rm and}\qquad\mathcal{D}(\hat{\xi})\mid\mathcal{G}_{n}\sim \mathcal{N}\left(0,\begin{pmatrix}\hat{Q} & -\hat{Q}\\
-\hat{Q} & \hat{Q}
\end{pmatrix}\right).\label{eq: main prop 1.5}
\end{equation}

Therefore, on the event $\mathcal{J}_{n}$, 
\begin{align*}
\sup_{x\in\mathbb{R}}\left|\Gamma(x,Q)-\Gamma(x,\hat{Q})\right| & \overset{(i)}{=}\sup_{x\in\mathbb{R}}\left|\mathbb{P}_{\mathcal{F}_{n}}\left(\max_{1\leq j\leq2p}[\mathcal{D}(\xi)]_{j}\leq x\right)-\mathbb{P}_{\mathcal{G}_{n}}\left(\max_{1\leq j\leq2p}[\mathcal{D}(\hat{\xi})]_{j}\leq x\right)\right|\\
 & \overset{(ii)}{\leq}C\|\hat{Q}-Q\|_{\infty}^{1/3}\left[1\vee\log\left(p/\|\hat{Q}-Q\|_{\infty}\right)\right]^{2/3}
\end{align*}
for some constant $C>0$ depending only on $b_{1}$ and $b_{2}$;
$(i)$ holds by the fact $\max_{1\leq j\leq2p}[\mathcal{D}(\xi)]_{j}=\|\xi\|_{\infty}$
and $\max_{1\leq j\leq2p}[\mathcal{D}(\hat{\xi})]_{j}=\|\hat{\xi}\|_{\infty}$
and $(ii)$ holds by (\ref{eq: main prop 1.5}) and Lemma 3.1 of \cite{chernozhukov2013gaussian}.
Since $\|\hat{Q}-Q\|_{\infty}=o_{P}(1/\sqrt{\log p})$, it follows
by (\ref{eq: main prop eq 0.5}) and the above display, that
\begin{equation}
\sup_{x\in\mathbb{R}}\left|\Gamma(x,Q)-\Gamma(x,\hat{Q})\right|=o_{P}(1).\label{eq: main prop eq 2}
\end{equation}

The desired result follows by (\ref{eq: main prop eq 1}) and (\ref{eq: main prop eq 2}).
\end{proof}

\subsection{Proof of   Theorem  \ref{thm: ADDS}}

\begin{proof}[\textbf{Proof of Theorem \ref{thm: ADDS}}]
The arguments are  similar to the proof of Theorem 7.2 of \cite{bickel2009simultaneous}.
Let $\Delta=\hat{b}(\sigma_{*})-b_{*}$, $J={\rm supp}(b_{*})$ and
$s=|J|$. Here, for any vector $v=(v_1,\cdots,v_p)^\top$, we  define ${\rm supp}(v)=\{j\mid |v|_j>0\}$. The proof proceeds in two steps. In the first step, we show
that $\sigma_{*}^{2}/2\leq\|H-G\hat{b}(\sigma_{*})\|_{2}^{2}/n$;
in the second step, we show the desired results.

\vskip 10pt 
\textbf{Step 1: show $\sigma_{*}^{2}/2\leq\|H-G\hat{b}(\sigma_{*})\|_{2}^{2}/n$}
\vskip 10pt 

Since $\|n^{-1}G^{\top}(H-Gb_{*})\|_{\infty}=\|n^{-1}G^{\top}\varepsilon\|_{\infty}\leq\eta\sigma_{*}$
and $b_{*}\in\mathcal{B}(\sigma_{*})$, we have $\|\hat{b}(\sigma_{*})\|_{1}\leq\|b_{*}\|_{1}$
and thus $\|[\hat{b}(\sigma_{*})]_{J}\|_{1}+\|[\hat{b}(\sigma_{*})]_{J^{c}}\|_{1}\leq\|b_{*,J}\|_{1}$.
It follows, by the triangular inequality,  that 
\begin{align}
\|\Delta_{J^{c}}\|_{1} &=\|[\hat{b}(\sigma_{*})]_{J^{c}}\|_{1}\leq\|b_{*,J}\|_{1}-\|[\hat{b}(\sigma_{*})]_{J}\|_{1}\leq\|\Delta_{J}\|_{1}\ 
\\
 \ \|\Delta\|_1&=\|\Delta_J\|_1+\|\Delta_{J^c}\|_1\leq2\|\Delta_J\|_1.\label{eq: scaled DS eq 0.1}
\end{align}

Since $\|n^{-1}G^{\top}(H-Gb_{*})\|_{\infty}\leq\eta\sigma_{*}$ and $\|n^{-1}G^{\top}(H-G\hat{b}(*))\|_{\infty}\leq\eta\sigma_{*}$,
the triangular inequality implies that $\|n^{-1}G^{\top}G\Delta\|_{\infty}\leq 2\eta\sigma_{*}$. Then $$
n^{-1}\Delta^{\top}G^{\top}G\Delta \leq\|\Delta\|_{1}\|n^{-1}G^{\top}G\Delta\|_{\infty}\leq 2\eta\sigma_{*}\|\Delta\|_{1}
\overset{(i)}{\leq}4\|\Delta_{J}\|_{1}\eta\sigma_{*}\overset{(ii)}{\leq}4\sqrt{s}\|\Delta_{J}\|_{2}\eta\sigma_{*},
$$
where $(i)$ follows by (\ref{eq: scaled DS eq 0.1}) and $(ii)$ follows by Holder's inequality.
By (\ref{eq: RE condition}) and (\ref{eq: scaled DS eq 0.1}), $n^{-1}\Delta^{\top}G^{\top}G\Delta\geq\kappa\|\Delta_{J}\|_{2}^{2}$.
Therefore, the above display implies that 
\begin{equation}
\|\Delta_{J}\|_{2}\leq4\sqrt{s}\eta\sigma_{*}/\kappa.\label{eq: scaled DS eq 0.2}
\end{equation}
and 
\begin{equation}
n^{-1/2}\|G\Delta\|_{2}=\sqrt{n^{-1}\Delta^{\top}G^{\top}G\Delta}\leq\sqrt{4\sqrt{s}\|\Delta_{J}\|_{2}\eta\sigma_{*}}\leq4\eta\sigma_{*}\sqrt{s/\kappa}.\label{eq: scaled DS eq 0.3}
\end{equation}

Then 
\begin{equation}
n^{-1/2}\|Y-G\hat{b}(\sigma_{*})\|_{2}\geq n^{-1/2}\|\varepsilon\|_{2}-n^{-1/2}\|G\Delta\|_{2}\overset{(i)}{\geq}\sqrt{3\sigma_{*}^{2}/4}-4\eta\sigma_{*}\sqrt{s/\kappa}\overset{(ii)}{\geq}\frac{\sigma_{*}}{\sqrt{2}},\label{eq: scaled DS eq 0.4}
\end{equation}
where $(i)$ follows by $n^{-1}\|\varepsilon\|_{2}^{2}\geq3\sigma_{*}^{2}/4$
and (\ref{eq: scaled DS eq 0.3}) and $(ii)$ follows by $28\eta\sqrt{s/\kappa}\leq1$. 

\vskip 10pt 
\textbf{Step 2: show the desired results}
\vskip 10pt

Recall $\tilde{b}=\hat{b}(\tilde{\sigma})$ and define $\delta=\tilde{b}-b_{*}$.
By (\ref{eq: scaled DS eq 0.4}), we have $\sigma_{*}\leq\tilde{\sigma}$
and thus $\mathcal{B}(\sigma_{*})\subseteq\mathcal{B}(\tilde{\sigma})$
(by the property of $\mathcal{B}(\sigma_{1})\subseteq\mathcal{B}(\sigma_{2})$
for $\sigma_{1}\leq\sigma_{2}$), implying that $\hat{b}(\sigma_{*})\in\mathcal{B}(\tilde{\sigma})$.
By construction, the mapping $\sigma\mapsto\|\hat{b}(\sigma)\|_{1}$ is non-increasing. It follows that $\|\widetilde{b}\|_{1}=\|\hat{b}(\tilde{\sigma})\|_1\leq\|\hat{b}(\sigma_{*})\|_{1}$. Recall from Step 1, we have that $\|\hat{b}(\sigma_{*})\|_{1}\leq\|b_{*}\|_{1}$.
 Therefore,
$$\|\widetilde{b}\|_{1}\leq\|b_{*}\|_{1},$$
 which means that $\|\widetilde{b}_{J}\|_{1}+\|\widetilde{b}_{J^{c}}\|_{1}\leq\|b_{*,J}\|_{1}$
and thus $\|\widetilde{b}_{J^{c}}\|_{1}\leq\|b_{*,J}\|_{1}-\|\widetilde{b}_{J}\|_{1}$.
It follows, by the triangular inequality, that 
\begin{equation}
\|\delta_{J^{c}}\|_{1}=\|\widetilde{b}_{J^{c}}\|_{1}\leq\|b_{*,J}\|_{1}-\|\widetilde{b}_{J}\|_{1}\leq\|\delta_{J}\|_{1}\ {\rm and}\ \|\delta\|_1=\|\delta_J\|_1+\|\delta_{J^c}\|_1\leq2\|\delta_J\|_1.\label{eq: Scaled DS eq 1}
\end{equation}

Since $\|n^{-1}G^{\top}(H-G\widetilde{b})\|_{\infty}\leq\eta\tilde{\sigma}$
and $\|n^{-1}G^{\top}\varepsilon\|_{\infty}\leq\eta\sigma_{*}$, we
have that 
\begin{align}
\|n^{-1}G^{\top}G\delta\|_{\infty}&=\|n^{-1}G^{\top}(H-G\widetilde{b}-\varepsilon)\|_{\infty}
\nonumber\\
&\leq\|n^{-1}G^{\top}(H-G\widetilde{b})\|_{\infty}+\|n^{-1}G^{\top}\varepsilon\|_{\infty}\leq\eta(\tilde{\sigma}+\sigma_{*}).\label{eq: scaled DS eq 1.5}
\end{align}

Hence, 
\begin{align}
n^{-1}\delta^{\top}G^{\top}G\delta&\overset{(i)}{\leq}\|\delta\|_{1}\|n^{-1}G^{\top}G\delta\|_{\infty} \nonumber 
\\
&\overset{(ii)}{\leq}2\|\delta_{J}\|_{1}\eta(\tilde{\sigma}+\sigma_{*}) \nonumber
\\
&\overset{(iii)}{\leq}2\sqrt{s}\|\delta_{J}\|_{2}\eta(\tilde{\sigma}+\sigma_{*}),\label{eq: scaled DS eq 2}
\end{align}
 where $(i)$ follows by Holder's inequality, $(ii)$ follows by  (\ref{eq: Scaled DS eq 1}) and (\ref{eq: scaled DS eq 1.5}) and $(iii)$ follows by Holder's
inequality. By (\ref{eq: RE condition}) and (\ref{eq: Scaled DS eq 1}),
we have $n^{-1}\|G\delta\|_{2}^{2}\geq\kappa\|\delta_{J}\|_{2}^{2}$.
This and (\ref{eq: scaled DS eq 2}) imply that 
\begin{equation}
\|\delta_{J}\|_{2}\leq2\sqrt{s}\eta(\tilde{\sigma}+\sigma_{*})/\kappa.\label{eq: scaled DS eq 3}
\end{equation}
and thus 
\begin{equation}
n^{-1}\|G\delta\|_{2}^{2}\leq2\sqrt{s}\|\delta_{J}\|_{2}\eta(\tilde{\sigma}+\sigma_{*})\leq4s\eta^{2}(\tilde{\sigma}+\sigma_{*})^{2}/\kappa.\label{eq: scaled DS eq 4}
\end{equation}

Since $\tilde{\sigma}^{2}/2\leq\|H-G\widetilde{b}\|_{2}^{2}/n$, we have
that 
\begin{align}
\frac{\tilde{\sigma}}{\sqrt{2}}\leq n^{-1/2}\|H-G\widetilde{b}\|_{2} \nonumber
&=
n^{-1/2}\|\varepsilon-G\delta\|_{2} \nonumber
\\
&\leq n^{-1/2}\|G\delta\|_{2}+n^{-1/2}\|\varepsilon\|_{2} \nonumber
\\
&\overset{(i)}{\leq}2\sqrt{s/\kappa}\eta(\tilde{\sigma}+\sigma_{*})+\sqrt{2}\sigma_{*},\label{eq: scaled DS eq 4.5}
\end{align}
where $(i)$ follows by (\ref{eq: scaled DS eq 4}) and the assumption
$n^{-1}\|\varepsilon\|_{2}^{2}\leq2\sigma_{*}^{2}$. Regrouping (\ref{eq: scaled DS eq 4.5}),
we have 
\[
\left(1/\sqrt{2}-2\eta\sqrt{s/\kappa}\right)\tilde{\sigma}\leq\left(2\eta\sqrt{s/\kappa}+\sqrt{2}\right)\sigma_{*}.
\]

Since $2\eta\sqrt{\|b_{*}\|_{0}/\kappa}\leq1/14$, it follows that
\begin{equation}
\tilde{\sigma}\leq\frac{2\eta\sqrt{s/\kappa}+\sqrt{2}}{1/\sqrt{2}-2\eta\sqrt{s/\kappa}}\sigma_{*}\leq\frac{1/14+\sqrt{2}}{1/\sqrt{2}-1/14}\sigma_{*}<3\sigma_{*}.\label{eq: scaled DS eq 5}
\end{equation}

By the definition of $\tilde{\sigma}$ (\ref{eq: auto-scaling DS step 2}),
$\sigma_{*}\leq\tilde{\sigma}$. We have proved claim (\ref{ADDS part 1}). We prove claim
(\ref{ADDS part 2}) by combining (\ref{eq: scaled DS eq 4}) and (\ref{eq: scaled DS eq 5}).
We obtain claim (\ref{ADDS part 3}) by noticing that
\begin{align*}
\|\delta\|_{1}&\overset{(i)}{\leq}2\|\delta_J\|_1
\overset{(ii)}{\leq}2\sqrt{s}\|\delta_{J}\|_{2}
\overset{(iii)}{\leq}4s\eta(\tilde{\sigma}+\sigma_{*})/\kappa\overset{(iv)}{\leq}16s\eta\sigma_{*}/\kappa,
\end{align*}
where $(i)$ follows by (\ref{eq: Scaled DS eq 1}), \((ii)\)\ follows by Holder's inequality, $(iii)$ follows by (\ref{eq: scaled DS eq 3})
and $(iv)$ follows by (\ref{eq: scaled DS eq 5}). 

To see claim (\ref{ADDS part 4}), first notice that $\tilde{\sigma}/\sqrt{2}\leq n^{-1/2}\|H-G\widetilde{b}\|_{2}$
by the constraint in the optimization problem (\ref{eq: auto-scaling DS step 2}).
By (\ref{ADDS part 1}), we have $\sigma_{*}/\sqrt{2}\leq \tilde{\sigma}/\sqrt{2}\leq n^{-1/2}\|H-G\widetilde{b}\|_{2}$.
On the other hand, 
\begin{align*}
n^{-1/2}\|H-G\widetilde{b}\|_{2}
&\overset{(i)}{\leq}\sqrt{2}\sigma_{*}+2\sqrt{s/\kappa}\eta(\tilde{\sigma}+\sigma_{*})
\\
&\overset{(ii)}{\leq}(\sqrt{2}+8/28)\sigma_{*}<2\sigma_{*},
\end{align*}
where $(i)$ follows by (\ref{eq: scaled DS eq 4.5}) and $(ii)$
follows by $\tilde{\sigma}\leq3\sigma_{*}$ (due to (\ref{ADDS part 1})) and $\sqrt{s/\kappa}\eta\leq1/28$.
Claim (\ref{ADDS part 4}) follows. 

To see claim (\ref{ADDS part 5}), we combine (\ref{ADDS part 1}) and the constraint $\|n^{-1}G^{\top}(H-G\widetilde{b})\|_{\infty}\leq\eta\tilde{\sigma}$
in (\ref{eq: DS optimization step 1}), obtaining $\|n^{-1}G^{\top}(H-G\widetilde{b})\|_{\infty}\leq3\eta\sigma_{*}$.
Since $\sigma_{*}\leq\sqrt{2}\hat{\sigma}$
(from (\ref{ADDS part 4})), we have that $\|n^{-1}G^{\top}(H-G\widetilde{b})\|_{\infty}\leq3\sqrt{2}\eta\hat{\sigma}$.
Claim (\ref{ADDS part 5}) follows. The proof is complete. 
\end{proof}

\subsection{Proof of Theorem \ref{thm: size result}}

We introduce some notations that will be used in the rest of the paper.  
Let $\hat{Q}=n^{-1}\sum_{i=1}^{n}\hat{v}_{i}\hat{v}_{i}^{\top}\hat{u}_{i}^{2}\hat{\sigma}_{u}^{-2}$
with $(\hat{v}_{1},\cdots,\hat{v}_{n})^{\top}=\hat{V}=Z-W\widetilde{\Pi}$. Let $\sigma_{v,j}^2=\mathbb{E}v_{1,j}^2$.
\begin{lem}
\label{lem: bias}Consider Algorithm \ref{alg: main}. Let Assumption
\ref{assu: regularity condition} hold. Then
\begin{itemize}
\item[(1)] $\max_{1\leq j\leq p}\|W(\widetilde{\Pi}_{j}-\Pi_{*,j})\|_{2}^{2}=O_{P}(s_{\Pi}\log p)$
and $\max_{1\leq j\leq p}\|\widetilde{\Pi}_{j}-\Pi_{*,j}\|_{1}=O_{P}(s_{\Pi}\sqrt{n^{-1}\log p})$.
\item[(2)] $\|n^{-1/2}(\Pi_{*}-\widetilde{\Pi})^{\top}W^{\top}(Y-W\widetilde{\theta})\hat{\sigma}_{u}^{-1}\|_{\infty}=O_{P}(s_{\Pi}n^{-1/2}\log p)$.
\end{itemize} \end{lem}
\begin{proof}[Proof of Lemma \ref{lem: bias}]
Define the event $\mathcal{A}=\bigcap_{j=1}^{p}\mathcal{A}_{j}$
with 
\[
\mathcal{A}_{j}=\{3\sigma_{v,j}^{2}/4\leq\|V_{j}\|_{2}^{2}n^{-1}\leq2\sigma_{v,j}^{2}\ {\rm and}\ \|n^{-1}W^{\top}V_{j}\|_{\infty}\leq\eta\sigma_{v,j}\}.
\]

Since $v_{i,j}^{2}$ and entries of $ w_{i}v_{i,j}$ have bounded sub-exponential norms (due to the sub-Gaussian property of $w_i $ and $v_{i,j} $), it follows,
by Bernstein's inequality and the union bound, that $\mathbb{P}(\mathcal{A})\rightarrow1$.
Moreover, by Theorem 6 in \cite{rudelson2013reconstruction}, the
restricted eigenvalue condition in (\ref{eq: RE condition}) (with
$G=W$) holds for some constant $\kappa>0$ with probability approaching
one. It follows, by Theorem \ref{thm: ADDS}, that 
$$\max_{1\leq j\leq p}\|W(\widetilde{\Pi}_{j}-\Pi_{*,j})\|_{2}^{2}=O_{P}(s_{\Pi}\log p)$$
and 
$$\max_{1\leq j\leq p}\|\widetilde{\Pi}_{j}-\Pi_{*,j}\|_{1}=O_{P}(s_{\Pi}\sqrt{n^{-1}\log p}).$$
This proves part (1). Observe that
\begin{align*}
\|n^{-1}W^{\top}(Y-W\widetilde{\theta})\|_{\infty}
&=\|n^{-1}W^{\top}(Y-W\hat{\theta}(\tilde{\sigma}_{u}))\|_{\infty}
\\
&\overset{(i)}{\leq}\eta\tilde{\sigma}_{u}
\\
&\overset{(ii)}{\leq}\eta\sqrt{2}n^{-1/2}\|Y-W\hat{\theta}(\tilde{\sigma}_{u})\|_{2}=\sqrt{2}\eta\hat{\sigma}_{u},
\end{align*}
where $(i)$ and $(ii)$ follow by (\ref{eq: def DS path}) and (\ref{eq: auto-scaling DS}),
respectively. Part (2) follows by Holder's inequality: 
\begin{align*}
& \left\Vert n^{-1/2}(\Pi_{*}-\widetilde{\Pi})^{\top}W^{\top}(Y-W\widetilde{\theta})\right\Vert _{\infty}\hat{\sigma}_{u}^{-1} \\
& \qquad \qquad  =\max_{1\leq j\leq p}\left|n^{-1/2}(\Pi_{*,j}-\widetilde{\Pi}_{j})^{\top}W^{\top}(Y-W\widetilde{\theta})\right|\hat{\sigma}_{u}^{-1}\\
 &\qquad \qquad  \leq\sqrt{n}\max_{1\leq j\leq p}\|\widetilde{\Pi}_{j}-\Pi_{*,j}\|_{1}\|n^{-1}W^{\top}(Y-W\widetilde{\theta})\|_{\infty}\hat{\sigma}_{u}^{-1}\\
 &\qquad \qquad  \leq\sqrt{n}\max_{1\leq j\leq p}\|\widetilde{\Pi}_{j}-\Pi_{*,j}\|_{1}\sqrt{2}\eta=O_{P}(s_{\Pi}n^{-1/2}\log p).
\end{align*}

The proof is complete.\end{proof}
\begin{lem}
\label{lem: approx variance}Consider Algorithm \ref{alg: main}.
Let Assumption \ref{assu: regularity condition} hold. Then 
\[
\|\hat{Q}-\mathbb{E}v_{1}v_{1}^{\top}\|_{\infty}=O_{P}\left(\left(s_{\Pi}n^{-1}\log p\right)\vee\sqrt{n^{-1}\log p}\right).
\]
\end{lem}
\begin{proof}[Proof of Lemma \ref{lem: approx variance}]
Notice that
\begin{eqnarray}
\|\hat{Q}-\mathbb{E}v_{1}v_{1}^{\top}\|_{\infty} & = & \left\Vert n^{-1}\sum_{i=1}^{n}\left(\hat{v}_{i}\hat{v}_{i}^{\top}-\mathbb{E}v_{1}v_{1}^{\top}\right)\right\Vert _{\infty}\nonumber \\
 & \leq & \left\Vert n^{-1}\sum_{i=1}^{n}\left(\hat{v}_{i}\hat{v}_{i}^{\top}-v_{i}v_{i}^{\top}\right)\right\Vert _{\infty}+\left\Vert n^{-1}\sum_{i=1}^{n}\left(v_{i}v_{i}^{\top}-\mathbb{E}v_{1}v_{1}^{\top}\right)\right\Vert _{\infty}.\label{eq: approx variance eq 1}
\end{eqnarray}

We bound both terms separately. Since entries in $v_{i}$ are Gaussian
with bounded variance, the entries in $v_{i}v_{i}^{\top}$ have sub-exponential
norms bounded above by some constant $K>0$. It follows, by Proposition
5.16 of \cite{vershynin2010introduction}, that $\forall t>0$, 
\begin{align*}
&\mathbb{P}\left(\left\Vert n^{-1}\sum_{i=1}^{n}\left(v_{i}v_{i}^{\top}-\mathbb{E}v_{1}v_{1}^{\top}\right)\right\Vert _{\infty}>t\sqrt{n^{-1}\log p}\right) 
\\
& \leq\sum_{j_{1}=1}^{p}\sum_{j_{2}=1}^{p}\mathbb{P}\left(\left|\sum_{i=1}^{n}v_{i,j_{1}}v_{i,j_{2}}-\mathbb{E}v_{1,j_{1}}v_{1,j_{2}}\right|>t\sqrt{n\log p}\right)\\
 & \leq2p^{2}\exp\left[-c\min\left(\frac{t^{2}n\log p}{K^{2}n},\frac{t\sqrt{n\log p}}{K}\right)\right],
\end{align*}
where $c>0$ is a universal constant. Therefore, by taking $t=4K^{2}/c$,
we have 
\begin{equation}
\left\Vert n^{-1}\sum_{i=1}^{n}\left(v_{i}v_{i}^{\top}-\mathbb{E}v_{1}v_{1}^{\top}\right)\right\Vert _{\infty}=O_{P}(\sqrt{n^{-1}\log p}).\label{eq: approx variance eq 2}
\end{equation}

Notice that $\hat{v}_{i}\hat{v}_{i}^{\top}-v_{i}v_{i}^{\top}=v_{i}w_{i}^{\top}(\Pi_{*}-\widetilde{\Pi})+(\Pi_{*}-\widetilde{\Pi})^{\top}w_{i}v_{i}^{\top}+(\Pi_{*}-\widetilde{\Pi})^{\top}w_{i}w_{i}^{\top}(\Pi_{*}-\widetilde{\Pi})$.
Therefore, 
\begin{align*}
&\left\Vert n^{-1}\sum_{i=1}^{n}\left(\hat{v}_{i}\hat{v}_{i}^{\top}-v_{i}v_{i}^{\top}\right)\right\Vert _{\infty}\\
 & \leq2\left\Vert n^{-1}\sum_{i=1}^{n}v_{i}w_{i}^{\top}(\Pi_{*}-\widetilde{\Pi})\right\Vert _{\infty}+\left\Vert n^{-1}\sum_{i=1}^{n}(\Pi_{*}-\widetilde{\Pi})^{\top}w_{i}w_{i}^{\top}(\Pi_{*}-\widetilde{\Pi})\right\Vert _{\infty}\\
 & = 2\max_{1\leq j\leq p}\max_{1\leq l\leq p}\left|n^{-1}\sum_{i=1}^{n}v_{i,l}w_{i}^{\top}(\Pi_{*,j}-\widetilde{\Pi}_{j})\right|+\max_{1\leq j\leq p}\left\Vert n^{-1}W(\widetilde{\Pi}_{j}-\Pi_{*,j})\right\Vert _{2}^{2}\\
 & \overset{(i)}{\leq}2\max_{1\leq l\leq p}\left\Vert n^{-1}\sum_{i=1}^{n}v_{i,l}w_{i}^{\top}\right\Vert _{\infty}\max_{1\leq j\leq p}\|\widetilde{\Pi}_{j}-\Pi_{*,j}\|_{1}+\max_{1\leq j\leq p}\left\Vert n^{-1}W(\widetilde{\Pi}_{j}-\Pi_{*,j})\right\Vert _{2}^{2}\\
 & \overset{(ii)}{\leq}\left\Vert n^{-1}\sum_{i=1}^{n}v_{i}w_{i}^{\top}\right\Vert _{\infty}O_{P}(s_{\Pi}\sqrt{n^{-1}\log p})+O_{P}(n^{-1}s_{\Pi}\log p),
\end{align*}
where $(i)$ follows by Holder's inequality and $(ii)$ follows by
Lemma \ref{lem: bias}. By the sub-Gaussian property of $v_{i}$ and
$w_{i}$, we can show that entries in $v_{i}w_{i}^{\top}$ have bounded
exponential norm and Proposition 5.16 of \cite{vershynin2010introduction}
and the union bound imply that $\|n^{-1}\sum_{i=1}^{n}v_{i}w_{i}^{\top}\|_{\infty}=O_{P}(\sqrt{n^{-1}\log p})$.
Therefore, 
\begin{equation}
\left\Vert n^{-1}\sum_{i=1}^{n}\left(\hat{v}_{i}\hat{v}_{i}^{\top}-v_{i}v_{i}^{\top}\right)\right\Vert _{\infty}=O_{P}(n^{-1}s_{\Pi}\log p).\label{eq: approx variance eq 3}
\end{equation}

The desired result follows by (\ref{eq: approx variance eq 1}), combined
with (\ref{eq: approx variance eq 2}) and (\ref{eq: approx variance eq 3}).
\end{proof}

\begin{proof}[\textbf{Proof of Theorem \ref{thm: size result}}]
Consider the test statistic (\ref{eq: test stat Gauss}). Assume that $H_{0}$
(\ref{eq: null hypo}) is true. We observe the following decomposition:
\begin{align}
T_{n}=\left\Vert n^{-1/2}\left(Z-W\widetilde{\Pi}\right)^{\top}\left(Y-W\widetilde{\theta}\right)\hat{\sigma}_{u}^{-1}\right\Vert _{\infty} & =\left\Vert n^{-1/2}\left(V+W(\Pi-\widetilde{\Pi})\right)^{\top}(Y-W\widetilde{\theta})\hat{\sigma}_{u}^{-1}\right\Vert _{\infty}\nonumber \\
 & =\left\Vert n^{-1/2}V^{\top}(Y-W\widetilde{\theta})\hat{\sigma}_{u}^{-1}+\Delta_{n}\right\Vert _{\infty},\label{eq: size decomposition}
\end{align}
where 
$$\Delta_{n}=n^{-1/2}(\Pi_{*}-\widetilde{\Pi})^{\top}W^{\top}(Y-W\widetilde{\theta})\hat{\sigma}_{u}^{-1}.$$
We invoke Theorem \ref{thm: approx assuming CLT} to obtain the desired
result.

Let $\mathcal{F}_{n}$ denote the $\sigma$-algebra generated by $U$
and $W$. Notice that under $H_{0}$, 
$$Y-W\widetilde{\theta}=U+W(\theta_{*}-\widetilde{\theta})$$ due to (\ref{eq:model2}).
Since $\widetilde{\theta}$ is a function of $(Y,W)$ (and thus a
function of $U$ and $W$), $\widetilde{\theta}$ is a function of
$U$ and $W$. Notice that $V$ is independent of $(U,W)$. It follows,
by the Gaussianity of $V$ and $n^{-1/2}\|Y-W\widetilde{\theta}\|_{2}=\hat{\sigma}_{u}$,
that 
\[
n^{-1/2}V^{\top}(Y-W\widetilde{\theta})\hat{\sigma}_{u}^{-1}\mid\mathcal{F}_{n}\sim \mathcal{N}(0,Q),
\]
where $Q=\mathbb{E}v_{1}v_{1}^{\top}$. In other words,
\[
\mathbb{P}\left(\|n^{-1/2}V^{\top}(Y-W\widetilde{\theta})\hat{\sigma}_{u}^{-1}\|_{\infty}\leq x\mid\mathcal{F}_{n}\right)=\Gamma(x,Q)\qquad\forall x\geq0.
\]

By Lemmas \ref{lem: bias} and \ref{lem: approx variance}, together
with $n^{-1}s_{\Pi}^{2}\log^{3}p=o(1)$, we have that 
\[
\|\Delta_{n}\|_{\infty}=o_{P}(1/\sqrt{\log p})\qquad{\rm and}\qquad\|\hat{Q}-Q\|_{\infty}=o_{P}(1/\sqrt{\log p}).
\]

Therefore, we have verified all the assumptions of Theorem \ref{thm: approx assuming CLT},
which then implies the desired result. 
\end{proof}

\subsection{Proof of Theorem \ref{thm: main result power sparse}}
\begin{lem}
\label{lem: sigma_star sparse alter}Consider Algorithm \ref{alg: main}.
Let $U(\gamma_{*})=V\gamma_{*}+U$. Suppose that Assumption \ref{assu: regularity power}
holds. Then with probability approaching one, 
\begin{itemize}
\item[] $3\sigma_{*}^{2}/4\leq n^{-1}\|U(\gamma_{*})\|_{2}^{2}\leq2\sigma_{*}^{2}$
 \item[] $\|n^{-1}W^{\top}U(\gamma_{*})\|_{\infty}\leq\eta\sigma_{*}$,
where $\sigma_{*}=\sqrt{\gamma_{*}^{\top}\Sigma_{V}\gamma_{*}+\sigma_{u}^{2}}$. 
\end{itemize}
\end{lem}
\begin{proof}[Proof of Lemma \ref{lem: sigma_star sparse alter}]
Notice that the entries of $U(\gamma_{*})/\sigma_{*}$ are i.i.d
$\mathcal{N}(0,1)$ random variables that are independent of $W$. We apply
Lemma \ref{lem: hoeffding plus union} with $L=p$, $x_{i,l}=w_{i,l}$
and $h_{i,l}=(v_{i}^{\top}\gamma_{*}+u_{i})/\sigma_{*}$ for $1\le l\leq L$.
It follows that 
\[
\mathbb{P}\left(\|n^{-1}W^{\top}U(\gamma_{*})\|_{\infty}>\eta \sigma_*\right)=\mathbb{P}\left(\|n^{-1}W^{\top}U(\gamma_{*})/\sigma_*\|_{\infty}>\eta\right)\rightarrow0.
\]

Notice that $n^{-1}\|U(\gamma_{*})\|_{2}^{2}\sigma_{*}^{-2}$ is the
average of independent $\chi^{2}(1)$ random variables. By the law
of large numbers, $n^{-1}\|U(\gamma_{*})\|_{2}^{2}\sigma_{*}^{-2}=1+o_{P}(1)$.
Therefore, 
$$3/4\leq n^{-1}\|U(\gamma_{*})\|_{2}^{2}\sigma_{*}^{-2}\leq2$$
with probability approaching one. The proof is complete. 
\end{proof}

\begin{proof}[\textbf{Proof of Theorem \ref{thm: main result power sparse}}]
By (\ref{eq:model2}) and (\ref{eq: graphical model}), we have   that $Y=W\theta_{*}+Z\gamma_{*}+U=W(\theta_{*}+\Pi_{*}\gamma_{*})+V\gamma_{*}+U$.
We apply Theorem \ref{thm: ADDS} with $(H,G,b_{*},\varepsilon)=(Y,W,\theta_{*}+\Pi_{*}\gamma_{*},V\gamma_{*}+U)$
and $\mathcal{B}(\sigma)=\mathbb{R}^{p}$ $\forall\sigma\geq0$. Notice
that, by Theorem 6 in \cite{rudelson2013reconstruction}, the restricted
eigenvalue condition in (\ref{eq: RE condition}) holds for some constant
$\kappa>0$ with probability approaching one. Also notice that 
$$\|\theta_{*}+\Pi_{*}\gamma_{*}\|_{0}\leq\|\theta_{*}\|_{0}+\|\Pi_{*}\gamma_{*}\|_{0}\leq\|\theta_{*}\|_{0}+\max_{1\leq j\leq p}\|\Pi_{*,j}\|_{0}\cdot\|\gamma_{*}\|_{0}=:s_{*}.$$
Thus, by Theorem \ref{thm: ADDS} and Lemma \ref{lem: sigma_star sparse alter},
together with $s_{*}^{2}n^{-1}\log p\rightarrow0$, it
follows that, with probability approaching one, 
\begin{equation}
\begin{cases}
\hat{\sigma}_{*}/\sqrt{2}\leq \hat{\sigma}_{u}\leq 2\sigma_{*} & (1)\\
\|\theta_{*}+\Pi_{*}\gamma_{*}-\widetilde{\theta}\|_{1}\hat{\sigma}_{u}^{-1}\leq16 s_{*}\eta\sigma_{*}\hat{\sigma}_{u}^{-1}/\kappa\leq16\sqrt{2}s_{*}\eta/\kappa & (2)\\
\|n^{-1}W^{\top}(Y-W\widetilde{\theta})\|_{\infty}\hat{\sigma}_{u}^{-1}\leq3\sqrt{2}\eta & (3)
\end{cases}\label{eq: power sparse eq 1}
\end{equation}
where $\sigma_{*}=\sqrt{\gamma_{*}^{\top}\Sigma_{V}\gamma_{*}+\sigma_{u}^{2}}$. 
Let $j_{*}\in\{1,\cdots,p\}$
satisfy that $\|\Sigma_{V}\gamma_{*}\|_{\infty}=|e_{j_{*}}^{\top}\Sigma_{V}\gamma_{*}|$,
where $e_{j_{*}}\in\mathbb{R}^{p}$ is the $j_{*}$th column of $I_{p}$. 

\vskip 10pt 
\textbf{Step 1: derive the behavior of the test statistic.}
\vskip 10pt 

By the triangular inequality, we have
\begin{align}
 & T_{n}=  \left\Vert n^{-1/2}\left(Z-W\widetilde{\Pi}\right)^{\top}\left(Y-W\widetilde{\theta}\right)\hat{\sigma}_{u}^{-1}\right\Vert _{\infty} \label{eq: power sparse eq 1.5}  \\
  \geq & \left|n^{-1/2}\left(Z_{j_{*}}-W\widetilde{\Pi}_{j_{*}}\right)^{\top}\left(Y-W\widetilde{\theta}\right)\hat{\sigma}_{u}^{-1}\right|\nonumber \\
 = & \left|n^{-1/2}\left(V_{j_{*}}+W(\Pi_{*,j_{*}}-\widetilde{\Pi}_{j_{*}})\right)^{\top}\left(W(\theta_{*}+\Pi_{*}\gamma_{*}-\widetilde{\theta})+V\gamma_{*}+U\right)\hat{\sigma}_{u}^{-1}\right|\nonumber \\
  \geq & \underset{J_{1}}{\underbrace{\left|n^{-1/2}V_{j_{*}}^{\top}V\gamma_{*}\hat{\sigma}_{u}^{-1}\right|}}-\underset{J_{2}}{\underbrace{\left|n^{-1/2}V_{j_{*}}^{\top}W(\theta_{*}+\Pi_{*}\gamma_{*}-\widetilde{\theta})\hat{\sigma}_{u}^{-1}\right|}}-\underset{J_{3}}{\underbrace{\left|n^{-1/2}V_{j_{*}}^{\top}U\hat{\sigma}_{u}^{-1}\right|}}\nonumber \\
   & -\underset{J_{4}}{\underbrace{\left|n^{-1/2}(\Pi_{*,j_{*}}-\widetilde{\Pi}_{j_{*}})^{\top}W^{\top}\left(W(\theta_{*}+\Pi_{*}\gamma_{*}-\widetilde{\theta})+V\gamma_{*}+U\right)\hat{\sigma}_{u}^{-1}\right|}}. \nonumber
\end{align}

Notice that 
\begin{align}
J_{4}&=\left|n^{-1/2}(\Pi_{*,j_{*}}-\widetilde{\Pi}_{j})^{\top}W^{\top}\left(Y-W\widetilde{\theta}\right)\hat{\sigma}_{u}^{-1}\right|
\nonumber 
\\
& \overset{(i)}{\leq}\sqrt{n}\|\widetilde{\Pi}_{j_{*}}-\Pi_{*,j_{*}}\|_{1}\|n^{-1}W^{\top}(Y-W\widetilde{\theta})\|_{\infty}\hat{\sigma}_{u}^{-1}\nonumber \\
 & \overset{(ii)}{\leq}\sqrt{n}O_{P}(s_{\Pi}\sqrt{n^{-1}\log p})3\sqrt{2}\eta\overset{(iii)}{=}o_{P}(1), \label{eq: power sparse eq 2}
\end{align}
where $(i)$ follows by Holder's inequality, $(ii)$ follows by $\max_{1\leq j\leq p}\|\widetilde{\Pi}_{j}-\Pi_{*,j}\|_{1}=O_{P}(s_{\Pi}\sqrt{n^{-1}\log p}) $ (due to Lemma \ref{lem: bias})
and (\ref{eq: power sparse eq 1})(3) and $(iii)$ holds by $s_{\Pi}=o(\sqrt{n}/\log p)$. 

Since $v_{i,j_{*}}$ has bounded sub-Gaussian norms, the law of large
numbers implies that $n^{-1}\sum_{i=1}^{n}v_{i,j_{*}}^{2}=\Sigma_{V,j_{*},j_{*}}+o_{P}(1)$
and hence 
$$\mathbb{P}\left(n^{-1}\sum_{i=1}^{n}v_{i,j_{*}}^{2}>C_{1}+1\right)\rightarrow0$$
for any constant $C_{1}\geq\max_{1\leq j\leq p}\Sigma_{V,j,j}$. By
the sub-Gaussian property of $w_{i}$, we apply Lemma \ref{lem: hoeffding plus union}
(with $x_{i,j}=w_{i,j}$ and $h_{i,j}=v_{i,j_{*}}$ for $1\leq j\leq p$).
It follows that 
$$\|n^{-1/2}V_{j_{*}}^{\top}W\|_{\infty}=\max_{1\leq j\leq p}\left\vert n^{-1/2}\sum_{i=1}^{n}w_{i,j}v_{i,j_{*}}\right\vert=O_{P}(\sqrt{\log p}).$$
By Holder's inequality and (\ref{eq: power sparse eq 1})(2), it follows
that 
\begin{equation}
J_{2}\leq\|n^{-1/2}V_{j_{*}}^{\top}W\|_{\infty}\|\theta_{*}+\Pi_{*}\gamma_{*}-\widetilde{\theta}\|_{1}\hat{\sigma}_{u}^{-1}\leq O_{P}(\sqrt{\log p})16\sqrt{2}s_{*}\eta/\kappa\overset{(i)}{=}o_{P}(1) \label{eq: power sparse eq 2.55},
\end{equation}
where $(i)$ holds by $s_{*}=o(\sqrt{n}/\log p)$. 

Notice that $\mathbb{E}|v_{i,j_{*}} u_{i}|^{2+\delta} = \mathbb{E}|v_{i,j_{*}}|^{2+\delta} \mathbb{E}|u_{i}|^{2+\delta} $ is bounded by a constant $C_2>0$, where $\delta>0$ is the constant in Assumption \ref{assu: regularity power}. The Lyapunov's central limit theorem implies that 
$$|n^{-1/2}V_{j_{*}}^{\top}U|=\left\vert n^{-1/2} \sum_{i=1}^{n }v_{i,j_{*}}u_{i}\right\vert=O_{P}(1).$$ Hence, 
\begin{equation}
\mathbb{P}\left(J_{3}>\sqrt{\log p}\right)=\mathbb{P}\left(\left|n^{-1/2}V_{j_{*}}^{\top}U\right|>\hat{\sigma}_{u}\sqrt{\log p}\right)\overset{(i)}{\leq}\mathbb{P}\left(|n^{-1/2}V_{j_{*}}^{\top}U|>\sigma_{*}\sqrt{\log (p)/2}\right)=o_{P}(1),\label{eq: power sparse eq 3}
\end{equation}
where $(i)$ holds by (\ref{eq: power sparse eq 1})(1).

Due to the (sub)-Gaussian property of $v_i$ and the definition of $\sigma_{*}$, $\mathbb{E}|v_{i,j_{*}}v_{i}^\top \gamma_{*}/\sigma_{*}|^3$ is bounded above by a constant. Again, the Lyapunov's central limit theorem implies that 
$$
\sqrt{n}\left|n^{-1}V_{j_{*}}^{\top}V\gamma_{*}-e_{j_{*}}^{\top}\Sigma_{V}\gamma_{*}\right|/\sigma_{*}=\left\vert n^{-1/2} \sum_{i=1}^{n}\left[(v_{i,j_{*}}v_{i}^\top \gamma_{*}/\sigma_{*})-\mathbb{E}(v_{i,j_{*}}v_{i}^\top \gamma_{*}/\sigma_{*})\right]\right\vert=O_{P}(1).
$$
Hence, it follows, by (\ref{eq: power sparse eq 1})(1), that
\begin{multline}
\mathbb{P}\left(\sqrt{n}\left|n^{-1}V_{j_{*}}^{\top}V\gamma_{*}-e_{j_{*}}^{\top}\Sigma_{V}\gamma_{*}\right|\hat{\sigma}_{u}^{-1}>\sqrt{\log p}\right) \\ \leq \mathbb{P}\left(\sqrt{n}\left|n^{-1}V_{j_{*}}^{\top}V\gamma_{*}-e_{j_{*}}^{\top}\Sigma_{V}\gamma_{*}\right|/\sigma_{*}>\sqrt{\log (p)/2}\right) =o_{P}(1).\label{eq: power sparse eq 4}
\end{multline}

Therefore, for any $K>0$, 
\begin{align*}
 & \mathbb{P}\left(T_{n}>(K-4)\sqrt{\log p}\right)\\
 & \overset{(i)}{\geq}\mathbb{P}\left(J_{1}>(K-1)\sqrt{\log p}\right)-\mathbb{P}\left(J_{2}>\sqrt{\log p}\right)-\mathbb{P}\left(J_{3}>\sqrt{\log p}\right)-\mathbb{P}\left(J_{4}>\sqrt{\log p}\right)\\
 & \overset{(ii)}{=}\mathbb{P}\left(J_{1}>(K-1)\sqrt{\log p}\right)-o(1)\\
 & \geq\mathbb{P}\left(\left\vert e_{j_{*}}^{\top}\Sigma_{V}\gamma_{*} \right\vert\hat{\sigma}_{u}^{-1}>K\sqrt{n^{-1}\log p}\right)-\mathbb{P}\left(\left|n^{-1}V_{j_{*}}^{\top}V\gamma_{*}-e_{j_{*}}^{\top}\Sigma_{V}\gamma_{*}\right|\hat{\sigma}_{u}^{-1}>\sqrt{n^{-1}\log p}\right)-o(1)\\
 & \overset{(iii)}{=}\mathbb{P}\left(\left\vert e_{j_{*}}^{\top}\Sigma_{V}\gamma_{*} \right\vert>K\hat{\sigma}_{u}\sqrt{n^{-1}\log p}\right)-o(1)\\
 & \overset{(iv)}{\geq}\mathbb{P}\left(\left\vert e_{j_{*}}^{\top}\Sigma_{V}\gamma_{*} \right\vert>2K\sigma_{*}\sqrt{n^{-1}\log p}\right)-o(1),
\end{align*}
where $(i)$ holds by (\ref{eq: power sparse eq 1.5}) and the sub-additivity
of probability measures, $(ii)$ holds by (\ref{eq: power sparse eq 2}), (\ref{eq: power sparse eq 2.55})
and (\ref{eq: power sparse eq 3}), $(iii)$ holds by (\ref{eq: power sparse eq 4})
and $(iv)$ holds by (\ref{eq: power sparse eq 1})(1). Recall that
$$\sigma_{*}^{2}=\gamma_{*}^{\top}\Sigma_{V}\gamma_{*}+\sigma_{u}^{2}.$$
Therefore, there exist constants $C_{3},C_{4}>0$ such that 
$$\sigma_{*}^{2}\leq C_{3}\|\gamma_{*}\|_{2}^{2}+C_{4}\leq(\sqrt{C_{3}}\|\gamma_{*}\|_{2}+\sqrt{C_{4}})^{2}.$$
Hence, the above display implies that for any $K>0$, 
\begin{multline}
\mathbb{P}\left(T_{n}>(K-4)\sqrt{\log p}\right)\geq\mathbb{P}\left(\left\vert e_{j_{*}}^{\top}\Sigma_{V}\gamma_{*} \right\vert>2K\sigma_{*}\sqrt{n^{-1}\log p}\right)-o(1)\\
\geq\mathbb{P}\left(\left\vert e_{j_{*}}^{\top}\Sigma_{V}\gamma_{*} \right\vert>2K[\sqrt{C_{3}}\|\gamma_{*}\|_{2}+\sqrt{C_{4}}]\sqrt{n^{-1}\log p}\right)-o(1).\label{eq: power sparse eq 5}
\end{multline}

\vskip 10pt 
\textbf{Step 2: derive the behavior of the critical value.}
\vskip 10pt 

Recall the elementary inequality that for $\xi\sim \mathcal{N}(0,1)$, $\mathbb{P}(|\xi|>x)\leq2\exp(-x^{2}/2)$
for all $x>0$. By the union bound, we have that 
$$1-\Gamma(x,\hat{Q})=\mathbb{P}(\|\zeta\|_{\infty}>x\mid\hat{Q})\leq2p\exp[-x^{2}/(2\|\hat{Q}\|_{\infty})],$$
where $\zeta\mid\hat{Q}\sim \mathcal{N}(0,\hat{Q})$. It follows that, $\forall\alpha\in(0,1)$,
\[
\Gamma^{-1}(1-\alpha,\hat{Q})\leq\sqrt{-2\|\hat{Q}\|_{\infty}\log(\alpha/(2p))}.
\]

Let $C_{5}>0$ be a constant such that $\|\mathbb{E}v_{1}v_{1}^{\top}\|_{\infty}\leq C_{5}$.
By Lemma \ref{lem: approx variance}, $\mathbb{P}(\|\hat{Q}\|_{\infty}\leq2C_{5})\rightarrow1$.
Hence, 
\begin{equation}
\mathbb{P}\left(\Gamma^{-1}(1-\alpha,\hat{Q})>2\sqrt{-C_{5}\log(\alpha/(2p))}\right)\rightarrow0\qquad\forall\alpha\in(0,1).\label{eq: power sparse eq 6}
\end{equation}

Let $K=4\sqrt{C_{5}\vee1}+8$. Notice that for large $p$, 
\[
\lim_{p\rightarrow\infty}\frac{(K-4)\sqrt{\log p}}{2\sqrt{-C_{5}\log(\alpha/(2p))}}\geq2>1\qquad\forall\alpha\in(0,1).
\]

Thus, it follows, by (\ref{eq: power sparse eq 5}) and (\ref{eq: power sparse eq 6}) as well as $\|\Sigma_{V}\gamma_{*}\|_{\infty}=|e_{j_{*}}^{\top}\Sigma_{V}\gamma_{*}|$,
that 
\[
\mathbb{P}\left(T_{n}>\Gamma^{-1}(1-\alpha,\hat{Q})\right)\geq\mathbb{P}\left(\|\Sigma_{V}\gamma_{*}\|_{\infty}>2K[\sqrt{C_{3}}\|\gamma_{*}\|_{2}+\sqrt{C_{4}}]\sqrt{n^{-1}\log p}\right)-o(1).
\]

Then the desired result holds with $K_{1}=2K\sqrt{C_{3}}$ and $K_{2}=2K\sqrt{C_{4}}$.
The proof is complete. 
\end{proof}

\subsection{Proof of Theorems \ref{thm: LR test} and \ref{thm: minimax}}
\begin{proof}[\textbf{Proof of Theorem \ref{thm: LR test}}]
Let 
$$\lambda_{0}=(\Sigma_{A},\Sigma_{B},\sigma_{u,A}^{2},\sigma_{u,B}^{2},\beta_{A},\beta_{A})$$
and 
$$\lambda_{1}=(\Sigma_{A},\Sigma_{B},\sigma_{u,A}^{2},\sigma_{u,B}^{2},\beta_{A},\beta_{A}+\gamma_{*}).$$
Let $w_{i}=(w_{A,i}^{\top},w_{B,i}^{\top})^{\top}\in\mathbb{R}^{2(p+1)}$
with $w_{A,i}=(y_{A,i},x_{A,i}^{\top})^{\top}\in\mathbb{R}^{p+1}$
and $w_{B,i}=(y_{B,i},x_{B,i}^{\top})^{\top}\in\mathbb{R}^{p+1}$.
The proof proceeds in two steps. First, we characterize the test statistic
and the critical value; second, we derive the behavior of the test
statistic under $\mathbb{P}_{\lambda_{1}}$. 

\vskip 10pt 
\textbf{Step 1: characterize the test statistic and the critical value.}
\vskip 10pt

Notice that the log likelihood under $\mathbb{P}_{\lambda_{0}}$ is
\[
2n\log(2\pi)^{-(p+1)/2}+\frac{n}{2}\log\det(\Omega_{A})-\frac{1}{2}\sum_{i=1}^{n}w_{A,i}^{\top}\Omega_{A}w_{A,i}+\frac{n}{2}\log\det(\Omega_{B,0})-\frac{1}{2}\sum_{i=1}^{n}w_{B,i}^{\top}\Omega_{B,0}w_{B,i},
\]
where 
\[
\Omega_{A}=\begin{bmatrix}1 & 0\\
-\beta_{A} & I_{p}
\end{bmatrix}\begin{bmatrix}\sigma_{u,A}^{-2} & 0\\
0 & \Sigma_{A}^{-1}
\end{bmatrix}\begin{bmatrix}1 & -\beta_{A}^{\top}\\
0 & I_{p}
\end{bmatrix}\quad{\rm and}\quad\Omega_{B,0}=\begin{bmatrix}1 & 0\\
-\beta_{A} & I_{p}
\end{bmatrix}\begin{bmatrix}\sigma_{u,B}^{-2} & 0\\
0 & \Sigma_{B}^{-1}
\end{bmatrix}\begin{bmatrix}1 & -\beta_{A}^{\top}\\
0 & I_{p}
\end{bmatrix}.
\]

The log likelihood under $\mathbb{P}_{\lambda_{1}}$ is 
\[
2n\log(2\pi)^{-(p+1)/2}+\frac{n}{2}\log\det(\Omega_{A})-\frac{1}{2}\sum_{i=1}^{n}w_{A,i}^{\top}\Omega_{A}w_{A,i}+\frac{n}{2}\log\det(\Omega_{B,1})-\frac{1}{2}\sum_{i=1}^{n}w_{B,i}^{\top}\Omega_{B,1}w_{B,i},
\]
where 
\[
\Omega_{B,1}=\begin{bmatrix}1 & 0\\
-\beta_{A}-\gamma_{*} & I_{p}
\end{bmatrix}\begin{bmatrix}\sigma_{u,B}^{-2} & 0\\
0 & \Sigma_{B}^{-1}
\end{bmatrix}\begin{bmatrix}1 & -\beta_{A}^{\top}-\gamma_{*}^{\top}\\
0 & I_{p}
\end{bmatrix}.
\]

Notice that $\det\Omega_{B,0}=\det\Omega_{B,1}$ and thus the likelihood
ratio test can be written with the test statistic being 
\begin{equation}
LR_{n}=\frac{1}{2}\sum_{i=1}^{n}w_{B,i}^{\top}(\Omega_{B,0}-\Omega_{B,1})w_{B,i}=\sum_{i=1}^{n}s_{i},\label{eq: LR test eq 1}
\end{equation}
where 
$$s_{i}=\sigma_{u,B}^{-2}y_{B,i}x_{B,i}^{\top}\gamma_{*}-\sigma_{u,B}^{-2}(x_{B,i}^{\top}\gamma_{*})(x_{B,i}^{\top}\beta_{A})-\sigma_{u,B}^{-2}(x_{B,i}^{\top}\gamma_{*})^{2}/2.$$
Let $c_{n}(\alpha)$ be the critical value for a test of nominal size
$\alpha$, i.e., $\mathbb{P}_{\lambda_{0}}(LR_{n}>c_{n}(\alpha))=\alpha$.

Notice that, under $\mathbb{P}_{\lambda_{0}}$, 
$$s_{i}=\sigma_{u,B}^{-2}x_{B,i}^{\top}\gamma_{*}u_{B,i}-\sigma_{u,B}^{-2}(x_{B,i}^{\top}\gamma_{*})^{2}/2.$$
By the Gaussian assumption, we have 
\[
\begin{cases}
\mathbb{E}_{\lambda_{0}}\left(s_{i}\|\gamma_{*}\|_{2}^{-1}\right)=-\sigma_{u,B}^{-2}\gamma_{*}^{\top}\Sigma_{B}\gamma_{*}\|\gamma_{*}\|_{2}^{-1}/2\\
Var_{\lambda_{0}}\left(s_{i}\|\gamma_{*}\|_{2}^{-1}\right)=\gamma_{*}^{\top}\Sigma_{B}\gamma_{*}\|\gamma_{*}\|_{2}^{-2}\sigma_{u,B}^{-2}+\sigma_{u,B}^{-4}(\gamma_{*}^{\top}\Sigma_{B}\gamma_{*})^{2}\|\gamma_{*}\|_{2}^{-2}/2\\
\mathbb{E}_{\lambda_{0}}|s_{i}\|\gamma_{*}\|_{2}^{-1}|^{3}=O(1).
\end{cases}
\]

By the Lyapunov's central limit theorem applied to $\{s_{i}\|\gamma_{*}\|_{2}^{-1}\}_{i=1}^{n}$,
we have that
\begin{equation}
\frac{n^{-1/2}\sum_{i=1}^{n}[s_{i}\|\gamma_{*}\|_{2}^{-1}-\mathbb{E}(s_{i}\|\gamma_{*}\|_{2}^{-1})]}{\sqrt{Var_{\lambda_{0}}(s_{i}\|\gamma_{*}\|_{2}^{-1})}}\overset{d}{\rightarrow}\mathcal{N}(0,1).\label{eq: LR test eq 1.5}
\end{equation}

By (\ref{eq: LR test eq 1}), we have 
\begin{align*}
1-\alpha & =\mathbb{P}_{\lambda_{0}}\left(LR_{n}\leq c_{n}(\alpha)\right)\\
 & =\mathbb{P}_{\lambda_{0}}\left(\frac{n^{-1/2}\sum_{i=1}^{n}[s_{i}\|\gamma_{*}\|_{2}^{-1}-\mathbb{E}(s_{i}\|\gamma_{*}\|_{2}^{-1})]}{\sqrt{Var_{\lambda_{0}}(s_{i}\|\gamma_{*}\|_{2}^{-1})}}\leq\frac{n^{-1/2}c_{n}(\alpha)+\sqrt{n}\gamma_{*}^{\top}\Sigma_{B}\gamma_{*}\sigma_{u,B}^{-2}/2}{\sqrt{(\gamma_{*}^{\top}\Sigma_{B}\gamma_{*})^{2}\sigma_{u,B}^{-4}/2+\gamma_{*}^{\top}\Sigma_{B}\gamma_{*}\sigma_{u,B}^{-2}}}\right)\\
 & \overset{(i)}{=}\Phi\left(\frac{n^{-1/2}c_{n}(\alpha)+\sqrt{n}\gamma_{*}^{\top}\Sigma_{B}\gamma_{*}\sigma_{u,B}^{-2}/2}{\sqrt{(\gamma_{*}^{\top}\Sigma_{B}\gamma_{*})^{2}\sigma_{u,B}^{-4}/2+\gamma_{*}^{\top}\Sigma_{B}\gamma_{*}\sigma_{u,B}^{-2}}}\right)+o(1),
\end{align*}
where $(i)$ follows by (\ref{eq: LR test eq 1.5}) and Polya's theorem
(Theorem 9.1.4 of \cite{athreya2006measure}). Therefore, 
\begin{equation}
\frac{n^{-1/2}c_{n}(\alpha)+\sqrt{n}\gamma_{*}^{\top}\Sigma_{B}\gamma_{*}\sigma_{u,B}^{-2}/2}{\sqrt{(\gamma_{*}^{\top}\Sigma_{B}\gamma_{*})^{2}\sigma_{u,B}^{-4}/2+\gamma_{*}^{\top}\Sigma_{B}\gamma_{*}\sigma_{u,B}^{-2}}}\rightarrow\Phi^{-1}(1-\alpha).\label{eq: LR test eq 2}
\end{equation}

\vskip 10pt 
\textbf{Step 2: behavior of the test statistic under $\mathbb{P}_{\lambda_{1}}$. }
\vskip 10pt

Notice that, under $\mathbb{P}_{\lambda_{1}}$, 
$$s_{i}=\sigma_{u,B}^{-2}x_{B,i}^{\top}\gamma_{*}u_{B,i}+\sigma_{u,B}^{-2}(x_{B,i}^{\top}\gamma_{*})^{2}/2.$$
Similarly as before, we have that 
\[
\begin{cases}
\mathbb{E}_{\lambda_{1}}\left(s_{i}\|\gamma_{*}\|_{2}^{-1}\right)=\gamma_{*}^{\top}\Sigma_{B}\gamma_{*}\|\gamma_{*}\|_{2}^{-1}\sigma_{u,B}^{-2}/2\\
Var_{\lambda_{1}}\left(s_{i}\|\gamma_{*}\|_{2}^{-1}\right)^{2}=\gamma_{*}^{\top}\Sigma_{B}\gamma_{*}\|\gamma_{*}\|_{2}^{-2}\sigma_{u,B}^{-2}+(\gamma_{*}^{\top}\Sigma_{B}\gamma_{*})^{2}\|\gamma_{*}\|_{2}^{-2}\sigma_{u,B}^{-4}/2\\
\mathbb{E}_{\lambda_{1}}|s_{i}\|\gamma_{*}\|_{2}^{-1}|^{3}=O(1).
\end{cases}
\]

By (\ref{eq: LR test eq 1}), we have 
\begin{align*}
&\mathbb{P}_{\lambda_{1}}\left(LR_{n}>c_{n}(\alpha)\right) 
\\& =1-\mathbb{P}_{\lambda_{1}}\left(\frac{n^{-1/2}\sum_{i=1}^{n}[s_{i}\|\gamma_{*}\|_{2}^{-1}-\mathbb{E}(s_{i}\|\gamma_{*}\|_{2}^{-1})]}{\sqrt{Var_{\lambda_{0}}(s_{i}\|\gamma_{*}\|_{2}^{-1})}}\leq\frac{n^{-1/2}c_{n}(\alpha)-\sqrt{n}\gamma_{*}^{\top}\Sigma_{B}\gamma_{*}\sigma_{u,B}^{-2}/2}{\sqrt{(\gamma_{*}^{\top}\Sigma_{B}\gamma_{*})^{2}\sigma_{u,B}^{-4}/2+\gamma_{*}^{\top}\Sigma_{B}\gamma_{*}\sigma_{u,B}^{-2}}}\right)\\
 & \overset{(i)}{=}1-\Phi\left(\frac{n^{-1/2}c_{n}(\alpha)-\sqrt{n}\gamma_{*}^{\top}\Sigma_{B}\gamma_{*}\sigma_{u,B}^{-2}/2}{\sqrt{(\gamma_{*}^{\top}\Sigma_{B}\gamma_{*})^{2}\sigma_{u,B}^{-4}/2+\gamma_{*}^{\top}\Sigma_{B}\gamma_{*}\sigma_{u,B}^{-2}}}\right)+o(1)\\
 & \overset{(ii)}{=}1-\Phi\left(\Phi^{-1}(1-\alpha)-d_{n}\right)+o(1)\quad{\rm for}\quad d_{n}=\frac{\sqrt{n}\gamma_{*}^{\top}\Sigma_{B}\gamma_{*}}{\sqrt{(\gamma_{*}^{\top}\Sigma_{B}\gamma_{*})^{2}/2+\gamma_{*}^{\top}\Sigma_{B}\gamma_{*}\sigma_{u,B}^{2}}},
\end{align*}
where $(i)$ follows Lyapunov's central limit theorem and Polya's
theorem and $(ii)$ follows by (\ref{eq: LR test eq 2}). The desired
result follows by the elementary identity of $1-\Phi(z)=\Phi(-z)$
$\forall z\in\mathbb{R}$. 
\end{proof}

\begin{proof}[\textbf{Proof of Theorem \ref{thm: minimax}}]
Let $\phi_{n}=\phi_{n}(Y_{A},X_{A},Y_{B},X_{B})$ be a test such
that $\limsup_{n\rightarrow\infty}\sup_{\lambda\in\Lambda_{0}}\mathbb{E}_{\lambda}\phi_{n}\leq\alpha$.
Define 
$$\lambda_{0}=(\Sigma_{A},\Sigma_{B},\sigma_{u,A}^{2},\sigma_{u,B}^{2},\beta_{A},\beta_{A}+\gamma_{*})$$
with $\Sigma_{A}=\Sigma_{B}=I_{p}(M_{1}+M_{2})/2$, $\sigma_{u,A}=\sigma_{u,B}=(M_{1}+M_{2})/2$
and $\gamma_{*}=0$. For $1\leq j\leq p$, let 
$$\gamma_{j}=c_{j}e_{j}\sqrt{n^{-1}\log p},$$
where $e_{j}$ is the $j$th column of $I_{p}$, $c_{j}=C/\sqrt{e_{j}^{\top}\Sigma_{B}e_{j}}$
and $C=\sigma_{u,B}/2$. Define 
$$\lambda_{j}=(\Sigma_{A},\Sigma_{B},\sigma_{u,A}^{2},\sigma_{u,B}^{2},\beta_{A},\beta_{A}+\gamma_{j}).$$
Lemma \ref{lem:1} implies that  $\Sigma_{V}=2I_{p}$. Then  $\lambda_{0}\in\Lambda_{0}$ and $\lambda_{j}\in\Lambda(\tau)$
with $\tau=(M_{1}+M_{2})/8$. Notice that 
\begin{align}
\liminf_{n\rightarrow\infty}\inf_{\lambda\in\Lambda(\tau)}\mathbb{E}_{\lambda}\phi_{n}-\alpha & \leq\liminf_{n\rightarrow\infty}\left[p^{-1}\sum_{j=1}^{p}(\mathbb{E}_{\lambda_{j}}\phi_{n}-\alpha)\right]\label{eq: minimax eq 0}\\
 & \overset{(i)}{\leq}\liminf_{n\rightarrow\infty}\left[p^{-1}\sum_{j=1}^{p}(\mathbb{E}_{\lambda_{j}}\phi_{n}-\mathbb{E}_{\lambda_{0}}\phi_{n})\right]\nonumber \\
 & \overset{(ii)}{=}\liminf_{n\rightarrow\infty}\mathbb{E}_{\lambda_{0}}\left[\phi_{n}p^{-1}\sum_{j=1}^{p}\left(\frac{d\mathbb{P}_{\lambda_{j}}}{d\mathbb{P}_{\lambda_{0}}}-1\right)\right]\nonumber \\
 & \overset{(iii)}{\leq}\liminf_{n\rightarrow\infty}\mathbb{E}_{\lambda_{0}}\left|p^{-1}\sum_{j=1}^{p}\left(\frac{d\mathbb{P}_{\lambda_{j}}}{d\mathbb{P}_{\lambda_{0}}}-1\right)\right|\nonumber\\
 &\overset{(iv)}{\leq}\liminf_{n\rightarrow\infty}\sqrt{\mathbb{E}_{\lambda_{0}}\left[p^{-1}\sum_{j=1}^{p}\left(\frac{d\mathbb{P}_{\lambda_{j}}}{d\mathbb{P}_{\lambda_{0}}}-1\right)\right]^{2}},\nonumber 
\end{align}
where $(i)$ holds by $\limsup_{n\rightarrow\infty}\mathbb{E}_{\lambda_{0}}\phi_{n}\leq\alpha$,
$(ii)$ holds by $\mathbb{E}_{\lambda_{j}}\phi_{n}=\mathbb{E}_{\lambda_{0}}\phi_{n}d\mathbb{P}_{\lambda_{j}}/d\mathbb{P}_{\lambda_{0}}$,
$(iii)$ holds by $|\phi_{n}|\leq1$ and $(iv)$ follows by Lyapunov's
inequality. 

By Step 1 of the proof of Theorem \ref{thm: LR test}, we have that
\begin{equation}
d\mathbb{P}_{\lambda_{j}}/d\mathbb{P}_{\lambda_{0}}=\exp(T_{j})\quad{\rm with}\quad T_{j}=\sum_{i=1}^{n}s_{i,j}\quad{\rm and}\quad s_{i,j}=\sigma_{u,B}^{-2}x_{B,i}^{\top}\gamma_{j}u_{B,i}-\frac{1}{2}\sigma_{u,B}^{-2}(x_{B,i}^{\top}\gamma_{j})^{2}.\label{eq: minimax eq 0.5}
\end{equation}

By the moment generating function (MGF) of Gaussian distributions,
$\mathbb{E}(\exp(s_{i,j})\mid x_{B,i})=1$ and thus 
\begin{equation}
\mathbb{E}\exp(T_{j})=\left[\mathbb{E}\exp(s_{i,j})\right]^{n}=1.\label{eq: minimax eq 0.7}
\end{equation}

Similarly, we also have $\mathbb{E}(\exp(2s_{i,j})\mid x_{B,i})=\exp(\sigma_{u,B}^{-2}(x_{B,i}^{\top}\gamma_{j})^{2})$.
Since $$(x_{B,i}^{\top}\gamma_{j})^{2}/(\gamma_{j}^{\top}\Sigma_{B}\gamma_{j})\sim\chi^{2}(1),$$
it follows, by $\sigma_{u,B}^{-2}\gamma_{j}^{\top}\Sigma_{B}\gamma_{j}<1/2$
and the MGF of chi-squared distributions, that 
\begin{align}
\mathbb{E}\exp(2T_{j})=\mathbb{E}\exp\left(\sum_{i=1}^{n}\sigma_{u,B}^{-2}(x_{B,i}^{\top}\gamma_{j})^{2}\right)&=\left\{ \mathbb{E}\exp\left[\left(\sigma_{u,B}^{-2}\gamma_{j}^{\top}\Sigma_{B}\gamma_{j}\right)\left((x_{B,1}^{\top}\gamma_{j})^{2}/(\gamma_{j}^{\top}\Sigma_{B}\gamma_{j})\right)\right]\right\} ^{n} \nonumber \\
&=(1-2\sigma_{u,B}^{-2}\gamma_{j}^{\top}\Sigma_{B}\gamma_{j})^{-n/2} \nonumber
\\
&=\left(1-\frac{1}{2}n^{-1}\log p\right)^{-n/2}.\label{eq: minimax eq 1}
\end{align}

\newpage

Next, observe
\begin{align}
& \mathbb{E}\left(p^{-1}\sum_{j=1}^{p}\exp(T_{j})\right)^{2}
\nonumber \\
 & =p^{-2}\sum_{j=1}^{p}\mathbb{E}\exp(2T_{j})+p^{-2}\sum_{j_{1}\neq j_{2}}\mathbb{E}\exp(T_{j_{1}}+T_{j_{2}}) \nonumber \\
 & =p^{-2}\sum_{j=1}^{p}\left(1-\frac{1}{2}n^{-1}\log p\right)^{-n/2}+p^{-2}\sum_{j_{1}\neq j_{2}}\mathbb{E}\left\{ \mathbb{E}[\exp(T_{j_{1}}+T_{j_{2}})\mid\{x_{B,i}\}_{i=1}^{n}]\right\} \nonumber \\
 & \overset{(i)}{=}p^{-1}\left(1-\frac{1}{2}n^{-1}\log p\right)^{-n/2}+p^{-2}\sum_{j_{1}\neq j_{2}}\left[\mathbb{E}\exp\left(\sigma_{u,B}^{-2}(x_{B,1}^{\top}\gamma_{j_{1}})(x_{B,1}^{\top}\gamma_{j_{2}})\right)\right]^{n},\label{eq: minimax eq 2} 
\end{align}
where $(i)$ follows by the Gaussian MGF and the fact that 
$$T_{j_{1}}+T_{j_{2}}\mid\{x_{B,i}\}_{i=1}^{n} \sim \mathcal{N} (\mu_t, \sigma^{2}_t)$$
is Gaussian with mean 
$\mu_t=-\sigma_{u,B}^{-2}\sum_{i=1}^{n}(x_{B,i}^{\top}\gamma_{j_{1}})^{2}/2-\sigma_{u,B}^{-2}\sum_{i=1}^{n}(x_{B,i}^{\top}\gamma_{j_{2}})^{2}/2$
and variance 
$$\sigma^{2}_t=\sigma_{u,B}^{-2}\sum_{i=1}^{n}(x_{B,i}^{\top}\gamma_{j_{1}}+x_{B,i}^{\top}\gamma_{j_{2}})^{2}.$$ 

Notice that, for $j_{1}\neq j_{2}$, $x_{B,1}^{\top}\gamma_{j_{1}}$
and $x_{B,1}^{\top}\gamma_{j_{2}}$ are independent Gaussian random
variables since $\Sigma_{B}$ is diagonal. Hence, for $j_{1}\neq j_{2}$,
\begin{eqnarray}
&&\mathbb{E}\exp\left(\sigma_{u,B}^{-2}(x_{B,1}^{\top}\gamma_{j_{1}})(x_{B,1}^{\top}\gamma_{j_{2}})\right) \\
 & = & \mathbb{E}\left\{ \mathbb{E}\left[\exp\left(\sigma_{u,B}^{-2}(x_{B,1}^{\top}\gamma_{j_{1}})(x_{B,1}^{\top}\gamma_{j_{2}})\right)\mid(x_{B,1}^{\top}\gamma_{j_{2}})\right]\right\} \nonumber \\
 & \overset{(i)}{=} & \mathbb{E}\left\{ \exp\left(\frac{1}{2}\sigma_{u,B}^{-4}(x_{B,1}^{\top}\gamma_{j_{2}})^{2}C^{2}n^{-1}\log p\right)\right\} \nonumber \\
 & \overset{(ii)}{=} & \left(1-\sigma_{u,B}^{-4}C^{4}n^{-2}\log^{2}p\right)^{-1/2}\nonumber \\
 & = & \left(1-\frac{1}{8}n^{-2}\log^{2}p\right)^{-1/2},\label{eq: minimax eq 3}
\end{eqnarray}
where $(i)$ follows by Gaussian MGF and the definition of $\gamma_{j}$s
and $(ii)$ follows by 
$$(x_{B,1}^{\top}\gamma_{j_{2}})^{2}/(\gamma_{j_{2}}^{\top}\Sigma_{B}\gamma_{j_{2}})\sim\chi^{2}(1),$$
chi-squared MGF and the definition of $\gamma_{j}$s. We combine (\ref{eq: minimax eq 2})
and (\ref{eq: minimax eq 3}), obtaining 
\begin{equation}
\mathbb{E}\left(p^{-1}\sum_{j=1}^{p}\exp(T_{j})\right)^{2}=\underset{J_{1}}{\underbrace{p^{-1}\left(1-\frac{1}{2}n^{-1}\log p\right)^{-n/2}}}+\underset{J_{2}}{\underbrace{\frac{p-1}{p}\left(1-\frac{1}{8}n^{-2}\log^{2}p\right)^{-n/2}}}.\label{eq: minimax eq 4}
\end{equation}

Since $n/\log p\rightarrow\infty$, we have that $[1+(-1/2)\cdot n^{-1}\log p]^{n/\log p}\rightarrow\exp(-1/2)$.
Therefore, 
\begin{align*}
\log J_{1}&=-\log p-\frac{\log p}{2}\log\left(1-\frac{1}{2}n^{-1}\log p\right)^{n/\log p}
\\
&=\left[-1-\frac{1}{2}\left(-\frac{1}{2}+o(1)\right)\right]\log p\rightarrow-\infty.
\end{align*}

Recall the fact that if $a_{n}\rightarrow0$, then $(1+n^{-1}a_{n})^{n}\rightarrow\exp(0)=1$.
Since $n^{-1}\log^{2}p\rightarrow0$, we have 
\[
J_{2}=\frac{p-1}{p}\cdot\frac{1}{\sqrt{\left[1+n^{-1}\cdot\left(-\frac{1}{8}n^{-1}\log^{2}p\right)\right]^{n}}}\rightarrow1\cdot\frac{1}{\sqrt{1}}=1.
\]

Thus, (\ref{eq: minimax eq 4}) and the above two displays imply that
$$\mathbb{E}\left(p^{-1}\sum_{j=1}^{p}\exp(T_{j})\right)^{2}=1+o(1).$$
By (\ref{eq: minimax eq 0.7}), we have that 
$\mathbb{E}\left(p^{-1}\sum_{j=1}^{p}\exp(T_{j})-1\right)^{2}=o(1)$.
By (\ref{eq: minimax eq 0}) and (\ref{eq: minimax eq 0.5}), it follows
that 
\[
\liminf_{n\rightarrow\infty}\inf_{\lambda\in\Lambda(\tau)}\mathbb{E}_{\lambda}\phi_{n}-\alpha\leq0.
\]

The proof is complete. 
\end{proof}

\subsection{Proof of Theorems \ref{thm: size nonGaussian} and \ref{thm: power sparse nonGauss}}
\begin{lem}
\label{lem: exp bound subGaussian}Let $X$ be a random variable.
Suppose that there exists a constant $c>0$ such that $\mathbb{P}(|X|>t)\leq\exp(1-ct^{2})$
$\forall t>0$. Then $\mathbb{E}\exp(|X|/D)<2$, where $D\geq\sqrt{7/\left[c\log(3/2)\right]}$.\end{lem}
\begin{proof}[Proof of Lemma \ref{lem: exp bound subGaussian}]
Let $Z=\exp(|X|/D)$. Since $Z\geq1$, we have the decomposition
$$Z=\sum_{i=1}^{\infty}Z\mathbf{1}\{i-1/2<Z\leq i+1/2\}.$$ Define the
sequence of constants $$b_{i}=(i-1/2)^{2}\exp[-cD^{2}\log^{2}(i-1/2)].$$ By Fubini's
theorem, 
\begin{align}
\mathbb{E}Z &=\sum_{i=1}^{\infty}\mathbb{E}Z\mathbf{1}\{i-1/2<Z\leq i+1/2\}   \leq\sum_{i=1}^{\infty}(i+1/2)\mathbb{P}(Z>i-1/2)\nonumber \\
 & =3/2+\sum_{i=2}^{\infty}(i+1/2)\mathbb{P}(Z>i-1/2)\nonumber \\
 & =3/2+\sum_{i=2}^{\infty}(i+1/2)\mathbb{P}\left[|X|>D\log(i-1/2)\right]\nonumber \\
 & \overset{(i)}{\leq}3/2+e\sum_{i=2}^{\infty}(i+1/2)\exp[-cD^{2}\log^{2}(i-1/2)],\nonumber \\
 & \overset{(ii)}{\leq}3/2+e\sum_{i=2}^{\infty}b_{i},\label{eq: subGaussian bnd eq 2}
\end{align}
where $(i)$ follows by $\mathbb{P}(|X|>t)\leq\exp(1-ct^{2})$ $\forall t>0$
and $(ii)$ follows by the elementary inequality that 
$$(i+1/2)\leq(i-1/2)^{2}$$
for $i\geq2$. Notice that, for $i\geq2$, 
\begin{align*}
\log b_{i}-\log(i^{-4})\leq\log b_{i}-\log(i-1/2)^{-4}&=  \left[6-cD^{2}\log(i-1/2)\right]\log(i-1/2)\\
 & \overset{(i)}{\leq}\left[6-cD^{2}\log(3/2)\right]\log(i-1/2)\\
 & \overset{(ii)}{\leq}-\log(i-1/2)\leq0,
\end{align*}
where $(i)$ holds by $i\ge2$ and $(ii)$ holds by the definition
of $D$ in the statement of the lemma. The above display implies that,
$\forall i\geq2$, 
$$b_{i}\leq i^{-4}.$$ It follows, by (\ref{eq: subGaussian bnd eq 2}),
that 
\[
\mathbb{E}Z\leq3/2+e\sum_{i=2}^{\infty}i^{-4}.
\]

It can be shown that $\sum_{i=2}^{\infty}i^{-4}=\pi^{4}/90-1\leq1/10$.
Thus, $\mathbb{E}Z<3/2+e/10<2$. \end{proof}
\begin{lem}
\label{lem: bias nonGaussian}Consider Algorithm \ref{alg: main nonGaussian}.
Let Assumption \ref{assu: size non-Gaussian} hold. Then
\begin{itemize}
\item[(1)] $\max_{1\leq j\leq p}\|W(\widetilde{\Pi}_{j}-\Pi_{*,j})\|_{2}^{2}=O_{P}(s_{\Pi}\log p)$
and $\max_{1\leq j\leq p}\|\widetilde{\Pi}_{j}-\Pi_{*,j}\|_{1}=O_{P}(s_{\Pi}\sqrt{n^{-1}\log p})$.
\item[(2)] $\|n^{-1/2}(\Pi_{*}-\widetilde{\Pi})^{\top}W^{\top}(Y-W\widetilde{\theta})\hat{\sigma}_{u}^{-1}\|_{\infty}=O_{P}(s_{\Pi}n^{-1/2}\log p)$. 
\end{itemize}\end{lem}
\begin{proof}[Proof of Lemma \ref{lem: bias nonGaussian}]
Part (1) follows the same argument as the proof of part(1) in Lemma
\ref{lem: bias} since changing the distribution of $Z$ and $W$
from Gaussian to sub-Gaussian does not affect the arguments. 

Part (2) also follows a similar argument. Notice that 
\begin{align*}
\|n^{-1}W^{\top}(Y-W\widetilde{\theta})\|_{\infty} &=\|n^{-1}W^{\top}(Y-W\hat{\theta}(\tilde{\sigma}_{u}))\|_{\infty}
\\
&\overset{(i)}{\leq}\eta\tilde{\sigma}_{u}
\\
&\overset{(ii)}{\leq}\eta\sqrt{2}n^{-1/2}\|Y-W\hat{\theta}(\tilde{\sigma}_{u})\|_{2}=\sqrt{2}\eta\hat{\sigma}_{u},
\end{align*}
where $(i)$ and $(ii)$ follow by the constraints (\ref{eq: non-gaussian DS-2})
and (\ref{eq: auto-scaling DS-2}), respectively. Therefore,
\begin{align*}
& \left\Vert n^{-1/2}(\Pi_{*}-\widetilde{\Pi})^{\top}W^{\top}(Y-W\widetilde{\theta})\right\Vert _{\infty}\hat{\sigma}_{u}^{-1} \\
& =\max_{1\leq j\leq p}\left|n^{-1/2}(\Pi_{*,j}-\widetilde{\Pi}_{j})^{\top}W^{\top}(Y-W\widetilde{\theta})\right|\hat{\sigma}_{u}^{-1}\\
 & \overset{(i)}{\leq}\sqrt{n}\max_{1\leq j\leq p}\|\widetilde{\Pi}_{j}-\Pi_{*,j}\|_{1}\|n^{-1}W^{\top}(Y-W\widetilde{\theta})\|_{\infty}\hat{\sigma}_{u}^{-1}\\
 & \leq\sqrt{n}\max_{1\leq j\leq p}\|\widetilde{\Pi}_{j}-\Pi_{*,j}\|_{1}\sqrt{2}\eta=O_{P}(s_{\Pi}n^{-1/2}\log p),
\end{align*}
where $(i)$ follows by Holder's inequality. This proves part (2).
The proof is complete.
\end{proof}

\begin{proof}[\textbf{Proof of Theorem \ref{thm: size nonGaussian}}]
The argument is similar to the proof of Theorem \ref{thm: size result},
except that we need to invoke a high-dimensional central limit theorem
under non-Gaussian designs. We proceed in two steps. First, we show
the desired result assuming a ``central limit theorem'' (stated
below in (\ref{eq: size thm nonGauss eq 4})) and then we show the
``central limit theorem''. 

\vskip 10pt 
\textbf{Step 1: show the desired result assuming a ``central limit
theorem'' }
\vskip 10pt 

Consider the test statistic $T_{n}^+$ (\ref{eq: test stat}).
Assume that $H_{0}$ (\ref{eq: null hypo}) is true. Then 
\begin{align}
T_{n}^{+}&=\left\Vert n^{-1/2}\left(Z-W\widetilde{\Pi}\right)^{\top}\left(Y-W\widetilde{\theta}\right)\hat{\sigma}_{u}^{-1}\right\Vert _{\infty}
\\
 & =\left\Vert n^{-1/2}\left(V+W(\Pi-\widetilde{\Pi})\right)^{\top}(Y-W\widetilde{\theta})\hat{\sigma}_{u}^{-1}\right\Vert _{\infty}\nonumber \\
 & =\left\Vert n^{-1/2}V^{\top}(Y-W\widetilde{\theta})\hat{\sigma}_{u}^{-1}+\Delta_{n}\right\Vert _{\infty},\nonumber \\
 & =\left\Vert n^{-1/2}\sum_{i=1}^{n}\Psi_{i}+\Delta_{n}\right\Vert _{\infty},\label{eq: size thm nonGauss eq 1}
\end{align}
where 
$$\Delta_{n}=n^{-1/2}(\Pi_{*}-\widetilde{\Pi})^{\top}W^{\top}(Y-W\widetilde{\theta})\hat{\sigma}_{u}^{-1}$$
and $\Psi_{i}=v_{i}\hat{u}_{i}\hat{\sigma}_{u}^{-1}$ with $\hat{u}_{i}=y_{i}-w_{i}^{\top}\widetilde{\theta}$.
Let $\mathcal{F}_{n}$ be the $\sigma$-algebra generated by $W$
and $Y$. 

Since $\hat{u}_{i}$ and $\hat{\sigma}_{u}$ are computed using only
$Y$ and $W$, which, under $H_{0}$ (\ref{eq: null hypo}), are independent
of $V$, it follows that 
$$\mathbb{E}(v_{i}v_{i}^{\top}\hat{u}_{i}^{2}\hat{\sigma}_{u}^{-2}\mid\mathcal{F}_{n})=\Sigma_{V}\hat{u}_{i}^{2}\hat{\sigma}_{u}^{-2}.$$
By $\hat{\sigma}_{u}^{2}=n^{-1}\|Y-W\widetilde{\theta}\|_{2}^{2}=n^{-1}\sum_{i=1}^{n}\hat{u}_{i}^{2}$,
we have $Q:=n^{-1}\sum_{i=1}^{n}\mathbb{E}(\Psi_{i}\Psi_{i}^{\top}\mid\mathcal{F}_{n})=\Sigma_{V}$.
By Assumption \ref{assu: size non-Gaussian} and Lemma \ref{lem:1}, there exist constant
constants $b_{1},b_{2}>0$ such that 
\begin{equation}
b_{1}\leq\min_{1\leq j\leq p}Q_{j,j}\leq\max_{1\leq j\leq p}Q_{j,j}\leq b_{2}.\label{eq: thm size nonGauss eq 2}
\end{equation}

By Lemmas \ref{lem: bias nonGaussian} and \ref{lem: approx variance},
we have
\begin{equation}
\begin{cases}
\|\Delta_{n}\|_{\infty}=O_{P}(s_{\Pi}\log p)=o_{P}(1/\sqrt{\log p})\\
\|\hat{Q}-Q\|_{\infty}=O_{P}\left((s_{\Pi}n^{-1}\log p)\vee\sqrt{n^{-1}\log p}\right)=o_{P}(1/\sqrt{\log p}).
\end{cases}\label{eq: thm size nonGauss eq 3}
\end{equation}

We prove the result assuming the following claim, which is proved
afterwards:
\begin{equation}
\sup_{x\in\mathbb{R}}\left\Vert \mathbb{P}\left(\left\Vert n^{-1/2}\sum_{i=1}^{n}\Psi_{i}\right\Vert _{\infty}\leq x\mid\mathcal{F}_{n}\right)-\Gamma(x,\Sigma_{V})\right\Vert =o_{P}(1).\label{eq: size thm nonGauss eq 4}
\end{equation}

We apply Theorem \ref{thm: approx assuming CLT} to the decomposition
(\ref{eq: size thm nonGauss eq 1}). From (\ref{eq: thm size nonGauss eq 3}),
(\ref{eq: thm size nonGauss eq 2}) and (\ref{eq: size thm nonGauss eq 4}),
all the assumptions of Theorem \ref{thm: approx assuming CLT} are
satisfied. Therefore, the desired result follows by Theorem \ref{thm: approx assuming CLT}.

\vskip 10pt 
\textbf{Step 2: show the ``central limit theorem''}
\vskip 10pt

It remains to prove the claim in (\ref{eq: size thm nonGauss eq 4}).
To this end, we invoke Proposition \ref{prop: HD CLT CCK}. Hence,
we only need to verify the following conditions. 
\begin{itemize}
\item[(a)] There exists a constant $b>0$ such that $\mathbb{P}\left(\min_{1\leq j\leq p}\Sigma_{V,j,j}\geq b\right)\rightarrow1$.
\item[(b)] There exists a sequence of $\mathcal{F}_{n}$-measurable random variables
$B_{n}>0$ such that $B_{n}=o(\sqrt{n}/\log^{7/2}(pn))$, $\max_{1\leq j\leq p}n^{-1}\sum_{i=1}^{n}|\hat{u}_{i}|^{3}\hat{\sigma}_{u}^{-3}\mathbb{E}|v_{i,j}|^{3}\leq B_{n}$
and $\max_{1\leq j\leq p}n^{-1}\sum_{i=1}^{n}|\hat{u}_{i}|^{4}\hat{\sigma}_{u}^{-4}\mathbb{E}|v_{i,j}|^{4}\leq B_{n}^{2}$. 
\item[(c)] $\max_{1\leq i\leq n,\ 1\leq j\leq p}\mathbb{E}[\exp(|v_{i,j}\hat{u}_{i}\hat{\sigma}_{u}^{-1}|/B_{n})\mid\mathcal{F}_{n}]\leq2$. 
\end{itemize}
Notice that Condition (a) follows by (\ref{eq: thm size nonGauss eq 2})
and $\Sigma_{V}=Q$. To show the other two conditions, notice that,
by the constraints (\ref{eq: non-gaussian DS-2}) and (\ref{eq: auto-scaling DS-2}),
$\max_{1\leq i\leq n}|\hat{u}_{i}|\leq\mu\tilde{\sigma}_{u}$ and
$\tilde{\sigma}_{u}^{2}/2\leq\hat{\sigma}_{u}^{2}$. Therefore, 
\begin{equation}
\max_{1\leq i\leq n}|\hat{u}_{i}|\hat{\sigma}_{u}^{-1}\leq\sqrt{2}\mu.\label{eq: size thm nonGauss eq 5}
\end{equation}

Since $v_{i,j}$ has a bounded sub-Gaussian norm, there exists a constant
$C_{1}>0$ such that $\forall1\leq j\leq p$ and $\forall1\leq i\leq n$,
$\mathbb{E}|v_{i,j}|^{3}\leq C_{1}$ and $\mathbb{E}v_{i,j}^{4}\leq C_{1}^{2}$.
By the sub-Gaussian property and Lemma \ref{lem: exp bound subGaussian},
there exists a constant $C_{2}>0$ such that $\max_{1\leq i\leq p,\ 1\leq i\leq n}\mathbb{E}\exp(|v_{i,j}|/C_{2})\leq2$.
We define 
\begin{equation}
B_{n}=4(\mu^{3}\vee1)(C_{1}\vee C_{2}).\label{eq: size thm nonGauss eq 6}
\end{equation}

By (\ref{eq: size thm nonGauss eq 5}), we have that 
\[
\begin{cases}
\max_{1\leq j\leq p} &n^{-1}\sum_{i=1}^{n}|\hat{u}_{i}|^{3}\hat{\sigma}_{u}^{-3}\mathbb{E}|v_{i,j}|^{3}
\\
& \leq C_{1}\max_{1\leq j\leq p}n^{-1}\sum_{i=1}^{n}|\hat{u}_{i}|^{3}\hat{\sigma}_{u}^{-3}\leq2\sqrt{2}\mu^{3}C_{1}\leq B_{n}\\
\max_{1\leq j\leq p} & n^{-1}\sum_{i=1}^{n}|\hat{u}_{i}|^{4}\hat{\sigma}_{u}^{-4}\mathbb{E}|v_{i,j}|^{4}
\\
&\leq C_{1}^{2}\max_{1\leq j\leq p}n^{-1}\sum_{i=1}^{n}|\hat{u}_{i}|^{4}\hat{\sigma}_{u}^{-4}\leq4\mu^{4}C_{1}^{2}\leq B_{n}^{2}\\
&\mathbb{E}[\exp(|v_{i,j}\hat{u}_{i}\hat{\sigma}_{u}^{-1}|/B_{n})\mid\mathcal{F}_{n}]\\
&
\leq\mathbb{E}[\exp(|v_{i,j}|\sqrt{2}\mu/B_{n})\mid\mathcal{F}_{n}]\leq\mathbb{E}[\exp(|v_{i,j}|/C_{2})\mid\mathcal{F}_{n}]\leq2.
\end{cases}
\]

By the rate conditions in Assumption \ref{assu: size non-Gaussian},
it is not hard to see that $B_{n}$ in (\ref{eq: size thm nonGauss eq 6})
satisfies $B_{n}=o(\sqrt{n}/\log^{7/2}(pn))$. We have showed Conditions
(b) and (c). The proof is complete. \end{proof}
\begin{lem}
\label{lem: sigma_star sparse alter nonGauss}Consider Algorithm \ref{alg: main nonGaussian}.
Let $U(\gamma_{*})=V\gamma_{*}+U$. Suppose that Assumption \ref{assu: size non-Gaussian}
holds. Then with probability approaching one, 
\begin{itemize}
\item[] $3\sigma_{*}^{2}/4\leq n^{-1}\|U(\gamma_{*})\|_{2}^{2}\leq2\sigma_{*}^{2}$,
\item[] $\|n^{-1}W^{\top}U(\gamma_{*})\|_{\infty}\leq\eta\sigma_{*}$ and
\item[] $\|V\gamma_{*}+U\|_{\infty}\leq\mu\sigma_{*}$, where $\sigma_{*}=\sqrt{\gamma_{*}^{\top}\Sigma_{V}\gamma_{*}+\sigma_{u}^{2}}$. 
\end{itemize}
\end{lem}
\begin{proof}[Proof of Lemma \ref{lem: sigma_star sparse alter nonGauss}]
Let $u_{i}(\gamma_{*})=v_{i}^{\top}\gamma_{*}+u_{i}$. Then $\mathbb{E}u_{i}^{2}(\gamma_{*})\sigma_{*}^{-2}=1$.
By the law of large numbers, $n^{-1}\|U(\gamma_{*})\|_{2}^{2}\sigma_{*}^{-2}=n^{-1}\sum_{i=1}^{n}u_{i}^{2}(\gamma_{*})\sigma_{*}^{-2}=1+o_{P}(1)$.
Hence, 
\[
\mathbb{P}\left(3/4\leq n^{-1}\|U(\gamma_{*})\|_{2}^{2}\sigma_{*}^{-2}\leq2\right)\rightarrow1.
\]

Since $\mathbb{P}(n^{-1}\sum_{i=1}^{n}u_{i}^{2}(\gamma_{*})\sigma_{*}^{-2}\geq2)\rightarrow0$
and entries of $w_{i}$ have sub-Gaussian norms bounded above by some
constant $C_{1}>0$, we can apply Lemma \ref{lem: hoeffding plus union}
(with $x_{i,j}=w_{i,j}$ and $h_{i,j}=u_{i}(\gamma_{*})\sigma_{*}^{-1}$
for $1\leq j\leq p$). It follows that 
\[
\mathbb{P}\left(\|n^{-1}W^{\top}U(\gamma_{*})\sigma_{*}^{-1}\|_{\infty}>C_{2}\sqrt{n^{-1}\log p}\right)\rightarrow0,
\]
where $C_{2}>0$ is a constant depending only on $C_{1}$. To see
$\mathbb{P}(\|V\gamma_{*}+U\|_{\infty}\leq\mu\sigma_{*})\rightarrow1$,
first notice that, by Minkowski's inequality and the bounded ninth
moment of $u_{1}$, there exists a constant $C_{3}>0$ such that 
$$[\mathbb{E}|u_{1}(\gamma_{*})|^{9}]^{1/9}\leq[\mathbb{E}|v_{1}^{\top}\gamma_{*}|^{9}]^{1/9}+[\mathbb{E}|u_{1}|^{9}]^{1/9}\le C_{3}\left( \|\gamma_*\|_2 \vee 1 \right).$$
Therefore, there exists a constant $C_4>0$ such that 
\begin{align}
\mathbb{P}\left(|u_{1}(\gamma_{*})|>\mu\sigma_{*}\right) & =\mathbb{P}\left(|u_{1}(\gamma_{*})|^{9}>\mu^{9}\sigma_{*}^{9}\right) \nonumber
\\
&\overset{(i)}{\leq}\frac{\mathbb{E}|u_{1}(\gamma_{*})|^{9}}{\mu^{9}\sigma_{*}^{9}}\nonumber
\\
&\leq\frac{C_3^9\left( \|\gamma_*\|_{2}^9 \vee 1 \right)}{\sigma_{*}^9n\log^3 p}  \nonumber
\\
&\overset{(ii)}{\leq} C_4 n^{-1} \log^{-3} p ,\label{eq: sigma star nonGauss sparse alter eq 1}
\end{align}
where $(i)$ follows by Markov's inequality and $(ii)$ holds by the definition of $\sigma_*$  and the fact that eigenvalues of $\Sigma_V$ is bounded away from zero (due to Lemma \ref{lem:1}). Hence, 
\begin{align*}
\mathbb{P}\left(\|V\gamma_{*}+U\|_{\infty}\leq\mu\sigma_{*}\right) &=\mathbb{P}\left(\max_{1\leq i\leq n}|u_{i}(\gamma_{*})|\leq\mu\sigma_{*}\right)\\&=\prod_{i=1}^{n}\mathbb{P}\left(|u_{i}(\gamma_{*})|\leq\mu\sigma_{*}\right)\\
&=\left[\mathbb{P}\left(|u_{1}(\gamma_{*})|\leq\mu\sigma_{*}\right)\right]^{n}=\left[1-\mathbb{P}\left(|u_{1}(\gamma_{*})|>\mu\sigma_{*}\right)\right]^{n}
\\
&\overset{(i)}{\geq}\left[1-C_{4}n^{-1}\log^{-3}p\right]^{n}\overset{(ii)}{\geq}1+o(1),
\end{align*}
where $(i)$ holds by (\ref{eq: sigma star nonGauss sparse alter eq 1})
and $(ii)$ holds by $(1+n^{-1}a_{n})^{n}\rightarrow\exp(0)=1$ for
$a_{n}=o(1)$ (here $a_{n}=-C_{4}\log^{-3}p$).
The proof is complete. 
\end{proof}

\begin{proof}[\textbf{Proof of Theorem \ref{thm: power sparse nonGauss}}]
The proof is almost identical to that of Theorem \ref{thm: main result power sparse},
except the reasoning for (\ref{eq: power sparse eq 1}). We now apply
Theorem \ref{thm: ADDS} with $\mathcal{B}(\sigma)=\{b\in\mathbb{R}^{p}\mid\|Y-Wb\|_{\infty}\leq\mu\sigma\}$
instead of $\mathcal{B}(\sigma)=\mathbb{R}^{p}$. We follow the argument
at the beginning of the proof of Theorem \ref{thm: main result power sparse}
with Lemma \ref{lem: sigma_star sparse alter} replaced by Lemma \ref{lem: sigma_star sparse alter nonGauss}.
The rest of the proof is the same. 
\end{proof}

\subsection{Technical tools}
\begin{lem}
\label{lem: basic prob 1}Let $X$ and $Y$ be two random vectors.
Then $\forall t,\varepsilon>0$, $$\left|\mathbb{P}\left(\|X\|_{\infty}\leq t\right)-\mathbb{P}\left(\|Y\|_{\infty}\leq t\right)\right|\leq\mathbb{P}\left(\|X-Y\|_{\infty}>\varepsilon\right)+\mathbb{P}\left(\|Y\|_{\infty}\in(t-\varepsilon,t+\varepsilon]\right).$$\end{lem}
\begin{proof}[Proof of Lemma \ref{lem: basic prob 1}]
By the triangular inequality, we have 
\begin{align*}
\mathbb{P}(\|X\|_{\infty}>t)&\leq\mathbb{P}(\|X-Y\|_{\infty}>\varepsilon)+\mathbb{P}(\|Y\|_{\infty}>t-\varepsilon)
\\
&=\mathbb{P}(\|X-Y\|_{\infty}>\varepsilon)+\mathbb{P}(\|Y\|_{\infty}>t)+\mathbb{P}(\|Y\|_{\infty}\in(t-\varepsilon,t])
\end{align*}
 and 
\begin{align*}
\mathbb{P}(\|X\|_{\infty}>t)&\geq\mathbb{P}(\|Y\|_{\infty}>t+\varepsilon)-\mathbb{P}(\|X-Y\|_{\infty}>\varepsilon)
\\
&=\mathbb{P}(\|Y\|_{\infty}>t)-\mathbb{P}(\|Y\|_{\infty}\in(t,t+\varepsilon])-\mathbb{P}(\|X-Y\|_{\infty}>\varepsilon).
\end{align*}

The above two display imply that 
\[
\left|\mathbb{P}\left(\|X\|_{\infty}>t\right)-\mathbb{P}\left(\|Y\|_{\infty}>t\right)\right|\leq\mathbb{P}\left(\|X-Y\|_{\infty}>\varepsilon\right)+\mathbb{P}\left(\|Y\|_{\infty}\in(t-\varepsilon,t+\varepsilon]\right).
\]

The result follows by noticing that $$\left|\mathbb{P}\left(\|X\|_{\infty}>t\right)-\mathbb{P}\left(\|Y\|_{\infty}>t\right)\right|=\left|\mathbb{P}\left(\|X\|_{\infty}\leq t\right)-\mathbb{P}\left(\|Y\|_{\infty}\leq t\right)\right|.$$ \end{proof}
\begin{lem}
\label{lem:anti-concentration}Let $Y=(Y_{1},\cdots,Y_{p})^{\top}$
be a random vector and $\mathcal{F}$ a $\sigma$-algebra. If $\mathbb{E}(Y\mid\mathcal{F})=0$,
$Y\mid\mathcal{F}$ is Gaussian and $\min_{1\leq j \leq p}\mathbb{E}(Y_{j}^{2}\mid\mathcal{F})\geq b$
almost surely for some constant $b>0$, then there exists a constant $C>0$
depending only on $b$ such that $\forall\varepsilon>0$. 
\[
\sup_{x\in\mathbb{R}}\mathbb{P}\left(\|Y\|_{\infty}\in(x-\varepsilon,x+\varepsilon]\mid\mathcal{F}\right)\leq C\varepsilon\sqrt{\log p}\quad a.s.
\]
\end{lem}
\begin{proof}[Proof of Lemma \ref{lem:anti-concentration}]
By Nazarov's anti-concentration inequality (Lemma A.1 in \cite{chernozukov2014central}),
there exists a constant $C_{b}>0$ depending only on $b$ such that
almost surely,
 $$ \begin{cases}
 \sup_{x\in\mathbb{R}}\mathbb{P}(\max_{1\leq j\leq p}Y_{j}\in(x-\varepsilon,x+\varepsilon]\mid\mathcal{F})\leq2C_{b}\varepsilon\sqrt{\log p}\\
\sup_{x\in\mathbb{R}}\mathbb{P}(\max_{1\leq j\leq p}(-Y_{j})\in(x-\varepsilon,x+\varepsilon]\mid\mathcal{F})\leq2C_{b}\varepsilon\sqrt{\log p}.
\end{cases}$$

Since $\|Y\|_{\infty}=\max\{\max_{1\leq j\leq p}Y_{j},\max_{1\leq j\leq p}(-Y_{j})\}$,
the desired result follows by 
\begin{multline*}
\sup_{x\in\mathbb{R}}\mathbb{P}(\|Y\|_{\infty}\in(x-\varepsilon,x+\varepsilon]\mid\mathcal{F})\\
\leq\sup_{x\in\mathbb{R}}\mathbb{P}(\max_{1\leq j\leq p}Y_{j}\in(x-\varepsilon,x+\varepsilon]\mid\mathcal{F}) 
\\
+\sup_{x\in\mathbb{R}}\mathbb{P}(\max_{1\leq j\leq p}(-Y_{j})\in(x-\varepsilon,x+\varepsilon]\mid\mathcal{F})\leq4C_{b}\varepsilon\sqrt{\log p}.
\end{multline*}
 \end{proof}
\begin{lem}
\label{lem: hoeffding plus union}Let $\{x_{i}\}_{i=1}^{n}$ and $\{h_{i}\}_{i=1}^{n}$
be two sequences of random vectors in $\mathbb{R}^{L}$ that are independent
across $i$. Suppose that $\{x_{i}\}_{i=1}^{n}$ and $\{h_{i}\}_{i=1}^{n}$
are also independent and that there exist constants $K_{1},K_{2}>0$
such that $\forall1\leq l\leq L$, the sub-Gaussian norm of $x_{i,l}$
is bounded above by $K_{1}$ and $\mathbb{P}(\max_{1\leq l\leq L}n^{-1}\sum_{i=1}^{n}h_{i,l}^{2}>K_{2})\rightarrow0$.

Then, if $L\rightarrow\infty$, then there exists a constant $K>0$ depending
only on $K_{1},K_{2}$ such that 
\[
\mathbb{P}\left(\max_{1\leq l\leq L}\left|n^{-1/2}\sum_{i=1}^{n}x_{i,l}h_{i,l}\right|>K\sqrt{\log L}\right)\rightarrow0.
\]
\end{lem}
\begin{proof}[Proof of Lemma \ref{lem: hoeffding plus union}]
Define the event $\mathcal{J}_{n}=\{\max_{1\leq l\leq L}n^{-1}\sum_{i=1}^{n}h_{i,l}^{2}\leq K_{2}\}$.
By the bounded sub-Gaussian norm of $x_{i,l}$, we have that, on the
event $\mathcal{J}_{n}$
\begin{align*}
&\mathbb{P}\left(\max_{1\leq l\leq L}\left|n^{-1/2}\sum_{i=1}^{n}x_{i,l}h_{i,l}\right|>t\mid\{h_{i}\}_{i=1}^{n}\right)
\\
&\leq\sum_{l=1}^{L}\mathbb{P}\left(\left|n^{-1/2}\sum_{i=1}^{n}x_{i,l}h_{i,l}\right|>t\mid\{h_{i}\}_{i=1}^{n}\right)\\
&\overset{(i)}{\leq}p\exp\left[1-\frac{ct^{2}}{K_{1}^{2}\max_{1\leq l\leq L}n^{-1}\sum_{i=1}^{n}h_{i,l}^{2}}\right]\overset{(ii)}{\leq}p\exp\left[1-\frac{ct^{2}}{K_{1}^{2}K_{2}}\right],
\end{align*}
where $c>0$ is a universal constant; $(i)$ holds by Proposition
5.10 of \cite{vershynin2010introduction} (applied to the conditional
probability measure). Letting $t=2K_{1}\sqrt{K_{2}c^{-1}\log L}$,
we have that 
\begin{align*}
& \mathbb{P}\left(\max_{1\leq l\leq L}\left|n^{-1/2}\sum_{i=1}^{n}x_{i,l}h_{i,l}\right|>2K_{1}\sqrt{K_{2}c^{-1}\log L}\right)
\\
&\qquad \qquad \leq\mathbb{P}\left(\mathcal{J}_{n}^{c}\right)+p\exp\left[1-\frac{c\left(2K_{1}\sqrt{K_{2}c^{-1}\log L}\right)^{2}}{K_{1}^{2}K_{2}}\right]\\
& \qquad \qquad =\mathbb{P}\left(\mathcal{J}_{n}^{c}\right)+e/L^{3}=o(1).
\end{align*}

Hence, the desired result holds with $K=2K_{1}\sqrt{K_{2}c^{-1}\log L}$. 
\end{proof}

     \bibliography{biblio_state_space}{}
  \bibliographystyle{imsart-nameyear}

\end{document}